\documentclass[10pt, reqno]{amsart}
\usepackage{amsmath, amstext, amsbsy, amssymb, amscd}
\usepackage[mathscr]{eucal}

\usepackage{color}

\newtheorem{theorem}{Theorem}[section]
\newtheorem{lemma}[theorem]{Lemma}
\newtheorem{proposition}[theorem]{Proposition}
\newtheorem{corollary}[theorem]{Corollary}
\theoremstyle{definition}
\newtheorem{definition}[theorem]{Definition}

\theoremstyle{remark}
\newtheorem{remark}[theorem]{Remark}
\newtheorem{conjecture}[theorem]{Conjecture}

\newtheorem{convention}[theorem]{Convention}

\numberwithin{equation}{section}
\newcommand{\ad}{\text{ad}}

\newcommand{\ch}{\mbox{ch} }
\newcommand{\C}{ \mathbb C }

\newcommand{\End}{{\rm End}}

\newcommand{\fock}{{\mathbb H}_X}
\newcommand{\Fock}{{\mathcal F}_X}

\newcommand{\Hn}{H^*(\Xn)}

\newcommand{\im}{ {\rm Im} }

\newcommand{\KX}{K_X}
\newcommand{\la}{\lambda}
\newcommand{\lambsq}{s(\lambda)}

\newcommand{\orbsym}{H^*_{\text{orb}}(X^{(n)})}

\newcommand{\Pee}{{\mathbb P}}

\newcommand{\Supp}{\text{\rm Supp}}

\newcommand{\vac}{|0\rangle}
\newcommand{\w}{\tilde}
\newcommand{\W}{\widetilde}

\newcommand{\Xn}{ X^{[n]}}

\newcommand{\Z}{ \mathbb Z }

\def\Yad{Y^{[\![\alpha,\delta]\!]}}
\def\Ybe{Y^{[\![\beta,\eta]\!]}}

\def\cU{\mathcal U}
\def\cV{{\mathcal V}}
\let\sub=\subset
\def\mapright#1{\,\smash{\mathop{\lra}\limits^{#1}}\,}
\def\pr{\mathrm{pr}}
\def\and{\quad\mathrm{and}\quad}

\def\sta{^*}
\def\adbracket{{[\![\alpha,\delta]\!]}}

\def\ev{\mathrm{ev}}
\def\cP{\mathcal P}
\def\ti{\tilde}
\def\virt{^{\mathrm{vir}}}
\def\lsta{_{\ast}}
\def\eps{\epsilon}

\def\sF{\mathscr F}
\def\sO{{\mathscr O}}
\def\upmo{^{-1}}
\def\QQ{{\mathbb Q}}
\def\RR{\mathbb R}
\def\bs{\bold s}

\def\beq{\begin{equation}}
\def\eeq{\end{equation}}
\def\st{^{\text{st}}}
\def\lra{\longrightarrow}
\def\image{\mathrm{Im}}
\def\ind{\mathfrak{in}}

\def\Po{{\mathbf P}^1}
\def\fD{{\mathfrak D}}
\let\lam=\Lambda
\def\dual{^\vee}

\def\NN{\mathbb N}

{\vskip-\lastskip\medskip
  \noindent
  {\em #1.}\enspace
  }%
{\qed\par\medskip
  }

\begin{document}
\title[Cohomological Crepant Resolution Conjecture]
{The Cohomological Crepant Resolution Conjecture for the Hilbert-Chow morphisms}

%\author[Jun Li]{Jun Li$^1$}
%\address{Department of Mathematics, Stanford University, Stanford, CA 94305, USA} 
%\email{jli@math.stanford.edu}
%\thanks{${}^1$Partially supported by an NSF grant}

\author[Wei-Ping Li]{Wei-Ping Li}
\address{Department of Mathematics, HKUST, Clear Water Bay, Kowloon, Hong Kong} 
\email{mawpli@ust.hk}
%\thanks{${}^1$Partially supported by the grant HKUST6170/99P}

\author[Zhenbo Qin]{Zhenbo Qin$^\dagger$}
\address{Department of Mathematics, University of Missouri, Columbia, MO 65211, USA} 
\email{qinz@missouri.edu}
\thanks{${}^\dagger$Partially supported by an NSF grant}

%\date{\today}

\begin{abstract}
In this paper, we prove that Ruan's Cohomological Crepant Resolution Conjecture holds 
for the Hilbert-Chow morphisms. There are two main ideas in the proof. The first one
is to use the representation theoretic approach proposed in \cite{QW} which involves 
vertex operator techniques. The second is to prove certain universality structures
about the $3$-pointed genus-$0$ extremal Gromov-Witten invariants of the Hilbert schemes
by using the indexing techniques from \cite{LiJ}, the product formula from \cite{Beh2} 
and the co-section localization from \cite{KL1, KL2, LL}. We then reduce Ruan's Conjecture from 
the case of an arbitrary surface to the case of smooth projective toric surfaces 
which has already been proved in \cite{Che}.
\end{abstract}

\maketitle
%\setcounter{tocdepth}{1}
%\tableofcontents

%%
%%
%%
%%
%%
%%
%%
\section{\bf Introduction}

In \cite{ChR}, Chen and Ruan defined the orbifold cohomology ring $H^*_{\rm CR}(Z)$
for an orbifold $Z$. Motivated by orbifold string theory from physics, 
Ruan \cite{Ruan} proposed the Cohomological Crepant Resolution Conjecture.
It eventually evolved into the Crepant Resolution Conjecture after
the work of Bryan-Graber, Coates-Corti-Iritani-Tseng and Coates-Ruan 
\cite{BG, CCIT, CoR}. Roughly speaking, assuming that an orbifold $Z$ has 
a crepant resolution $W$, then the Crepant Resolution Conjecture
predicts that the orbifold Gromov-Witten theory of $Z$ is 
ring isomorphic (in the sense of analytic continuations, 
symplectic transformations and change of variables of type $q = -e^{i \theta}$) to 
the ordinary cohomology ring of $W$ plus those quantum corrections 
on $W$ which are related to curves contracted by the crepant resolution. 
We refer to \cite{BG, Che, Coa} and the references there for other 
excellent examples confirming the Crepant Resolution Conjecture.

In this paper, we prove that Ruan's Cohomological Crepant Resolution Conjecture holds 
for the Hilbert-Chow morphisms. Let $X$ be a smooth projective complex surface, 
and $\Xn$ be the Hilbert scheme of points in $X$. Sending an element in $\Xn$ to
its support in the symmetric product $X^{(n)}$, we obtain the Hilbert-Chow morphism 
$\rho_n: \Xn \rightarrow X^{(n)}$, which is a crepant resolution of singularities.
Let $H^*_{\rho_n}(\Xn)$ be the quantum corrected cohomology ring (see Sect.~4 for details).

\begin{theorem}  \label{Intro-Thm1}
Let $X$ be a simply connected smooth projective surface. Then, Ruan's Cohomological 
Crepant Resolution Conjecture holds for the Hilbert-Chow morphism
$\rho_n$, i.e., the rings $H^*_{\rho_n}(\Xn)$ and $H^*_{\rm CR}(X^{(n)})$ are isomorphic.
\end{theorem}

This theorem has been proved earlier when $n = 2, 3$ \cite{ELQ, LQ}, when $K_X$ is trivial 
\cite{FG, LS}, and when $X$ is a smooth toric surface \cite{Che}. We also refer to 
\cite{LQW4, MO, OP, QW, Zho} for discussions when $X$ is quasi-projective.

There are two main ingredients in our proof of Theorem~\ref{Intro-Thm1}. The first one is 
the axiomatization approach originated from \cite{Leh, LQW1} and formulated in \cite{QW}. 
This approach involves Heisenberg algebra actions and vertex operator techniques pioneered 
in \cite{Gro, Nak}. We recall that a graded Frobenius algebra over a field $k$ is 
a finite dimensional graded vector space $A$ with 
a graded associative multiplication $A \otimes A \to A$ and unit element $1_A$
together with a linear form $T: A \to k$ such that the induced bilinear form 
$\langle a, b \rangle := T(ab)$ is nondegenerate. For $k \ge 1$, the $k$-th co-product 
$\tau_{k*}: A \to A^{\otimes k}$ is defined by requiring 
$
\langle \tau_{k*}(a), b_1 \otimes \cdots \otimes b_k \rangle 
= T(ab_1 \cdots b_k).
$
Now the axiomatization in \cite{QW} states 
that the algebra structure on each $A^{[n]}$ in a sequence of 
graded Frobenius algebras $A^{[n]}$ ($n \ge 0$) is determined if 
\begin{enumerate}
\item[{(A1)}] the direct sum $\bigoplus_n A^{[n]}$ affords the structure 
of the Fock space of a Heisenberg algebra modeled on $A := A^{[1]}$.

\item[{(A2)}] There exists a sequence
of elements $\W G_k(\alpha,n) \in A^{[n]}$ depending on $\alpha \in
A$ (linearly) and a non-negative integer $k$. Define the operators $\W {\mathfrak G}_k(\alpha)$
on $\bigoplus_n A^{[n]}$ which act on the component $A^{[n]}$ via 
multiplication by $\W G_k(\alpha,n) \in A^{[n]}$. The operators $\W {\mathfrak G}_k(\alpha)$ and 
the Heisenberg generators satisfy:
\begin{eqnarray}  
\W {\mathfrak G}_1(1_A) &=& -\displaystyle{\frac16} 
  :\mathfrak a^3:_0 (\tau_{3*}1_A),                \label{Axiom21}    \\
{[} \W {\mathfrak G}_k(\alpha), \mathfrak a_{-1}(\beta)]
  &=& {1 \over k!} \, \mathfrak a^{\{k\}}_{-1}(\alpha \beta)            \label{Axiom22}
\end{eqnarray}
where $:\mathfrak a^3:_0$ is the zero mode in the normally ordered product 
$:\mathfrak a^3:$, and $\mathfrak a^{\{k\}}_{-1}(\alpha)$ denotes the $k$-th 
derivative with $\mathfrak a^{\{ 0 \}}_{-1}(\alpha) = \mathfrak a_{-1}(\alpha)$
%$\mathfrak a^{\{1\}}_{-1}(\alpha) = [\W {\mathfrak G}_1(1_A), \mathfrak a_{-1}(\alpha)]$
and $\mathfrak a^{\{k\}}_{-1}(\alpha) = [\W {\mathfrak G}_1(1_A), 
\mathfrak a^{\{k-1\}}_{-1}(\alpha)]$ for $k \ge 1$.
\end{enumerate}

When (A1) and (A2) are satisfied, the algebra $A^{[n]}$ is generated by the elements
$$
\W G_k(\alpha,n) \in A^{[n]}, \quad \alpha \in A, k \ge 0.
$$
In addition, the product is determined by (\ref{Axiom21}) and (\ref{Axiom22}).
On one hand, with $A^{[n]} = H^*_{\rm CR}(X^{(n)})$ (viewed as an algebra over $\mathbb C$), 
the results in \cite{QW} (see Theorem~\ref{orb:heis} below) indicate that (A1) and (A2) 
hold for the rings $H^*_{\rm CR}(X^{(n)})$. On the other hand, by \cite{Gro, Nak} and \cite{LL}, 
the rings $A^{[n]} = H^*_{\rho_n}(\Xn) (= H^*(\Xn)$ as vector spaces) also satisfy (A1) and 
(\ref{Axiom21}). Moreover, using \cite{Che}, we prove that (\ref{Axiom22}) 
holds when $X$ is a smooth projective toric surface.

To prove that the rings $A^{[n]} = H^*_{\rho_n}(\Xn)$ satisfy (\ref{Axiom22})
for an arbitrary surface $X$, our second main ingredient comes into play. 
It involves finer analysis of the virtual fundamental cycle using the method in \cite{LiJ} and 
the co-section localization technique in \cite{KL1, KL2, LL}. Let $X^{[n, d]}$ be 
the moduli space of $3$-pointed genus-$0$ degree-$d$ stable maps to $\Xn$.
By \cite{LL}, every stable map $(\varphi, C) \in X^{[n, d]}$ has a standard decomposition
$\varphi = (\varphi_1, \ldots, \varphi_l) \in X^{[n, d]}$ where the stable reduction
$\varphi_i\st$ is contained in $X^{[n_i, d_i]}$ for some $n_i$ and $d_i$, 
$\rho_{n_i}(\im(\varphi_i)) = n_i x_i$, the points $x_1, \ldots, x_l$ are distinct,
and $\varphi(p) = \sum_i \varphi _i(p)$ for all $p \in C$. We use the ideas from \cite{LiJ} to
index the support of $\rho_n(\im(\varphi)) = \sum_i n_i x_i \in X^{(n)}$. 
This is done by introducing the notion of 
$3$-pointed genus-$0$ degree-$\delta$ $\alpha$-stable maps to $\Xn$, where $\alpha =
(\alpha_1, \cdots, \alpha_l)$ denotes a partition of the set $[n] = \{1, \ldots, n\}$
and $\delta = (\delta_1, \cdots, \delta_l)$ with $\delta_i$'s being nonnegative integers. 
The set of such pairs $(\alpha, \delta)$ with $\sum_i \delta_i = d$ is denoted by 
$\mathcal P_{[n], d}$.  The techniques in \cite{LiJ} and the product formula in \cite{Beh2} for 
Gromov-Witten invariants enable us to express the virtual fundamental cycle
$[X^{[n, d]}]^{\rm vir}$ in terms of certain discrepancy cycles 
$[\Theta^{[\![\alpha, \delta]\!]}]$, $(\alpha, \delta) \in \mathcal P_{[n], d}$.
{In fact, one of the key points in the paper is to study such decomposition of $[X^{[n, d]}]^{\rm vir}$ as a sum of cycles indexed by the partition type of $\rho_n(\im(\varphi))\in X^{(n)}$. However, this cannot  be done on the moduli space $X^{[n, d]} $. The technique to overcome this impasse is to introduce the Hilbert scheme of $\alpha$-points $X^{[\![n]\!]}$ and an non-separated space $X^{[\![\le n]\!]}$ following \cite{LiJ}. Then the cycle $\ev_*([X^{[\![n, d]\!]}]^{\rm vir})$ is a sum of various $[\Theta^{[\![\alpha, \delta]\!]}]$ in $(X^{[\![\le n]\!]})^3$. Even though the space $X^{[\![\le n]\!]}$ is not Hausdorff in analytic topology, all the operations involving $X^{[\![\le n]\!]}$  in this paper are all algebraic topological, such as pullbacks of cohomology classes and cap products, which are defined on any topological spaces.}
Combining with the co-section localization theory in \cite{KL1, KL2, LL},
pairings with $[\Theta^{[\![\alpha, \delta]\!]}]$ can be studied via $C^\infty$-maps from 
$X$ to the Grassmannians. For $d \ge 1$, we assemble those $[\Theta^{[\![\alpha, \delta]\!]}]$, 
$(\alpha, \delta) \in \mathcal P_{[n], d}$ with $\delta_i > 0$ for every $i$ into a homology class
$\mathfrak Z_{n, d} \in H_*((\Xn)^3)$. {Note that we are back to the original Hilbert scheme $X^{[n]}$. }
Now the structure of the $3$-pointed genus-$0$ extremal Gromov-Witten invariants of 
$\Xn$ is given by the following two theorems. 

\begin{theorem}  \label{A1ToAk}
Let $A_1, A_2, A_3 \in H^*(\Xn)$ be Heisenberg monomial classes, and  
$\pi_{m, i}$ be the $i$-th projection on $(X^{[m]})^3$.
Then, $\langle A_1, A_2, A_3 \rangle_{0, d \beta_n}$ is equal to
\begin{eqnarray}  \label{A1ToAk.0}
\sum_{m \le n} \sum_{A_{1,1} \circ A_{1,2} = A_1 \atop {A_{2,1} \circ A_{2,2} = A_2 \atop
A_{3,1} \circ A_{3,2} = A_3}}  \langle A_{1,1}, A_{2,1}, A_{3,1}  \rangle \cdot 
\left \langle  {\mathfrak Z}_{m, d}, \, \prod_{i=1}^3 \pi_{m, i}^*A_{i,2} \right\rangle.
\end{eqnarray}
\end{theorem}

\begin{theorem}  \label{KX2-universal}
Let $A_1, A_2, A_3 \in H^*(\Xn)$ be Heisenberg monomial classes.
\begin{enumerate}
\item[{\rm (i)}] If $A_i$ contains a factor $\mathfrak a_{-j}(x)$ for some $i$, then 
$\langle {\mathfrak Z}_{n, d} ,\,\,\prod_{i=1}^3 \pi_{n, i}^*A_i \rangle  = 0$.

\item[{\rm (ii)}] For $1 \le i \le 3$, let $A_i = \mathfrak a_{-\la^{(i)}}(1_X) 
\mathfrak a_{-n_{i,1}}(\alpha_{i,1}) \cdots \mathfrak a_{-n_{i,u_i}}(\alpha_{i,u_i}) \vac$
where $u_i \ge 0$ and $|\alpha_{i,1}| = \ldots = |\alpha_{i,u_i}| = 2$. Then, 
\begin{eqnarray}   \label{KX2-universal.0}
\left\langle {\mathfrak Z}_{n, d}, \, \prod_{i=1}^3 \pi_{n, i}^*A_i \right\rangle
= \prod_{i=1}^3 \prod_{j=1}^{u_i} \langle K_X, \alpha_{i,j} \rangle \cdot p
\end{eqnarray}
where $p$ is a polynomial in $\langle K_X, K_X \rangle$ whose degree is at most 
$(n - \sum_{i, j} n_{i, j})/2$, and whose coefficients depend only on $d, n, \la^{(i)}, n_{i,j}$ 
(and hence are independent of the surface $X$ and the classes $\alpha_{i,j}$).
\end{enumerate}
\end{theorem}

We refer to Definition~\ref{AOverB} for the operation $\circ$ appearing in 
\eqref{A1ToAk.0}, and to Definition~\ref{partition} for the notation  
$\mathfrak a_{-\la^{(i)}}(1_X)$ appearing in Theorem~\ref{KX2-universal}~(ii). Geometrically, 
we may think of the pairing $\langle {\mathfrak Z}_{m, d},  \,\prod_{i=1}^3 \pi_{m, i}^*A_{i,2}\rangle$ in
\eqref{A1ToAk.0} as the contributions of the non-constant components $\varphi_i$ 
in the standard decomposition of $\varphi = (\varphi_1, \ldots, \varphi_l) \in X^{[n, d]}$,
while those constant components $\varphi_i$ contribute to the factor
$\langle A_{1,1}, A_{2,1}, A_{3,1} \rangle$ in \eqref{A1ToAk.0}.

Using Theorem~\ref{A1ToAk} and Theorem~\ref{KX2-universal},
we are able to reduce the proof of (\ref{Axiom22}) for $A^{[n]} = H^*_{\rho_n}(\Xn)$ 
from an arbitrary surface $X$ to the case when $X$ is a smooth projective toric surface.
This proves (\ref{Axiom22}) for $A^{[n]} = H^*_{\rho_n}(\Xn)$ and hence 
completes the proof of Theorem~\ref{Intro-Thm1}.

%We remark that in a straight-forward fashion, our ideas and methods can be modified to handle
%the quantum orbifold cohomology rings of the symmetric products $X^{(n)}$.
%It follows that the Crepant Resolution Conjecture for the Hilbert-Chow morphisms
%$\rho_n: \Xn \rightarrow X^{(n)}$ is true as well. Details will appear in a sequel.

%In another direction, when $k \ge 4$, results similar to Theorem~\ref{A1ToAk} and 
%Theorem~\ref{KX2-universal} hold for the $k$-pointed genus-$0$ extremal Gromov-Witten 
%invariants of $\Xn$. However, since $\overline{\mathfrak M}_{0, k}$ is no longer 
%a smooth point, the technicalities involved are much more subtle and complicated. 

Finally, this paper is organized as follows. In Sect.~2, we review the Hilbert schemes of 
points on surfaces and Heisenberg algebras. In Sect.~3, 
we recall from \cite{QW} the results regarding $H^*_{\rm CR}(X^{(n)})$.
In Sect.~4, we review Ruan's Cohomological Crepant Resolution Conjecture.
%in the setting of the Hilbert-Chow morphisms. 
In Sect.~5, we prove Theorem~\ref{A1ToAk} and Theorem~\ref{KX2-universal}. 
In Sect.~6, we verify (\ref{Axiom22}) and Theorem~\ref{Intro-Thm1}.

\medskip\noindent
{\bf Conventions}: All the homology and cohomology groups are in $\C$-coefficients 
unless otherwise specified. For a subvariety $Z$ of a smooth projective variety $Y$, 
we will use $Z$ or $[Z]$ to denote the corresponding cycle/cohomology class, 
and use $1_Y$ to denote the fundamental cohomology class of $Y$. 
The symbol $A \cdot B$ denotes  the cup product for $A, B \in H^*(Y)$.
%, or the pairing for $A \in H_*(Y)$ and $B \in H^*(Y)$. 
For $A_1, \ldots, A_k \in H^*(Y)$, let
$\langle A_1, \ldots, A_k \rangle = \int_Y A_1 \cdots A_k$. By abuse of notation, for 
$A\in H_*(W)$ and $B\in H^*(W)$ of an arbitrary topological space $W$, $\langle A, B\rangle$ also stands for the natural paring between the homology group and the cohomology
group. For subsets $A$ and $B$ of  $W$, $A\cap B\subset W$ stands for the intersection of the two subsets; for $A\in H_*(W)$ and 
$B\in H^*(W)$, $A\cap B\in H_*(W)$ stands for the cap product.

\medskip\noindent
{\bf Acknowledgment}: The authors thank Professor Jun Li for offering enormous helps 
and suggesting valuable ideas, without which this paper would be impossible to complete. 
In particular, the crucial Lemma~\ref{lem0.5}, Lemma~\ref{lma5.6-LiJ} and their proofs 
are due to him. The authors also thank Professors Wan Keng Cheong, Yongbin Ruan and 
Weiqiang Wang for stimulating discussions.

\section{\bf Hilbert schemes of points on surfaces}
\label{sect_Hilbert}

%In this section, we review standard facts about the Hilbert schemes of points on surfaces.
Let $X$ be a smooth projective complex surface with the canonical
class $\KX$ and the Euler class $e_X$, and $\Xn$ be the Hilbert
scheme of points in $X$. An element in $\Xn$ is represented by a
length-$n$ $0$-dimensional closed subscheme $\xi$ of $X$. 
%For $\xi \in \Xn$, let $I_{\xi}$ be the corresponding sheaf of ideals. 
It is well known that $\Xn$ is smooth. 
%Sending an element in $\Xn$ to its support in the symmetric product $X^{(n)}$, 
%we obtain the Hilbert-Chow morphism 
%$$\rho_n: \Xn \rightarrow X^{(n)},$$
%which is a resolution of singularities. 
For a subset $Y \subset X$, define
$$
M_n(Y) = \{ \xi \in \Xn|\, \Supp(\xi) = \{x\} \text{ for some $x \in Y$} \}.
$$
Let ${\mathcal Z}_n=\{(\xi, x) \subset \Xn\times X \, | \, x\in {\rm Supp}{(\xi)}\}$
be the universal codimension-$2$ subscheme of $\Xn\times X$.
Let $p_1$ and $p_2$ be the two projections of $\Xn \times X$. Let
\begin{eqnarray*}
\fock = \bigoplus_{n=0}^{+\infty} \Hn
\end{eqnarray*}
be the direct sum of total cohomology groups of the Hilbert schemes $\Xn$.

For $m \ge 0$ and $n > 0$, let $Q^{[m,m]} = \emptyset$ and define
$Q^{[m+n,m]}$ to be the closed subset:
$$
\{ (\xi, x, \eta) \in X^{[m+n]} \times X \times X^{[m]} \, | \,
\xi \supset \eta \text{ and } \mbox{Supp}(I_\eta/I_\xi) = \{ x \} \}.
$$

We recall Nakajima's definition of the Heisenberg operators
\cite{Nak}. Let $n > 0$. The linear operator $\mathfrak
a_{-n}(\alpha) \in \End(\fock)$ with $\alpha \in H^*(X)$ is defined by
$$
\mathfrak a_{-n}(\alpha)(a) = \tilde{p}_{1*}([Q^{[m+n,m]}] \cdot
\tilde{\rho}^*\alpha \cdot \tilde{p}_2^*a)
$$
 for $a \in H^*(X^{[m]})$, where $\tilde{p}_1, \tilde{\rho},
\tilde{p}_2$ are the projections of $X^{[m+n]} \times X \times
X^{[m]}$ to $X^{[m+n]}, X, X^{[m]}$ respectively. Define the linear operator
$\mathfrak a_{n}(\alpha) \in \End(\fock)$ to be $(-1)^n$ times the
operator obtained from the definition of $\mathfrak
a_{-n}(\alpha)$ by switching the roles of $\tilde{p}_1$ and $\tilde{p}_2$. 
%We often refer to $\mathfrak a_{-n}(\alpha)$ (respectively, $\mathfrak a_n(\alpha)$) 
%as the {\em creation} (respectively,  {\em annihilation}) operator. 
We also set $\mathfrak a_0(\alpha) =0$.

For $n > 0$ and a homogeneous class $\alpha \in H^*(X)$, let
$|\alpha| = s$ if $\alpha \in H^s(X)$, and let $G_i(\alpha, n)$ be
the homogeneous component in $H^{|\alpha|+2i}(\Xn)$ of
\[
 G(\alpha, n) = p_{1*}(\ch({\mathcal O}_{{\mathcal Z}_n}) \cdot
 p_2^*{\rm td}(X) \cdot p_2^*\alpha) \in \Hn
\]
where $\ch({\mathcal O}_{{\mathcal Z}_n})$ denotes the Chern character of 
${\mathcal O}_{{\mathcal Z}_n}$ and ${\rm td}(X) $ denotes the Todd class. 
Set $G_i(\alpha, 0) =0$. We extend the notion $G_i(\alpha, n)$ linearly to 
an arbitrary class $\alpha \in H^*(X)$. 
The {\it Chern character operator} ${\mathfrak G}_i(\alpha) \in
\End({\fock})$ is defined to be the operator acting on the
component $H^*(\Xn)$ by the cup product with $G_i(\alpha, n)$. It
was proved in \cite{LQW1} that the cohomology ring of $\Xn$ is
generated by the classes $G_{i}(\alpha, n)$ where $0 \le i < n$
and $\alpha$ runs over a linear basis of $H^*(X)$. Let $\mathfrak
d = \mathfrak G_1(1_X)$ where $1_X$ is the fundamental cohomology
class of $X$. The operator $\mathfrak d$ was first introduced in \cite{Leh}.
For a linear operator $\mathfrak f \in \End(\fock)$, define its
{\it derivative} $\mathfrak f'$ by $\mathfrak f' = [\mathfrak d,
\mathfrak f]$. The $k$-th derivative $\mathfrak f^{(k)}$ is
defined inductively by $\mathfrak f^{(k)} = [\mathfrak d, \mathfrak f^{(k-1)}]$.

Let $:\frak a_{m_1}\frak a_{m_2}:$ be $\frak a_{m_1}\frak a_{m_2}$
when $m_1 \le m_2$ and $\frak a_{m_2}\frak a_{m_1}$ when 
$m_1 > m_2$. For $k \ge 1$, $\tau_{k*}: H^*(X) \to H^*(X^k)$ is the linear
map induced by the diagonal embedding $\tau_k: X \to X^k$, and
$\mathfrak a_{m_1} \cdots \mathfrak a_{m_k}(\tau_{k*}(\alpha))$ denotes
$
\sum_j \mathfrak a_{m_1}(\alpha_{j,1}) \cdots \mathfrak a_{m_k}(\alpha_{j,k})
$
when $\tau_{k*}\alpha = \sum_j \alpha_{j,1} \otimes \cdots \otimes
\alpha_{j, k}$ via the K\"unneth decomposition of $H^*(X^k)$.

The following is a combination of various theorems from \cite{Nak,
Gro, Leh, LQW1}. Our notations and convention of signs are
consistent with \cite{LQW2}.

\begin{theorem} \label{commutator}
%Let $X$ be a smooth complex projective surface with the canonical
%class $K_X$ and the Euler class $e_X$.
Let $k \ge 0, n,m \in \Z$ and $\alpha, \beta \in H^*(X)$. Then,
\begin{enumerate}
\item[{\rm (i)}] the operators $\mathfrak a_n(\alpha)$ satisfy
a Heisenberg algebra commutation relation:
\begin{eqnarray*}
\displaystyle{[\mathfrak a_m(\alpha), \mathfrak a_n(\beta)] = -m
\; \delta_{m,-n} \cdot \langle \alpha, \beta \rangle \cdot {\rm Id}_{\fock}}.
\end{eqnarray*}
%where ${\rm Id}_{\fock}$ stands for the identity map of $\fock$.
The space $\fock$ is an irreducible module over the Heisenberg
algebra generated by the operators $\mathfrak a_n(\alpha)$ with a
highest~weight~vector $\vac=1 \in H^0(X^{[0]}) \cong \mathbb C$.

\item[{\rm (ii)}] $\displaystyle{\mathfrak G_1(\alpha) 
= - \frac16 :\mathfrak a^3:_0(\tau_{3*}\alpha) - \sum_{n>0} 
  \frac{n-1}2 :\mathfrak a_n \mathfrak a_{-n}:(\tau_{2*}(K_X \alpha))}$.

\item[{\rm (iii)}] $\displaystyle{[\mathfrak G_k(\alpha),
\mathfrak a_{-1}(\beta)] = \frac{1}{k!} \cdot \mathfrak
a_{-1}^{(k)}(\alpha \beta)}$.
\end{enumerate}
\end{theorem}

The Lie brackets in Theorem~\ref{commutator} are understood in the super
sense according to the parity of the degrees of the
cohomology classes involved. Also, Theorem~\ref{commutator}~(i) implies that 
$\fock$ is linearly spanned by the cohomology classes
$
\mathfrak a_{-n_1}(\alpha_1) \cdots \mathfrak a_{-n_k}(\alpha_k) \vac
$
where $k \ge 0$ and $n_1, \ldots, n_k > 0$. These classes are called
{\it Heisenberg monomial classes}.

\begin{definition} \label{partition}
Let $\alpha \in H^*(X)$, and $\lambda = (\cdots
(-2)^{m_{-2}}(-1)^{m_{-1}} 1^{m_1}2^{m_2} \cdots)$ be a {\em
generalized partition} of the integer $n = \sum_i i m_i$ whose
part $i\in \Z$ has multiplicity $m_i$. Define $\ell(\lambda) =
\sum_i m_i$, $|\lambda| = \sum_i i m_i = n$, $\lambsq  = \sum_i
i^2 m_i$, $\lambda^! = \prod_i m_i!$, and
$$
   \mathfrak a_{\lambda}(\alpha) 
=\prod_i \big ( \mathfrak a_i(\alpha) \big )^{m_i}, \quad
   \mathfrak a_{\lambda}(\tau_*\alpha) 
=\left ( \prod_i (\mathfrak a_i)^{m_i} \right ) (\tau_{\ell(\lambda)*}\alpha)
$$
where $\prod_i (\mathfrak a_i)^{m_i} $ is understood to be
$\cdots \mathfrak a_{-2}^{m_{-2}} \mathfrak a_{-1}^{m_{-1}}
 \mathfrak a_{1}^{m_{1}} \mathfrak a_{2}^{m_{2}}\cdots$.
A generalized partition becomes a {\em partition} in the usual sense if  
$m_ i = 0$ for every $i < 0$. A partition $\lambda$ of $n$ is denoted by $\lambda \vdash n$.
\end{definition}

The next three theorems were proved in \cite{LQW3}.

\begin{theorem} \label{deriv_th}
Let $k \ge 0$, $n \in \Z$, and $\alpha\in H^*(X)$. Then,
$\mathfrak a_n^{(k)}(\alpha)$ equals
\begin{eqnarray*}
& &(-n)^k k! \left ( \sum_{\ell(\lambda) = k+1, |\lambda|=n}
   {1 \over \lambda^!} \mathfrak a_{\lambda}(\tau_{*}\alpha) -
   \sum_{\ell(\lambda) = k-1, |\lambda|=n}
   {\lambsq - 1 \over 24 \lambda^!}
   \mathfrak a_{\lambda}(\tau_{*}(e_X\alpha)) \right )   \\
&+&\sum_{{\epsilon} \in \{K_X, K_X^2\}}  \,\,\,
   \sum_{\ell(\lambda) = k+1-|{\epsilon}|/2, |\lambda|=n}
   {f_{|{\epsilon}|}(\lambda) \over \lambda^!}
   \mathfrak a_{\lambda}(\tau_{*}(\epsilon\alpha))
\end{eqnarray*}
where all the numbers $f_{|{\epsilon}|}(\lambda)$ are independent of $X$ and $\alpha$.
\end{theorem}

\begin{theorem} \label{char_th}
Let $k \ge 0$ and $\alpha\in H^*(X)$. Then, $\mathfrak
G_k(\alpha)$ is equal to
\begin{eqnarray*}
& &- \sum_{\ell(\lambda) = k+2, |\lambda|=0}
   {1 \over \lambda^!} \mathfrak a_{\lambda}(\tau_{*}\alpha)
   + \sum_{\ell(\lambda) = k, |\lambda|=0}
   {\lambsq - 2 \over 24\lambda^!}
   \mathfrak a_{\lambda}(\tau_{*}(e_X\alpha))  \\
&+&\sum_{{\epsilon} \in \{K_X, K_X^2\}}  \,\,\,
   \sum_{\ell(\lambda) = k+2-|{\epsilon}|/2, |\lambda|=0}
   {g_{|{\epsilon}|}(\lambda) \over \lambda^!}
   \mathfrak a_{\lambda}(\tau_{*}(\epsilon\alpha))
\end{eqnarray*}
where all the numbers $g_{|{\epsilon}|}(\lambda)$ are independent of $X$
and $\alpha$.
\end{theorem}

\begin{theorem} \label{g_{1_X, e}}
Let $n \ge 1$, $k \ge 0$, and $\alpha \in H^*(X)$. Then, $G_k(\alpha, n)$ is equal to
\begin{eqnarray*}
& &\sum_{0 \le j \le k} \sum_{\lambda
   \vdash (j+1) \atop \ell(\lambda)=k-j+1}
   {(-1)^{|\lambda|-1} \over \lambda^! \cdot |\lambda|!}
   \cdot {\bf 1}_{-(n-j-1)} \mathfrak a_{-\lambda}(\tau_*\alpha)\vac \\
&+&\sum_{0 \le j \le k} \sum_{\lambda \vdash (j+1) \atop \ell(\lambda)=k-j-1}
   {(-1)^{|\lambda|} \over \lambda^! \cdot |\lambda|!}
   \cdot {|\lambda| + \lambsq - 2 \over 24}
   \cdot {\bf 1}_{-(n-j-1)}
   \mathfrak a_{-\lambda}(\tau_*(e_X \alpha))\vac  \\
&+&\sum_{\epsilon \in \{K_X, K_X^2\} \atop 0 \le j \le k}
   \sum_{\lambda \vdash (j+1) \atop \ell(\lambda)=k-j+1-|{\epsilon}|/2}
   {(-1)^{|\lambda|}g_{|{\epsilon}|}(\lambda+(1^{j+1}))
   \over \lambda^! \cdot |\lambda|!} \cdot {\bf 1}_{-(n-j-1)}
   \mathfrak a_{-\lambda}(\tau_*(\epsilon\alpha))\vac
\end{eqnarray*}
where ${\bf 1}_{-(n-j-1)}$ denotes $\mathfrak
a_{-1}(1_X)^{n-j-1}/(n-j-1)!$ when $(n-j-1) \ge 0$ and is $0$ when
$(n-j-1) < 0$, the universal function $g_{|{\epsilon}|}$ is from
Theorem~\ref{char_th}, and $\lambda+(1^{j+1})$ is the partition
obtained from $\lambda$ by adding $(j+1)$ to the multiplicity of $1$.
\end{theorem}

\begin{lemma} \label{tau_k_tau_{k-1}}
%Let $k, s \ge 1$, $n_1, \ldots, n_k, m_1, \ldots, m_s \in \Z$, and
%$\alpha, \beta \in H^*(X)$. Then, the commutator
$[\mathfrak a_{n_1} \cdots \mathfrak a_{n_{k}} (\tau_{k*}\alpha),
\mathfrak a_{m_1} \cdots \mathfrak a_{m_{s}}(\tau_{s*}\beta)]$ is equal to
\begin{eqnarray*}
-\sum_{t=1}^k \sum_{j=1}^s n_t \delta_{n_t,-m_j} \cdot \left (
\prod_{l=1}^{j-1} \mathfrak a_{m_l} \prod_{1 \le u \le k, u
\ne t} \mathfrak a_{n_u} \prod_{l=j+1}^{s} \mathfrak a_{m_l}
\right )(\tau_{(k+s-2)*}(\alpha\beta)).
\end{eqnarray*}
\end{lemma}

The above lemma was proved in \cite{LQW2}, and will be used implicitly in 
many proofs throughout the paper.
The following geometric result was proved in \cite{LQW5}.

\begin{proposition} \label{prop_geom}
%Let $l, k \ge 0$, $s_i \ge 0$ ($1 \le i \le l$), $n_i > 0$ ($1 \le i \le k$). 
Let the classes $\alpha_1, \ldots, \alpha_k \in \oplus_{i=1}^4 H^i(X)$ be 
respectively represented by the cycles $X_1, \ldots, X_k \subset X$ 
in general position. Then, the Heisenberg monomial class 
\begin{eqnarray*}
\displaystyle{\left (
\prod_{i=1}^t {\frak a_{-i}(1_X)^{s_i}\over s_i!} \right )
\left ( \prod_{j=1}^k \frak a_{-n_j}(\alpha_j) \right ) \vac}
\end{eqnarray*}
is represented by the closure of the subset consisting of elements of the form
\begin{eqnarray}  \label{prop_geom.1}
\sum_{i=1}^t (\xi_{i, 1} + \ldots + \xi_{i, s_i}) + \sum_{j=1}^k \xi_j
\end{eqnarray}
where $\xi_{i, m} \in M_i(x_{i, m})$ for some $x_{i, m} \in X$, 
$\xi_j \in M_{n_j}(x_j)$ for some $x_j \in X_j$, 
and all the points $x_{i, m}$, $1 \le i \le t$, $1 \le m \le s_i$ 
and $x_j$, $1 \le j \le k$ are distinct.
\end{proposition}

Theorem~2.9 in \cite{LQW4} expresses a Heisenberg monomial class in terms of 
a polynomial of the classes $G_k(\gamma, n)$. The following lemma is a special case.

\begin{lemma}      \label{PolyGkAlphaRmk}
Let $\la \vdash n_0$, $\alpha \in H^*(X)$ with $|\alpha| = 2$, and $m \ge 1$. 
\begin{enumerate}
\item[{\rm (i)}] Then, the class ${\bf 1}_{-(n - n_0)} \mathfrak a_{-\la}(x) \vac \in H^*(\Xn)$
can be written as a polynomial of the classes $G_k(x, n), k \ge 0$. Moreover, the coefficients 
and the integers $k$ depend only on $\la$ (hence, are independent of $n$ and $X$);

\item[{\rm (ii)}] If the odd Betti numbers of the surface $X$ are equal to zero, then 
$$
{\bf 1}_{-(n - n_0-m)} \mathfrak a_{-\la}(x) \mathfrak a_{-m}(\alpha)\vac =
\langle K_X, \alpha \rangle \cdot F_1(n) + \sum_i G_{k_i}(\alpha, n) \cdot F_{2, i}(n)
$$ 
where $F_1(n)$ and $F_{2, i}(n)$ are polynomials of the classes $G_k(x, n), k \ge 0$. Moreover, 
the coefficients of $F_1(n), F_{2, i}(n)$ and the integers $k, k_i$ depend only on $\la$ and $m$ 
(hence, are independent of $n, \alpha$ and $X$).
\end{enumerate}
\end{lemma} 
\begin{proof}
These follow from the same proof of Theorem~2.9 in \cite{LQW4} 
by setting $\mathcal I = \mathbb C \cdot x \subset H^*(X)$ and 
$\mathcal I = \mathbb C \cdot x + \mathbb C \cdot \alpha \subset H^*(X)$ respectively.
\end{proof}

Next, we define some convenient operations which will be used intensively.

\begin{definition}   \label{AOverB}
Let $A = \mathfrak a_{-n_1}(\alpha_1) \cdots \mathfrak a_{-n_l}(\alpha_l) \vac$ 
where $n_1, \ldots, n_l > 0$. 
\begin{enumerate}
\item[{\rm (i)}] If $B = \mathfrak a_{-m_1}(\beta_1) \cdots \mathfrak a_{-m_s}(\beta_s) \vac$ 
with $m_1, \ldots, m_s > 0$, then we define
\begin{eqnarray}  \label{AOverB.1}
A \circ B = \mathfrak a_{-n_1}(\alpha_1) \cdots \mathfrak a_{-n_l}(\alpha_l) 
\mathfrak a_{-m_1}(\beta_1) \cdots \mathfrak a_{-m_s}(\beta_s) \vac.
\end{eqnarray}

\item[{\rm (ii)}] We use the symbol $B \subset A$ if 
$B = \mathfrak a_{-n_{i_1}}(\alpha_{i_1}) \cdots \mathfrak a_{-n_{i_s}}(\alpha_{i_s}) \vac$ 
with $1 \le i_1 < \ldots < i_s \le l$. In this case, we use $A/B$ or $AB^{-1}$ 
or $\displaystyle{A \over B}$ to denote the cohomology class obtained from 
$A$ by deleting the factors $\mathfrak a_{-n_{i_1}}(\alpha_{i_1}),  \ldots ,
\mathfrak a_{-n_{i_s}}(\alpha_{i_s})$.
\end{enumerate}
\end{definition}

\section{\bf The ring $H^*_{\rm CR}(X^{(n)})$}
\label{sec:orbkey}

For an orbifold $Z$, the ring $H^*_{\rm CR}(Z)$ was defined by 
Chen and Ruan \cite{ChR}. For a global orbifold $M/G$ where $M$ is 
a complex manifold with a finite group $G$ action, the ring structure of
$H^*_{\rm CR}(M/G)$ was further clarified in \cite{FG, Uri}.

Next, let $X$ be a closed complex manifold, and let $X^{(n)} = X^n/S_n$ be 
the $n$-th symmetric product of $X$. An explicit description of the
ring structure of $H^*_{\rm CR}(X^{(n)})$ has been obtained in \cite{FG}.
An alternative approach to the ring structure of $H^*_{\rm CR}(X^{(n)})$ 
is given in \cite{QW} via Heisenberg algebra actions. Put
$$
\Fock =\bigoplus_{n=0}^{+\infty} \orbsym. 
$$
In \cite{QW}, for $\alpha \in H^*(X)$ and $n \in \mathbb Z$, 
the Heisenberg operators $\mathfrak p_{n}(\alpha) \in \End (\Fock)$ were defined 
via the restriction and induction maps. Moreover for $k \ge 0$,
the elements $O^k(\alpha,n) \in H^*_{\rm CR}(X^{(n)})$ were introduced 
via the Jucys-Murphy elements in the symmetric groups. Put
$O_k(\alpha,n) = {1/k!} \cdot O^k(\alpha,n)$.
Let the operator $\mathfrak O_k(\alpha) \in \End (\Fock)$ be the orbifold ring
product with $O_k(\alpha,n)$ in $H^*_{\rm CR}(X^{(n)})$ for every $n \ge 0$.
The operator $\mathfrak O_1(1_X)$ plays the role of the boundary operator $\mathfrak
d = \mathfrak G_1(1_X)$ for the Hilbert schemes. 
Define $\mathfrak p^{\{k\}}_{m}(\alpha)$ inductively by putting 
$\mathfrak p^{\{ 0 \}}_m(\alpha) = \mathfrak p_m(\alpha)$ and 
$\mathfrak p^{\{k\}}_m(\alpha) = [\mathfrak O_1(1_X), \mathfrak p_{m}^{\{k-1\}}(\alpha)]$ for 
$k \ge 1$. The following result was proved in \cite{QW}.

\begin{theorem} \label{orb:heis}
Let $X$ be a closed complex manifold. Then,
\begin{enumerate}
\item[{\rm (i)}] the operators $\mathfrak p_n(\alpha) \in \End(\Fock)$ 
$(n\in \Z, \alpha \in H^*(X))$
generate a Heisenberg (super)algebra with commutation relations given by
\begin{eqnarray*}
[ \mathfrak p_{m} (\alpha), \mathfrak p_{n} (\beta)]
= m \delta_{m,-n} \cdot \langle \alpha, \beta \rangle \cdot \text{\rm Id}_{\Fock}
\end{eqnarray*}
where $n,m \in \Z, \;\alpha,\beta \in H^*(X)$, and
$\Fock$ is an irreducible representation of the Heisenberg algebra
with the vacuum vector $\vac =1 \in H^*(pt) \cong \C$.

\item[{\rm (ii)}] $\displaystyle{\mathfrak O_1(1_X) = -\frac16 :\mathfrak p^3:_0 (\tau_{*}1_X)}$.
In general, $\mathfrak O_k(\alpha)$ is equal to
\begin{eqnarray}       \label{Oprt-OkAlphaN}
(-1)^k \cdot \left ( \sum_{\ell(\lambda) = k+2, |\lambda|=0}
\frac{1}{\lambda^!} \mathfrak p_{\lambda}(\tau_{*}\alpha)
+ \sum_{\ell(\lambda) = k, |\lambda|= 0} \frac{\lambsq - 2}{24\lambda^!} 
\mathfrak p_{\lambda}(\tau_{*}(e_X \alpha)) \right ).
\end{eqnarray}

\item[{\rm (iii)}] $\displaystyle{[\mathfrak O_k(\alpha), \mathfrak p_{-1}(\beta)] 
= {1 \over k!} \,\mathfrak p^{\{ k \}}_{-1}(\alpha \beta)}$, and both sides are equal to
$$
(-1)^k \cdot \left ( \sum_{\ell(\lambda) = k+1, |\lambda| = -1}
   {1 \over \lambda^!} \mathfrak p_{\lambda}(\tau_{*}(\alpha \beta)) +
   \sum_{\ell(\lambda) = k-1, |\lambda|= -1}
   {\lambsq - 1 \over 24 \lambda^!}
   \mathfrak p_{\lambda}(\tau_{*}(e_X \alpha \beta)) \right ). 
$$
\end{enumerate}
\end{theorem}

Note that there is a fundamental sign difference in the two commutators of
Theorems~\ref{commutator}~(i) and Theorems~\ref{orb:heis}~(i). 
Since $O_k(\alpha,n) = \mathfrak O_k(\alpha) \mathfrak p_{-1}(1_X)^n\vac/n!$, 
we see from formula (\ref{Oprt-OkAlphaN}) that the class $O_k(\alpha, n)$ is equal to
\begin{eqnarray}       \label{OkAlphaN}
& &(-1)^k \cdot \left ( \sum_{0 \le j \le k} \sum_{\lambda \vdash (j+1) \atop \ell(\lambda)=k-j+1}
   {1 \over \lambda^! \cdot |\lambda|!}
   \cdot {\bf 1}_{-(n-j-1)} \mathfrak p_{-\lambda}(\tau_*\alpha)\vac \right . \nonumber  \\
&+&\left . \sum_{0 \le j \le k} \sum_{\lambda \vdash (j+1) \atop \ell(\lambda)=k-j-1}
   {1 \over \lambda^! \cdot |\lambda|!} \cdot {|\lambda| + \lambsq - 2 \over 24}
   \cdot {\bf 1}_{-(n-j-1)} \mathfrak p_{-\lambda}(\tau_*(e_X \alpha))\vac \right ). \quad
\end{eqnarray}
Moreover, as noted in \cite{QW}, the ring $H^*_{\rm CR}(X^{(n)})$ is completely determined
by Theorem~\ref{orb:heis}~(i), the formula of $\mathfrak O_1(1_X)$ in Theorem~\ref{orb:heis}~(ii),
and Theorem~\ref{orb:heis}~(iii). In particular, the ring $H^*_{\rm CR}(X^{(n)})$ is 
generated by the classes $O_k(\alpha,n)$ where $k \ge 0$ and $\alpha$ runs over 
a fixed linear basis of $H^*(X)$.

\section{\bf Ruan's Cohomological Crepant Resolution Conjecture}
\label{sect_RuanCCRC}

In this section, we briefly review the definition of Gromov-Witten invariants,
%and the quantum corrected cohomology ring $H^*_{\rho_n}(\Xn)$,
and recall Ruan's Cohomological Crepant Resolution Conjecture for the Hilbert-Chow morphisms. 

Let $Y$ be a smooth projective variety.
%A $k$-pointed {\it stable map} to $Y$ consists of a complete nodal curve $D$ with $k$ distinct 
%ordered smooth points $p_1, \ldots, p_k$ and a morphism $\mu: D \to Y$ such that
%the data $(\mu, D, p_1, \ldots, p_k)$ has only finitely many automorphisms. In this case, 
%the stable map is denoted by $[\mu: (D; p_1, \ldots, p_k) \to Y]$.
For a fixed homology class $\beta \in H_2(Y, \mathbb Z)$,
let $\overline {\frak M}_{g, k}(Y, \beta)$ be the coarse moduli space
parameterizing all the stable maps $[\mu: (D; p_1, \ldots, p_k) \to Y]$
such that $\mu_*[D] = \beta$ and the arithmetic genus of $D$ is $g$.
The $i$-th evaluation map ${\rm ev}_i\colon \overline {\frak M}_{g, k}(Y, \beta) \to Y$
is defined by ${\rm ev}_i([\mu: (D; p_1, \ldots, p_k) \to Y]) = \mu(p_i) \in Y$.
It is known \cite{FP, LT1, LT2, Beh1, BF} that $\overline {\frak M}_{g, k}(Y, \beta)$ is projective 
and has a virtual fundamental cycle
$[\overline {\frak M}_{g, k}(Y, \beta)]^{\text{\rm vir}} \in
A_{d_0}(\overline {\frak M}_{g, k}(Y, \beta))$ where
$d_0 = -(K_Y \cdot \beta) + (\dim (Y) - 3)(1-g) + k.$
%is the expected complex dimension of $\overline {\frak M}_{g, k}(Y, \beta)$, 
%and $A_{d_0}(\overline {\frak M}_{g, k}(Y, \beta))$
%is the Chow group of $d_0$-dimensional cycles in $\overline {\frak M}_{g, k}(Y, \beta)$. 
Let $\alpha_1, \ldots, \alpha_k \in H^*(Y)$, and ${\rm ev} = {\rm ev}_1 \times \cdots 
\times {\rm ev}_k: \overline {\frak M}_{g, k}(Y, \beta) \to Y^k$.
Then, the $k$-pointed Gromov-Witten invariant is defined by
\begin{eqnarray}\label{def-GW}
\langle \alpha_1, \ldots, \alpha_k \rangle_{g, \beta} \,\,
= \int_{[\overline {\frak M}_{g, k}(Y, \beta)]^{\text{\rm vir}}}
{\rm ev}^*(\alpha_1 \otimes \ldots \otimes \alpha_k).
\end{eqnarray}

Next, let $X$ be a smooth complex projective surface. 
%It is well-known that the Hilbert-Chow morphism $\rho_n: X^{[n]} \to X^{(n)}$ is 
%a crepant resolution of the global orbifold $X^{(n)}$. 
Define the homology class 
\begin{eqnarray}     \label{def-betaN} 
\beta_n = M_2(x_1) + x_2 + \ldots + x_{n-1} \in H_2(\Xn; \Z) 
\end{eqnarray}
where $x_1, \ldots , x_{n-1}$ are fixed distinct points in $X$.
An irreducible curve $C \subset X^{[n]}$ is contracted to 
a point by $\rho_n$ if and only if $C \sim d \beta_n$ for some integer $d > 0$. 
Let $q$ be a formal variable. For $w_1, w_2, w_3 \in H^*(\Xn)$, define a function of $q$:
$$
\langle w_1, w_2, w_3 \rangle_{\rho_n}(q) = \sum_{d \ge 0} \,\,
\langle w_1, w_2, w_3 \rangle_{0, d \beta_n} \,\, q^d.
$$

\begin{definition}
The {\it quantum corrected 
cohomology ring} $H^*_{\rho_n}(\Xn)$ is the group $H^*(\Xn)$ together 
with the {\it quantum corrected product} $w_1 \cdot_{\rho_n} w_2$ defined by 
\begin{eqnarray}   \label{def:QuanCorrProd}
\langle w_1 \cdot_{\rho_n} w_2, w_3 \rangle = \langle w_1, w_2, w_3 \rangle_{\rho_n}(-1).
\end{eqnarray}
\end{definition}

\begin{conjecture} \label{conj_ruan}
{\rm (Ruan's Cohomological Crepant Resolution Conjecture)}
The quantum corrected cohomology ring $H^*_{\rho_n}(\Xn)$ is ring isomorphic to
$H^*_{\rm CR}(X^{(n)})$.
\end{conjecture}

Our idea to deal with Conjecture~\ref{conj_ruan} is to use the axiomatization approach
mentioned in the Introduction. On one hand, letting $A^{[n]} = H^*_{\rm CR}(X^{(n)})$ 
and $\W G_k(\alpha,n) = O_k(\alpha,n)$,
we see from Theorem~\ref{orb:heis} that both (A1) and (A2) in the Introduction hold for 
the rings $H^*_{\rm CR}(X^{(n)})$. On the other hand, by \cite{Gro, Nak}, the rings 
$A^{[n]} = H^*_{\rho_n}(\Xn)$ also satisfy (A1) with $A = A^{[1]} = H^*(X)$. 
To deal with Axiom (A2) for $H^*_{\rho_n}(\Xn)$, we now define the elements 
$\W G_k(\alpha,n) \in H^*_{\rho_n}(\Xn)$.

\begin{definition}
Let $k \ge 0$ and $\alpha \in H^*(X)$. Define $\W G_k(\alpha,n) \in H^*_{\rho_n}(\Xn)$ to be
\begin{eqnarray}     \label{WGkAlphaN}  
& &\sum_{0 \le j \le k} \sum_{\lambda
   \vdash (j+1) \atop \ell(\lambda)=k-j+1}
   {(-1)^{|\lambda|-1} \over \lambda^! \cdot |\lambda|!}
   \cdot {\bf 1}_{-(n-j-1)} \mathfrak a_{-\lambda}(\tau_*\alpha)\vac \nonumber  \\
&+&\sum_{0 \le j \le k} \sum_{\lambda \vdash (j+1) \atop \ell(\lambda)=k-j-1}
   {(-1)^{|\lambda|} \over \lambda^! \cdot |\lambda|!}
   \cdot {|\lambda| + \lambsq - 2 \over 24}
   \cdot {\bf 1}_{-(n-j-1)}
   \mathfrak a_{-\lambda}(\tau_*(e_X \alpha))\vac.     \quad
\end{eqnarray}
\end{definition}

\begin{remark}   \label{WGkAlphaNRmk}
By definition, $\W G_0(\alpha,n) = {\bf 1}_{-(n-1)} \mathfrak a_{-1}(\alpha) \vac
= G_0(\alpha,n)$. Also, 
$$
\W G_1(\alpha,n) = - {1 \over 2} \, {\bf 1}_{-(n-2)} \mathfrak a_{-2}(\alpha) \vac 
= G_1(\alpha,n).
$$
%However, it is unclear whether we actually have $\W G_k(\alpha,n) = G_k(\alpha,n)$ for $k \ge 2$.
In general, we see from Theorem~\ref{g_{1_X, e}} that the class $\W G_k(\alpha,n)$ consists of
those terms in $G_k(\alpha,n)$ which do not contain the canonical divisor $K_X$.
\end{remark}

Note from the definition of the operator 
$\W {\mathfrak G}_k(\alpha)$ on $\oplus_n H^*_{\rho_n}(\Xn)$ that 
$$
    \langle \W {\mathfrak G}_k(\alpha)w_1, w_2 \rangle
= \langle \W G_k(\alpha,n) \cdot_{\rho_n} w_1, w_2 \rangle  
= \langle \W G_k(\alpha,n), w_1, w_2 \rangle_{\rho_n}(-1)
$$
for $w_1, w_2 \in H^*_{\rho_n}(\Xn)$. For convenience, we introduce the operator
$\W {\mathfrak G}_k(\alpha; q)$ by
\begin{eqnarray}  \label{OprtWGq}
\langle \W {\mathfrak G}_k(\alpha; q)w_1, w_2 \rangle
= \sum_{d \ge 0} \,\, \langle \W G_k(\alpha,n), w_1, w_2 \rangle_{0, d \beta_n} \,\, q^d.
\end{eqnarray} 

In the rest of this section, let the smooth projective surface $X$ be simply connected. 
By Remark~\ref{WGkAlphaNRmk}, $\W G_1(1_X,n) = G_1(1_X,n)$. Thus by \cite{LL}, 
\begin{eqnarray}   
%\W {\mathfrak G}_1(1_X; q)
% &=&-\displaystyle{\frac16} :\mathfrak a^3:_0 (\tau_{3*}1_X)    \nonumber \\
% & &+ \,\, \sum_{k > 0} \left ( {k \over 2}{(-q)^k+1 \over (-q)^k-1} - 
%    {1 \over 2}{(-q)+1 \over (-q)-1} \right ) \mathfrak a_{-k}\mathfrak a_{k}(\tau_{2*}K_X).  
%                     \label{OprtWG1Xq} \\
\W {\mathfrak G}_1(1_X) 
 &=&-\displaystyle{\frac16} :\mathfrak a^3:_0 (\tau_{3*}1_X).
                     \label{OprtWG1X}
\end{eqnarray}
So (\ref{Axiom21}) holds for the rings $H^*_{\rho_n}(\Xn)$ as well. 
To verify Ruan's conjecture for $\rho_n$, it remains to show that 
(\ref{Axiom22}) holds for $H^*_{\rho_n}(\Xn)$. For the right-hand-side of 
(\ref{Axiom22}), we have the following which follows from (\ref{OprtWG1X}) and 
the same proof of Theorem~\ref{deriv_th} (i.e., Theorem~4.4 in \cite{LQW3}).

\begin{lemma} \label{derivLLlma}
%Let $X$ be a simply connected smooth complex projective surface.
Let $k \ge 0$, $m \in \Z$, and $\alpha\in H^*(X)$. Then, 
$\mathfrak a_m^{\{ k \}}(\alpha)$ is equal to
\begin{equation}
(-m)^k k!\, \left ( \sum_{\ell(\lambda) = k+1, |\lambda| = m}
   {1 \over \lambda^!} \mathfrak a_{\lambda}(\tau_{*}\alpha) -
   \sum_{\ell(\lambda) = k-1, |\lambda|=m}
   {\lambsq - 1 \over 24 \lambda^!}
   \mathfrak a_{\lambda}(\tau_{*}(e_X \alpha)) \right ).        \tag*{$\qed$}
\end{equation}
\end{lemma}

Comparing with Theorem~\ref{deriv_th}, we see that the operator $\mathfrak a_m^{\{ k \}}(\alpha)$ 
consists of those terms in $\mathfrak a_m^{(k)}(\alpha)$ which do not contain the canonical divisor $K_X$.

\begin{lemma} \label{Axiom22P2}
Let $X$ be a smooth toric surface. Then (\ref{Axiom22}) holds for $X^{[n]}$.
\end{lemma}
\begin{proof}
Recall that $\Pee^2$ and the Hirzebruch surfaces $\mathbb F_a$ are smooth toric surfaces,
and admit $\mathbb T = (\C^*)^2$-actions. By the Proposition in Subsection~2.5 of \cite{Ful},
$X$ is obtained from $\Pee^2$ or $\mathbb F_a$ by a succession of blow-ups at 
$\mathbb T$-fixed points.

Now let $\mathfrak a_m^{\mathbb T}(\alpha), H^{*, \mathbb T}_{\rho_n}(\Xn)$ and 
$\mathfrak p_m^{\mathbb T}(\alpha), H^{*, \mathbb T}_{\rm CR}(X^{(n)})$ be the equivariant 
versions of $\mathfrak a_m(\alpha), H^*_{\rho_n}(\Xn)$ and $\mathfrak p_m(\alpha), 
H^*_{\rm CR}(X^{(n)})$ respectively. By \cite{Che}, the equivariant version of 
Conjecture~\ref{conj_ruan} holds for $X$, i.e., there exists a ring isomorphism
$$
\Psi_n^{\mathbb T}: H^{*, \mathbb T}_{\rm CR}(X^{(n)}) \to H^{*, \mathbb T}_{\rho_n}(\Xn) 
$$
sending $\sqrt{-1}^{n_1 + \ldots + n_s -s} \mathfrak p_{-n_1}^{\mathbb T}(\alpha_1) \cdots 
\mathfrak p_{-n_s}^{\mathbb T}(\alpha_s)\vac$ to $\mathfrak a_{-n_1}^{\mathbb T}(\alpha_1)
\cdots \mathfrak a_{-n_s}^{\mathbb T}(\alpha_s)\vac$. Note that up to a scalar factor 
which depends only on the partition $\la = (n_1, \ldots, n_s)$ and the tuple 
$\overrightarrow \alpha = (\alpha_1, \ldots, \alpha_s)$, our notation 
$\mathfrak p_{-n_1}^{\mathbb T}(\alpha_1) \cdots \mathfrak p_{-n_s}^{\mathbb T}(\alpha_s)\vac$
coincides with the notation $\la \big ( \overrightarrow \alpha \big )$ used in \cite{Che}.
Also, our notation $\mathfrak a_m^{\mathbb T}(\alpha)$ coincides with the notation 
$\mathfrak p_m(\alpha)$ used in \cite{Che}. The integer $n_1 + \ldots + n_s -s$ is the age.
Passing the map $\Psi_n^{\mathbb T}$ to the ordinary cohomology, we obtain a ring isomorphism
$$
\Psi_n: H^{*}_{\rm CR}(X^{(n)}) \to H^{*}_{\rho_n}(\Xn) 
$$
which sends $\sqrt{-1}^{n_1 + \ldots + n_s -s} \mathfrak p_{-n_1}(\alpha_1) \cdots 
\mathfrak p_{-n_s}(\alpha_s)\vac$ to $\mathfrak a_{-n_1}(\alpha_1)
\cdots \mathfrak a_{-n_s}(\alpha_s)\vac$. Using (\ref{OkAlphaN}) and (\ref{WGkAlphaN}),
we see that $\Psi_n \big ( \sqrt{-1}^k O_k(\alpha, n) \big ) = \W G_k(\alpha, n)$. 

Next, let $A = \mathfrak a_{-n_1}(\alpha_1) \cdots \mathfrak a_{-n_s}(\alpha_s)\vac
\in H^*(X^{[n-1]})$. By definition, 
\begin{eqnarray*}    
   [\W {\mathfrak G}_k(\alpha), \mathfrak a_{-1}(\beta)]A
&=&\W {\mathfrak G}_k(\alpha)\mathfrak a_{-1}(\beta) A 
    - \mathfrak a_{-1}(\beta) \W {\mathfrak G}_k(\alpha) A  \\
&=&\W {\mathfrak G}_k(\alpha, n) \cdot \mathfrak a_{-1}(\beta) A 
    - \mathfrak a_{-1}(\beta) \big ( \W {\mathfrak G}_k(\alpha, n-1) \cdot A \big ).
\end{eqnarray*}
Put $P = \mathfrak p_{-n_1}(\alpha_1) \cdots \mathfrak p_{-n_s}(\alpha_s)\vac$
and $a = n_1 + \ldots + n_s -s$. Let $\bullet$ denote the orbifold ring product.
Then, $\Psi_n \Big ( \mathfrak p_{-1}(\beta) \big ( \sqrt{-1}^k O_k(\alpha, n-1) 
\bullet \sqrt{-1}^a P \big ) \Big )$ equals
$$
\mathfrak a_{-1}(\beta) \Psi_n \Big ( \sqrt{-1}^k O_k(\alpha, n-1) 
      \bullet \sqrt{-1}^a P \Big )  
= \mathfrak a_{-1}(\beta) \big ( \W {\mathfrak G}_k(\alpha, n-1) \cdot A \big ),
$$
and $\Psi_n \big ( \sqrt{-1}^a \mathfrak p_{-1}(\beta) P \big ) = \mathfrak a_{-1}(\beta) A$. 
So $[\W {\mathfrak G}_k(\alpha), \mathfrak a_{-1}(\beta)]A$ is equal to
$$
\Psi_n \big ( \sqrt{-1}^k O_k(\alpha, n) \bullet \sqrt{-1}^a \mathfrak p_{-1}(\beta) P \big ) 
- \Psi_n \big (\mathfrak p_{-1}(\beta) \big ( \sqrt{-1}^k O_k(\alpha, n-1) 
      \bullet \sqrt{-1}^a P \big ) \big ).
$$
Since $O_k(\alpha, n) \bullet \mathfrak p_{-1}(\beta) P 
= \mathfrak O_k(\alpha)\mathfrak p_{-1}(\beta) P$, we obtain
\begin{eqnarray*}  
[\W {\mathfrak G}_k(\alpha), \mathfrak a_{-1}(\beta)]A   
= \sqrt{-1}^{k+a} \cdot \Psi_n \big ( [\mathfrak O_k(\alpha), \mathfrak p_{-1}(\beta)] P \big ).
\end{eqnarray*}
By Theorem~\ref{orb:heis}~(iii), we conclude that 
$[\W {\mathfrak G}_k(\alpha), \mathfrak a_{-1}(\beta)]A$ is equal to
$$
\sqrt{-1}^{k+a} \cdot (-1)^k \cdot \Psi_n \left ( \sum_{\ell(\lambda) = k+1 \atop |\lambda| = -1}
   {1 \over \lambda^!} \mathfrak p_{\lambda}(\tau_{*}(\alpha \beta))P 
   + \sum_{\ell(\lambda) = k-1 \atop |\lambda|= -1} {\lambsq - 1 \over 24 \lambda^!}
   \mathfrak p_{\lambda}(\tau_{*}(e_X \alpha \beta))P   \right ).
$$
Finally, by the definition of $\Psi_n$ and Lemma~\ref{derivLLlma}, 
$[\W {\mathfrak G}_k(\alpha), \mathfrak a_{-1}(\beta)]A$ is equal to
$$
\sum_{\ell(\lambda) = k+1 \atop |\lambda| = -1}
   {1 \over \lambda^!} \mathfrak a_{\lambda}(\tau_{*}(\alpha \beta))A - 
   \sum_{\ell(\lambda) = k-1 \atop |\lambda|= -1} {\lambsq - 1 \over 24 \lambda^!}
   \mathfrak a_{\lambda}(\tau_{*}(e_X \alpha \beta))A  
= {1 \over k!} \, \mathfrak a^{\{ k \}}_{-1}(\alpha \beta)A.
$$
Therefore, $[\W {\mathfrak G}_k(\alpha), \mathfrak a_{-1}(\beta)] =
\displaystyle{1/k!} \cdot \mathfrak a^{\{ k \}}_{-1}(\alpha \beta)$. 
Hence (\ref{Axiom22}) holds.
\end{proof}

\section{\bf Extremal Gromov-Witten invariants of Hilbert schemes}
\label{sect_Extremal}

In this section, we study the structure of extremal Gromov-Witten 
invariants of $\Xn$ for a smooth projective surface $X$.
We will use the ideas and approaches in \cite{LiJ}, and
adopt many presentations, notations and results directly from \cite{LiJ}. In addition,
the product formula in \cite{Beh2} and the co-section localization in \cite{KL1, KL2, LL}
for Gromov-Witten theory will play important roles.
For convenience, we assume that $X$ is simply connected. 
%Under this assumption, we do not need to worry about the signs involved
%in interchanging two cohomology classes or two Heisenberg operators.

%
%
%
%
%
%
%
\subsection{Hilbert schemes of $\alpha$-points and partial equivalence}
\def\Yn{Y^{[n]}_T}
\label{subsect_AlphaPts}
\def\hc{\mathfrak{hc}}
\def\Lam{\Lambda}
\def\fa{\mathfrak a}
\def\fd{\mathfrak d}
\def\cM{{\mathcal M}}
\def\be{{\beta,\eta}}
\def\ad{{\alpha,\delta}}
\def\EE{{\mathbb E}}
\def\LL{{\mathbb L}}
\def\cC{\mathcal C}
\def\sH{\mathscr H}
\def\bF{{\mathbf F}}
\def\sF{{\mathscr F}}
\def\bC{\mathbf C}
\def\sC{\mathscr C}
\def\AD{{\Lambda,d}}
\def\cZ{\mathcal Z}
\def\barM{\overline {\mathfrak M}}

{In this subsection, we introduce some new spaces related to Hilbert schemes to provide a platform where, in the subsequent subsections, we can construct cycles $\mathfrak Z_{n, d} \in H_*((\Xn)^3)$ derived from various virtual cycles of moduli spaces of stable maps to these new spaces. }

Let $Y \to T$ be a smooth family of projective surfaces over a smooth, projective base $T$. 
The relative Hilbert scheme of 
length-$n$ $0$-dimensional closed subschemes is denoted by $Y^{[n]}_T$. It is over $T$
and for any $t\in T$, $Y^{[n]}_T \times_T \{t\}=(Y_t)^{[n]}$.
%the Hilbert scheme of length-$n$ $0$-dimensional closed subschemes in $Y_t=Y\times_T \{t\}$. 
Define its relative fiber product $Y^n_T =Y\times_T\cdots\times_T Y$ ($n$ times), 
and its relative symmetric product $Y^{(n)}_T =Y^n_T/S_n$.
%\footnote{Note that we skip the subscript $T$ usually included indicating the 
%relative nature of these construction. We hope this will not cause any confusion.}

Let $\Lambda$ be a finite set with $|\Lambda| = n$. We define $Y^{[\Lambda]}_T = \Yn$,
$Y^{(\Lambda)}_T = Y^{(n)}_T$, and for accounting purpose, denote
$$
Y^\Lambda_T = \{(x_a)_{a \in \Lambda}|\, x_a \in Y_t\  \text{for some $t\in T$} \}.
$$
Using the Hilbert-Chow morphism $\rho_\Lambda := \rho_n: Y^{[\Lambda]}_T \to 
Y^{(\Lambda)}_T$, we define the Hilbert scheme of $\Lambda$-points to be 
\begin{eqnarray}  \label{Yalpha}
Y^{[\![\Lambda]\!]}_T = Y^{[\Lambda]}_T \times_{Y_T^{(\Lam)}} Y^\Lambda_T.
\end{eqnarray}
{These spaces $Y^{[\![\Lambda]\!]}_T $ can be thought of as Hilbert schemes of ordered points.}

%(Our convention is to use subscript $\rho_\Lambda$ to mean the fiber product over the target 
%of the morphism $\rho_\Lambda$ via the morphism $\rho_\Lambda$.)
Let $\mathcal P_\Lambda$ be the set of partitions or equivalence relations on $\Lambda$.
When $\alpha \in \mathcal P_\Lambda$ consists of ${l}$ equivalence classes 
$\alpha_1, \ldots, \alpha_{l}$, we write $\alpha = (\alpha_1, \ldots, \alpha_{l})$. For such $\alpha$,
we form the {\it relative Hilbert scheme of $\alpha$-points} as follows:
\begin{eqnarray}  \label{alpha-pt}
Y^{(\alpha)}_T = \prod_{i=1}^{l} \! Y^{(\alpha_i)}_T, 
\quad Y^{[\alpha]}_T = \prod_{i=1}^{l} \! Y^{[\alpha_i]}_T,
\quad Y^{[\![\alpha]\!]}_T = \prod_{i=1}^{l} \! Y^{[\![\alpha_i]\!]}_T,
\end{eqnarray}
where the products are taken relative to $T$.
Note that $Y^{[\![\alpha]\!]}_T = Y^{[\alpha]}_T \times_{Y^{(\alpha)}_T} Y^\Lambda_T$.
The ``indexing'' morphism is defined to be the second projection
\beq\label{i}
\ind: Y^{[\![\alpha]\!]}_T\lra Y_T^\Lam.
\eeq

%In case $\Lambda = [n]$, we adopt the convention:
%$$
%Y^\Lambda_T = Y^n_T, \quad Y^{(\Lambda)}_T = Y^{(n)}_T, \quad Y^{[\Lambda]}_T = Y^{[n]}_T, 
%\quad Y^{[\![\Lambda]\!]}_T = Y^{[\![n]\!]}_T.
%$$

The spaces $Y^{[\![\alpha]\!]}_T$ and $Y^{[\![\beta]\!]}_T$ are birational. 
To make this precise, we first fix our convention on a partial ordering on $\cP_\Lambda$.
We agree
$$
``\alpha \ge \beta" \Longleftrightarrow ``a \sim_\beta b \Rightarrow a \sim_\alpha b".
$$
Namely, $\alpha \ge \beta$ if $\beta$ is finer than $\alpha$. 
When $\beta = (\beta_1, \ldots, \beta_r)$, we put 
$$
\alpha \wedge \beta = (\alpha_1 \cap \beta_1, \ldots, \alpha_{l} \cap \beta_r)
$$ 
which is the largest element among all that are less than or 
equal to both $\alpha$ and $\beta$. 
Note that $\mathcal P_\Lambda$ contains a maximal and a minimal element.
The maximal element is $\Lambda$
consisting of a single equivalence class $\Lambda$; the minimal element is 
$1^\Lambda$ whose equivalence classes are single element sets. 
 
For $\alpha > \beta \in \mathcal P_\Lambda$, define 
$$
\Xi^{\alpha}_\beta=\{x \in Y^\Lambda_T\mid \exists\ a, b \in \Lambda\ 
\text{so that}\ \ x_a = x_b,\ a \sim_\alpha b, \ a \not \sim_\beta b \}.
$$
For $\alpha \ne \beta \in \mathcal P_\Lambda$, define $\Xi^{\alpha}_{\beta} = 
\Xi^{\alpha}_{\alpha \wedge \beta} \cup \Xi^{\beta}_{\alpha \wedge \beta}$.
The discrepancy between $Y^{[\![\alpha]\!]}_T$ and $Y^{[\![\beta]\!]}_T$ 
(in  $Y^{[\![\alpha]\!]}_T$) and its complement are defined to be
\beq\label{Complmt}
\Xi_{\beta}^{[\![\alpha]\!]} 
= Y^{[\![\alpha]\!]}_T \times_{Y^\Lambda_T} \Xi^\alpha_{\beta},
\and 
Y_{\beta}^{[\![\alpha]\!]} = Y^{[\![\alpha]\!]}_T - 
\Xi_{\beta}^{[\![\alpha]\!]}. 
\eeq
More precisely, by Lemma~1.2 in \cite{LiJ}, there exists a functorial open embedding
$\zeta_\alpha^\beta: Y_{\beta}^{[\![\alpha]\!]} \to Y_T^{[\![\beta]\!]}$
induced by the universal property of the respective moduli spaces such that 
$\image(\zeta^\beta_\alpha) = Y^{[\![\beta]\!]}_\alpha$.
Thus we obtain an isomorphism (equivalence) 
$\zeta_\alpha^\beta: Y_{\beta}^{[\![\alpha]\!]} \mapright{\cong}
Y_\alpha^{[\![\beta]\!]}$. We define
\begin{eqnarray}  \label{YLeAlpha}
Y^{[\![\le \alpha]\!]}_T = \Bigl( \coprod_{\beta \le \alpha} Y^{[\![\beta]\!]}_T \Bigr) /\sim,
\end{eqnarray}
where the equivalence is by identifying $Y_{\gamma}^{[\![\beta]\!]} \subset Y^{[\![\beta]\!]}_T$
and $Y_{\beta}^{[\![\gamma]\!]} \subset Y^{[\![\gamma]\!]}_T$ via $\zeta_\beta^\gamma$ for 
all $\beta, \gamma \le \alpha$. Note that $Y^{[\![\le \alpha]\!]}_T$ is non-separated 
(except when $\alpha = 1^\Lambda$), and contains the spaces $Y^{[\![\le \beta]\!]}_T$, 
$\beta \le \alpha$, as open subschemes.

{ Even though the non-separated space $Y^{[\![\le \beta]\!]}_T$ comes into the picture, in later subsections, we only perform standard algebraic topological operations on these non-Hausdorff spaces such as pull-backs of cohomology classes and cap products. These operations  are allowed on any topological spaces (see \cite{GH, Iv, Sp}).}

\subsection{Stable maps to Hilbert schemes of ordered points}
\def\Ynd{Y^{[n,d]}_T}

We incorporate stable maps into the above constructions. This is motivated by
the standard decompositions of stable morphisms introduced in \cite{LL}.
%and by the results in \cite{Beh2}. 
For $d\ge 0$, we let 
$$
\Ynd :=\barM_{0, 3}(\Yn, d \beta_n)
$$
be the relative moduli space of $3$-pointed genus-$0$ stable maps to $\Yn$ of class $d\beta_n$.
%Here for notational simplicity, we abbreviate this space to $\Ynd$, as indicated. Note
%that since we will focus on genus-$0$ and $3$-pointed case, we will skip the subscript``0,3''
%in the notation. Also, since $d$ will be reserved to mean the class $d\beta_n$, when $d$ appears in the
%symbol, we mean stable maps of classes $d\beta_n$. We let
%$$\ev_k: \Ynd\lra Y_T^{[n]},\quad [u, C, p_i]\mapsto u(p_k)\in \Yn
%$$
%be its $k$-th evaluation morphism. 

%Thus $Y^{[n,d]}$ contains all the
%essential information for the moduli of three-pointed genus-$0$ stable morphisms to $\Yn$ of 
%classes $d\beta_n$.

%Let $Y\to T$ be a family of surfaces mentioned before. For simplicity, let 
%$\overline {\frak M}^{Y, [n]}_{k, d}$ be the relative moduli space of stable morphisms
%$[\mu: (D; p_1, \ldots, p_k) \to Y^{[n]}]$ with $g(D) = 0$ and $\mu_*[D] = d \beta_n$.
%Since $\rho_n(\mu(D))$ is a point in $Y^{(n)}$, $\overline {\frak M}^{Y, [n]}_{k, d}$ 
%is well defined. 

%We now build the notion of ordered stable maps to the Hilbert scheme of points.
We study the standard decomposition of $[u,C]\in Y^{[n,d]}_T$.
Given $[u,C]\in Y^{[n,d]}_T$, composed with the Hilbert-Chow morphism $\rho_n$, we obtain
$\rho_n \circ  u: C\to Y_T^{(n)}$. Since the fundamental class of $u(C)$ is a multiple of 
the null class $\beta_n$, and $C$ is connected, $\rho_n \circ u$ is a constant map.
We express $\rho_n \circ u(C)=\sum_{i=1}^l n_i x_i$,
where $n_i\in \NN_+$ such that $\sum n_i=n$, and $x_i$ are distinct.
With such data, for $p\in C$, we can decompose
$u(p)=z_1(p)\cup\cdots\cup z_l(p)$
such that $z_i(p)\in Y_T^{[n_i]}$, and $\rho_{n_i}(z_i(p))=n_i x_i$.
Because $x_i$ are distinct, such decomposition is unique. 
We define
\beq\label{ui}
u_i: C\to Y^{[n_i]}_T, \quad u_i(p)=z_i(p).
\eeq
Because of the uniqueness of the decomposition, one checks that $u_i$ are morphisms; 
since $u\lsta[C]=d\beta_n$, we have $u_{i\ast}[C]=d_i\beta_{n_i}$ for some $d_i\ge 0$ 
such that $\sum d_i=d$. Using such data, we can define the Hilbert-Chow map from 
$Y_T^{[n,d]}$ to the weighted symmetric product of $Y$.

For the pair $(n,d)$, we define the weighted symmetric product of $Y$ to be
$$Y_T^{(n,d)}=\bigl\{ \sum_{i=1}^ld_i[n_i x_i]\big |  1\le l\le n,
\ x_1,\cdots, x_l\in Y_t\ \text{distinct, for a $t\in T$}\bigr\}.
$$
Here the formal summation $\sum d_i[n_i x_i]$ is subject to the constraints 
$d_i\in \NN$, $\sum d_i=d$, $n_i\in \NN_+$ and $\sum n_i=n$. Also, 
$[n_ix_i]$ represents the multiplicity-$n_i$ $0$-cycle supported at $x_i$, 
and $d_i$ is its weight. Thus $d_i[n_ix_i]\ne [d_in_ix_i]$ and $0[x_i]$ is non-trivial.
Endow $Y_T^{(n,d)}$ with the obvious topology so that it is a stratified space 
such that the forgetful map $Y_T^{(n,d)}\to Y_T^{(n)}$ is 
continuous, proper and having finite fibers.

We define the Hilbert-Chow map:
\beq\label{HC} 
\hc: Y_T^{[n,d]}\lra Y^{(n,d)}_T,\quad [u]\mapsto \sum_{i=1}^l d_i [n_ix_i],
\eeq
where $(d_i, n_i, x_i)$ are data associated to $(u_i)$ from \eqref{ui}.
Define $\hc_1: Y_T^{[n,d]}\to Y^{(n)}_T$ to be the composite of $\hc$ with the forgetful
map $Y^{(n,d)}_T\to Y_T^{(n)}$. For a finite set $\Lambda$ (of order $n$), define
\begin{eqnarray}  \label{MYn}
Y^{[\![ \Lambda,d]\!]}_T = Y^{[n,d]}_T \times_{Y^{(n)}_T} Y^{\Lambda}_T.
\end{eqnarray}
%(We use subscript $\hc_1$ to mean that it is a fiber-product over $Y^{(n)}_T$ with
%the first arrow $\hc_1:Y^{[n,d]}_T \to Y^{(n)}_T$.)
%We call the projection $\ind: Y^{[\![n,d]\!]}_T\to Y^{\Lam}_T$ the indexing map. 

\begin{definition}   \label{wted-ptn}
We call $(\alpha, \delta)$ {\it a weighted partition} of $\Lambda$ if $\alpha = 
(\alpha_1, \ldots, \alpha_{l}) \in \mathcal P_\Lambda$ and $\delta = (\delta_1, \ldots, 
\delta_{l})$, $\delta_i \ge 0$ for every $i$. We define $\sum_i \delta_i$ to be 
{\it the total weight} of $(\alpha, \delta)$. For $(\Lambda, d)$, we denote by 
$\mathcal P_{\Lambda, d}$ the set of all weighted partitions of $\Lambda$ with
total weight $d$. We say that $(\alpha, \delta) \ge (\beta, \eta)$ if 
$\alpha \ge \beta$ and $\sum_{\beta_i \subset \alpha_j} \eta_i = \delta_j$ for every $j$.
\end{definition}

For $(\alpha, \delta) \in \mathcal P_{\Lambda, d}$, we form the relative moduli space of $3$-pointed
genus-$0$ degree-$\delta$ $\alpha$-stable morphisms to the Hilbert scheme of points:
\beq\label{Y-alpha}
Y^{[\![\alpha,\delta]\!]}_T =Y^{[\![\alpha_1,\delta_1]\!]}_T \times_T\cdots\times_T 
Y^{[\![\alpha_{l},\delta_{l}]\!]}_T.
\eeq
%The product of the indexing map of $Y_T^{[\![\alpha_i,\delta_i]\!]}$ gives us the indexing map
%$\ind: Y^{[\![ \Lambda,d]\!]}_T\to Y^{\Lambda}_T$. 
To simplify notations, the composition of $Y^{[\![ \Lambda,d]\!]}_T \to Y^{[\Lambda,d]}_T$ 
and $\hc_1: Y^{[\Lambda,d]}_T \to Y_T^{(n)}$ will again be denoted by $\hc_1$.
\subsection{Birationality}
{The key result Lemma \ref{lem0.3} provides the comparison between  $\Yad_T$ and $\Ybe_T$, which will be used in later subsections for the comparison of normal cones for  $\Yad_T$ and $\Ybe_T$.}

For $(\alpha,\delta)>(\beta,\eta)$,\footnote{Without further mentioning $\alpha = 
(\alpha_1, \ldots, \alpha_{l})$ and $\beta=(\beta_1,\cdots,\beta_r)$.} 
the pair $\Yad_T$ and $\Ybe_T$ are ``birational''.
To make this more precise, we introduce some notations. Given an element 
$$\xi=([u,C], (y_a))\in Y_T^{[\![\Lambda,d]\!]} 
=Y_T^{[n,d]}\times_{Y^{(n)}_T}Y^\Lam_T,
$$
where $\hc([u])=\sum_{i=1}^l d_i[n_ix_i]$ and such that $\sum n_i x_i=\sum_a y_a$ 
(as $0$-cycles in $Y^{(n)}_T$), we define a pair $(\fa(\xi),\fd(\xi))\in \cP_{\Lam,d}$ by
$$ \fa(\xi)=(\fa_1,\cdots,\fa_l), \ \fa_i=\{a\in\Lam\mid y_a=x_i\}; \ \
\fd(\xi)=(d_1,\cdots, d_l).
$$

\begin{definition}
For $(\beta,\eta)\in\cP_{\Lam,d}$, we define
\begin{eqnarray*}
Y^{[\![\Lam,d]\!]}_{(\beta,\eta)} &=& \bigl\{ \xi\in Y^{[\![\Lam,d]\!]}_T\mid 
   (\fa(\xi),\fd(\xi))\le (\beta,\eta)\bigr\},  \\
Y^{[\![\be]\!]}_{(\Lam,d)} &=& \{(\xi_1,\cdots,\xi_r)\in Y_T^{[\![\be]\!]}\mid \hc_1(\xi_1),
\cdots, \hc_1(\xi_r)\ \text{mutually disjoint}\}.
\end{eqnarray*}
For $(\beta,\eta)\le (\alpha,\delta)$, we define (as fiber products over $T$)
$$
\Yad_{(\beta,\eta)}
=\prod_{i=1}^l Y^{[\![\alpha_i,\delta_i]\!]}_{(\beta\cap \alpha_i,\eta\cap \delta_i)}
\and
Y^{[\![\be]\!]}_{(\ad)}
=\prod_{i=1}^l Y^{[\![\beta\cap \alpha_i,\eta\cap \delta_i]\!]}_{(\alpha_i,\delta_i)}.
$$
%\times_T\cdots\times_T
%Y^{[\![\alpha_{l},\delta_{l}]\!]}_{(\beta\cap \alpha_{l},\eta\cap \delta_{l})}.
%$$
\end{definition}

\begin{lemma}\label{lem0.3}
For $(\ad) >(\beta,\eta)$, we have a natural, proper surjective morphism 
\beq\label{Lma1.1-LiJ-Y}\zeta^{\beta, \eta}_{\alpha, \delta} : 
 \Yad_{(\beta,\eta)} \longrightarrow \Ybe_{(\ad)}.
\eeq
\end{lemma}

\begin{proof}
By definition, we only need to prove the case $(\alpha,\delta)=(\Lam,d)$. 
Let $\xi=([u,C,p_i],(y_a))\in Y^{[\![\Lam,d]\!]}_{(\beta,\eta)}$, 
with $\hc([u])=\sum_{i=1}^l d_i[n_ix_i]$.
Let $u_i: C\to Y_T^{[n_i]}$ be as in \eqref{ui}. 
Denote $\fa(\xi)=(\fa_1,\cdots,\fa_l)$ and $\fd(\xi)=(d_1,\cdots, d_l)$.
Since $\xi\in Y^{[\![\Lam,d]\!]}_{(\beta,\eta)}$, we have $(\fa(\xi),\fd(\xi))
\le (\beta,\eta)$. Thus we can form
$$u_{\beta_i}: C\lra Y^{[\eta_i]}_T;\quad u_{\beta_i}(p)
=\cup_{\fa_j\sub \beta_i} u_j(p)\in Y_T^{[\eta_i]}.
$$
Because the degree of $u_j$ is $d_j$, and $(\fa(\xi), \fd(\xi))\le (\be)$, the degree of 
$u_{\beta_i}$ is $\eta_i$. For $1 \le i \le r$, let
$u_{\beta_i}\st: C_{\beta_i}\lra Y_T^{[\eta_i]}$
be the stabilization of $[u_{\beta_i}, C, p_i]$. Then 
$(u_{\beta_1}\st,\cdots,u_{\beta_{r}}\st)\in \Ybe_T$. 
%\footnote{$\beta=(\beta_1,\cdots,\beta_{r})$ as usual.}
It is routine to check that
$$\zeta_{\Lam,d}^{\beta,\eta}:  Y^{[\![\Lam,d]\!]}_{(\beta,\eta)} \longrightarrow \Ybe_T;
\quad ([u,C],(y_a)_\Lam)\mapsto (u_{\beta_1}\st,\cdots,u_{\beta_{r}}\st),
$$
defines a morphism. By the definition of $Y^{[\![\be]\!]}_{(\Lam,d)}$,
we have $\im(\zeta^{\be}_{\Lam,d}) \sub Y^{[\![\be]\!]}_{(\Lam,d)}$.

We now show that $\im(\zeta^{\be}_{\Lam,d}) = Y^{[\![\be]\!]}_{(\Lam,d)}$. Note that 
a closed point in $Y^{[\![\be]\!]}_{(\Lam,d)}$ is an $r$-tuple $(\xi_1,\cdots,\xi_r)$ with
$\xi_i\in Y^{[\![\beta_i,\eta_i]\!]}_T$ such that $\hc_1(\xi_1),\cdots\hc_1(\xi_r)$ are
mutually disjoint. Let $\xi_i=[u_i,C_i,p_{i,j}]$. 
Since $[C_i,p_{i,j}]$ are 3-pointed genus-$0$ nodal curves,
we can find a 3-pointed genus-$0$ $[C,p_j]$ and contraction morphisms $\phi_i: C\to C_i$ so
that $\phi_i(p_j)=p_{i,j}$, $j=1,2,3$. Since $\hc_1(\xi_1),\cdots\hc_1(\xi_r)$ are mutually disjoint, 
the assignment $p \mapsto u(p)=u_1\circ\phi_1(p)\cup\cdots\cup u_r\circ\phi_r(p)\in Y^{[n]}_T$
defines a morphism $u: C\to Y^{[n]}_T$. We let
$\xi=[u,C,p_j]\st$ be its stabilization. Then $\xi\in Y^{[\![\Lam,d]\!]}_{(\be)}$,
and $\zeta^{\be}_{\Lam,d}(\xi)=(\xi_1,\cdots,\xi_r)$. 
Hence $\im(\zeta^{\be}_{\Lam,d}) = Y^{[\![\be]\!]}_{(\Lam,d)}$.
%This proves that $\zeta^{\be}_{\Lam,d}$ is surjective.

We check that $\zeta^{\be}_{\Lam,d}$ is proper. Let $s_0\in S$ be a pointed smooth curve
over $T$; let $S\sta=S-s_0$. Suppose $\xi\sta$ is an $S\sta$-family in $Y^{[\![\Lam,d]\!]}_{(\beta,\eta)}$
so that $\zeta^{\be}_{\Lam,d}(\xi\sta)=(\xi_1\sta,\cdots,\xi_r\sta)$ extends to
an $S$-family $(\xi_1,\cdots,\xi_r)$, we need to show that, possibly after a base change, 
$\xi\sta$ extends to $\xi$ so that $\zeta^{\be}_{\AD}(\xi)=(\xi_1,\cdots,\xi_r)$.

Since $Y^{[\![\Lam,d]\!]}_T$ is $T$-proper, 
possibly after a base change, we can extend $\xi\sta$ to an $S$-family
$\xi$ in $Y^{[\![\Lam,d]\!]}_T$. 
Let $\xi$ be given by $([u,C,p_j], (y_a))$, where each term implicitly is an $S$-family.
Let $y_{\beta_i}=\sum_{a\in \beta_i} y_a: S\to Y^{(\beta_i)}_T$. 
By definition, $\xi(s_0)=\xi\times_S \{s_0\}\in Y^{[\![\Lam,d]\!]}_{(\be)}$
if $y_{\beta_1}(s_0),\cdots,y_{\beta_r}(s_0)$ are mutually disjoint. Since 
$\zeta^{\be}_{\Lam,d}(\xi\sta)=(\xi_1\sta,\cdots,\xi_r\sta)$, we have
$y_{\beta_i}|_{S\sta}=\hc_1\circ\xi_i\sta$.  Since $Y^{(n)}_T$ is separated, we have
$y_{\beta_i}(s_0)=\hc_1(\xi_i(s_0))$.
Further, since $(\xi_1(s_0),\cdots,\xi_r(s_0))\in Y^{[\![\be]\!]}_{(\Lam,d)}$, 
$\hc_1(\xi_1(s_0)),\cdots,\hc_1(\xi_r(s_0))$ 
are mutually disjoint. This proves that $\xi(s_0)\in Y^{[\![\Lam,d]\!]}_{(\be)}$.
Then $\xi$ lies in $Y^{[\![\AD]\!]}_{(\be)}$, and by the separatedness of 
$Y^{[\![\Lam,d]\!]}_T$, we have $\zeta^{\be}_{\Lam,d}(\xi)=(\xi_1,\cdots,\xi_r)$. 
This proves the properness.
\end{proof}

The morphism $\zeta^{\beta, \eta}_{\alpha, \delta}$ fits into a fiber diagram that will
be crucial for our virtual cycle comparison. As we only need the case where $(\beta,\eta) < (\ad)$ is
derived by a single splitting, meaning that $r=l+1$, we will state it in the case
$(\alpha,\delta)=(\Lam, d)$, and $(\beta,\eta)=((\beta_1,\beta_2),(d_1,d_2))$.

We first introduce necessary notation, following Behrend \cite{Beh2}. 
Given a semi-group $G=\NN$ or $\NN^{2}$, we call
a triple $(C,p_j,\tau)$ a pointed $G$-weighted nodal curve if $(C,p_i)$ is a pointed
nodal curve and $\tau$ is a map from the set of irreducible components of $C$ to $G$.
We say $(C,p_j,\tau)$ is stable if for any $C_0\cong \Po\sub C$, either $\tau([C_0])\ne 0$ or
$C_0$ contains at least three special points of $(C,p_j)$. (A special point of $(C,p_j)$ is either a node
or a marked point.) 

We denote by $\cM_{0,3}(d)$ the Artin stack of total weights $d$ $\NN$-weighted
$3$-pointed genus-$0$ nodal curves.
We denote by $\fD(d_1,d_2)$ the Artin stack of the data 
$$\bigl\{ (C,p_j,\tau)\to (C_1,p_{1,j},\tau_1),(C,p_j,\tau)\to (C_2,p_{2,j},\tau_2)\bigr\}
$$
so that $(C,p_j,\tau)$ is a stable total weight $(d_1,d_2)$ $\NN^{2}$-weighted
$3$-pointed genus-$0$ nodal curve, $(C_i,p_{i,j},\tau_i)\in \cM_{0,3}(d_i)$, and the two arrows
induce isomorphisms $(C,p_j,\pr_i\circ\tau)\st\cong (C_i,p_{i,j},\tau_i)$, where
$\pr_i: \NN^{2}\to\NN$ is the $i$-th projection. (See the diagram (3) in \cite{Beh2} for details.)

\begin{lemma}         \label{0.4}
Let $\beta=(\beta_1,\beta_2)$ be a partition of length two,
and let $\eta=(d_1,d_2)$ with $d=d_1+d_2$.
We have a Cartesian diagram
$$\begin{CD}
Y^{[\![\Lam,d]\!]}_{(\be)} @>>> Y^{[\![\beta, \eta]\!]}_{(\lam,d)}\\
@VVV @VVV\\
\fD(d_1,d_2) @>{(\epsilon_1,\epsilon_2)}>> \cM_{0,3}(d_1)\times\cM_{0,3}(d_2)
\end{CD}
$$
Further, $(\epsilon_1,\epsilon_2)$ is proper and birational.
\end{lemma}
\begin{proof}
The proof is a direct application of Proposition~5 in \cite{Beh2} plus the definition of 
$Y^{[\![\Lam,d]\!]}_{(\be)}$. Note that the second vertical arrow is induced by 
$Y^{[\![\beta, \eta]\!]}_{(\lam,d)}\sub 
Y^{[\![\beta_1, d_1]\!]}_T\times_T Y^{[\![\beta_2,d_2]\!]}_T$
and the forgetful morphism $Y^{[\![\beta_i, d_i]\!]}_T\to \cM_{0,3}(d_i)$.
\end{proof}
\subsection{Virtual classes and comparison of normal cones}

As $Y_T^n\to Y_T^{(n)}$ is a finite quotient map by a finite group, it is flat.
Let $[Y^{[\alpha,\delta]}_T]\virt$ be the virtual class of $Y^{[\alpha,\delta]}_T$.
We define $[Y^{[\![\alpha,\delta]\!]}_T]\virt$ to be the flat pullback of $[Y^{[\alpha,\delta]}_T]\virt$.
%We let
%\beq\label{gad}
%g_\adbracket: Y^{[\![\alpha,\delta]\!]}_T\lra Y^{[\![\le\Lambda, d]\!]}_T
%\eeq
%be the tautological inclusion. 
Our goal is to inductively construct cycle representatives
of the virtual classes of $Y_T^\adbracket$ that are compatible via the comparison $\zeta_{\ad}^{\be}$.

We recall the construction of virtual cycles in \cite{BF, LT1}. 
Let $(\EE_{[\ad]})\dual\to \LL_{Y^{[\ad]}_T/T\times (\cM_{0,3})^l}$ 
be the standard perfect relative obstruction theory\footnote{Here $\EE_{[\ad]}$ is a derived object locally 
presented as a two-term complex of locally free sheaves placed at $[0,1]$.} of 
$Y^{[\ad]}_T\to T\times (\cM_{0,3})^l$; let 
$\mathbf C_{[\ad]}\sub \bF_{[\ad]}:= h^1/h^0(\EE_{[\ad]})$ be its intrinsic normal cone.
To use analytic Gysin map, we put it in a vector bundle. 
Following \cite{BF, LT1}, we can find a vector bundle (locally free sheaf) $E_{[\ad]}$ on $Y^{[\ad]}_T$
and a surjection of bundle-stack $E_{[\ad]}\to h^1/h^0(\EE_{[\ad]})$. Let 
$C_{[\ad]}\sub E_{[\ad]}$ be the flat pullback of $\mathbf C_{[\ad]}$. Then $[Y^{[\ad]}_T]\virt=0_{E_{[\ad]}}^! [C_{[\ad]}]$,
the image of the Gysin map of the zero-section of $E_{[\ad]}$. 
Let $\rho_{\ad}: Y^{[\![\ad]\!]}_T\lra Y^{[\ad]}_T$
be the tautological projection, $E_{[\![\ad]\!]}=\rho_{\ad}\sta E_{[\ad]}$, and 
$C_{[\![\ad]\!]}\sub E_{[\![\ad]\!]}$
be the flat pullback of $C_{[\ad]}$ via $E_{[\![\ad]\!]}\to E_{[\ad]}$.
The {\it virtual class} of $Y_T^{[\![\alpha, \delta]\!]}$ is equal  to
\begin{eqnarray}  \label{Def4.7-LiJ}
\big [Y_T^{[\![\alpha, \delta]\!]} \big ]^{\rm vir}
=(\rho_{\ad})\sta[Y^{[\ad]}_T]\virt
= 0_{E_{[\![\alpha,\delta]\!]}}^*[C_{[\![\alpha,\delta]\!]}]  \in H_*(|Y_T^{[\![\alpha, \delta]\!]}|; \mathbb Q),
\end{eqnarray}
where $0_{E_{[\![\alpha,\delta]\!]}}^*$ is the Gysin homomorphism of the zero section of 
$E_{[\![\alpha,\delta]\!]}$, and $|Y^{[\![\ad]\!]}_T|$ is the coarse moduli space of $Y^{[\![\ad]\!]}_T$.
Also, put $\EE_{[\![\ad]\!]}=\rho_{\ad}\sta \EE_{[\ad]}$, and let
$\bF_{[\![\ad]\!]} =h^1/h^0(\EE_{[\![\ad]\!]})=\rho_{\ad}\sta\bF_{[\ad]}$ be the flat pullback. 
Let $\bC_{[\![\ad]\!]}\sub \bF_{[\![\ad]\!]}$
be the flat pullback of $\bC_{[\ad]}$ via $\bF_{[\![\ad]\!]}\to \bF_{[\ad]}$.

We now compare the cycles $C_{[\![\ad]\!]}$ using $\zeta_{\ad}^{\be}$. 
The tricky part is that the vector bundles $E_{[\![\ad]\!]}$ are not comparable. Thus we will state the
comparison using cycles in $\bF_{[\![\ad]\!]}$, and later will use the obstruction
sheaf for accounting purpose.

\begin{lemma}    \label{lem0.5}
For pairs $(\ad)> (\be)$, we have canonical isomorphisms
$$\varphi_\ad^\be: (\zeta_\ad^\be)\sta (\bF_{[\![\be]\!]}|_{Y^{[\![\be]\!]}_{(\ad)}}) 
\mapright{\cong} \bF_{[\![\ad]\!]}|_{Y^{[\![\ad]\!]}_{(\be)}}
$$
that satisfy the cocycle condition: we have
$\varphi_\ad^\be\circ (\zeta_\ad^\be)\sta(\varphi_\be^{\gamma,\varepsilon}) 
=\varphi_\ad^{\gamma,\varepsilon}$ for any triple $(\ad)> (\be)>(\gamma,\varepsilon)$. 
Further, let $\bar\varphi_\ad^\be: \bF_{[\![\ad]\!]}|_{Y^{[\![\ad]\!]}_{(\be)}} \lra 
\bF_{[\![\be]\!]}|_{Y^{[\![\be]\!]}_{(\ad)}}$
be the projection induced by $\varphi_\ad^\be$, which is proper by Lemma \ref{lem0.3}. Then
$$(\bar\varphi_\ad^\be)_\ast[\bC_{[\![\ad]\!]}|_{Y^{[\![\ad]\!]}_{(\be)}}]
=[\bC_{[\![\be]\!]}|_{Y^{[\![\be]\!]}_{(\ad)}}].
$$
\end{lemma}

\begin{proof}
By induction, we only need to prove the case where $\ell(\beta)=\ell(\alpha)+1$;
by definition this follows from the case $(\ad)=(\lam,d)$ and $\beta=(\beta_1,\beta_2)$ with $\eta=(d_1,d_2)$,
which we suppose in the remainder of this proof.
%Since the pair $C_{[\![\AD]\!]}\sub \bF_{[\![\AD]\!]}$ is the flat pullback of $C_{[\AD]}\sub \bF_{[\AD]}$. 
%we only need to prove the corresponding results for the later. 

Let $y=(y_a)\in Y^\lam_T$ be a closed point so that $y_{\beta_1}= \rho_{\beta_1}((y_a)_{a\in\beta_1})
\in Y^{(\beta_1)}_T$ and $y_{\beta_2}\in Y^{(\beta_2)}_T$ (defined similarly) are disjoint.
We then form
$$V_i=Y_T^{[\beta_i]}\times_{Y_T^{(\beta_i)}}\{y_{\beta_i}\}\and
V=Y_T^{[\Lam]}\times_{Y_T^{(\Lam)}}\{\rho_\Lambda(y)\}.
$$
Note that $y_{\beta_1}\cap y_{\beta_2}=\emptyset$ implies that $V_1\times_T V_2 \sub Y^{[\beta]}_T$.
%(respectively, $V\sub Y^{[\Lam]}_\beta$). 
Also, there exists a canonical isomorphism $\zeta^\Lam_\beta: V_1\times_T V_2 \to V$.
Let $\hat V_i$ (respectively, $\hat V$) be 
the formal completion of $Y^{[\beta_i]}_T$ (respectively, $Y^{[\Lambda]}_T$)
along $V_i$ (respectively, $V$). 
%Since $\zeta^\Lam_\beta(V_1\times_T V_2)=V$, $\zeta^\Lam_\beta: Y^{[\beta]}_\Lam \to %Y^{[\Lam]}_\beta$ induces a morphism
The isomorphism $\zeta^\Lam_\beta$ induces 
$$\hat\zeta^\Lam_\beta: \hat V_1\times_T\hat V_2\lra \hat V,
$$
which is injective and smooth.

For notational simplicity, we denote $\barM(\hat V_i)=\barM_{0,3}(\hat V_i,d_i)$ with
$\iota_2$ in \eqref{lem0.5.1} being the tautological morphism induced by 
$\hat V_i\to Y^{[\beta_i]}_T$; we let 
$$
\barM(\hat V_1\times_T\hat V_2)=\barM_{0,3}(\hat V_1\times_T\hat V_2, (d_1,d_2))
$$
with $\iota_1$ in \eqref{lem0.5.1} being the tautological morphism induced by 
$\hat V_1\times_T\hat V_2\to Y^{[\Lam]}_T$.

We consider the following commutative diagram of arrows, where $\phi$ is defined 
by sending $[u,C,p_j]\in \barM(\hat V_1\times_T\hat V_2)$ to $(\xi_1,\xi_2)$ with
$\xi_i=[\pi_i\circ u, C, p_j]\st$ for $\pi_i: \hat V_1\times_T\hat V_2\to\hat V_i$ the projection;
$\phi'$ is induced by $\phi$.
\beq    \label{lem0.5.1}
\begin{CD}
Y^{[\![\AD]\!]}_{(\be)} @>{\zeta^{\be}_{\AD}}>> Y^{[\![\be]\!]}_{(\AD)} \\
@AA{\varphi_1}A @AA{\varphi_2}A\\
\barM(\hat V_1\times_T\hat V_2)\times_{Y^{(\beta)}_T}Y^\beta_T @>{\phi'}>>
\bigl(\barM(\hat V_1)\times_T\barM(\hat V_2)\bigr) \times_{Y^{(\beta)}_T}Y^\beta_T \\
@VV{\psi_1}V @VV{\psi_2}V\\
\barM(\hat V_1\times_T\hat V_2) @>\phi>> \barM(\hat V_1)\times_T\barM(\hat V_2)\\
@VV{\iota_1}V@VV{\iota_2}V\\
Y^{[\AD]}_T @. Y^{[\be]}_T
\end{CD}
\eeq

We let $\bC_{1}\sub \bF_{1}$ be the intrinsic normal cone in the bundle stack of the obstruction complex
of the prefect relative obstruction theory of $\barM(\hat V_1\times_T\hat V_2)\to T\times\cM_{0,3}$.
Because $\hat V_1\times_T\hat V_2\to Y^{[\AD]}_T$ is injective and smooth, we have
$\iota_1\sta(\bC_{[\AD]}\sub \bF_{[\AD]}) = (\bC_{1}\sub \bF_{1})$. 
Since $\bC_{[\![\AD]\!]}\sub \bF_{[\![\AD]\!]}$ is the 
pullback of $\bC_{[\AD]}\sub \bF_{[\AD]}$, we conclude
$\varphi_1\sta(\bC_{[\![\AD]\!]}\sub \bF_{[\![\AD]\!]})=\psi_1\sta(\bC_1\sub\bF_1)$.

Similarly, letting $\bC_{2}\sub \bF_{2}$ be the intrinsic normal cone in the bundle stack of 
the obstruction complex of the prefect relative obstruction theory of 
$\barM(\hat V_1)\times_T\barM(\hat V_2)\to T\times (\cM_{0,3})^2$,
we have $\varphi_2\sta(\bC_{[\![\be]\!]}\sub \bF_{[\![\be]\!]})=\psi_2\sta(\bC_2\sub\bF_2)$.
Since $\varphi_1$ and $\varphi_2$ are injective and smooth, since $\phi'$ is proper, since the top square 
is commutative, and since the image of $\varphi_1$ (respectively, of $\varphi_2$) covers 
$Y^{[\![\AD]\!]}_{(\be)}$ (respectively, $Y^{[\![\be]\!]}_{(\AD)}$) for $y$ varying through 
$Y^\Lambda_T$ satisfying $y_{\beta_1}\cap y_{\beta_2}=\emptyset$,
to prove that $\bF_{[\![\AD]\!]}=(\zeta_{\AD}^{\be})\sta\bF_{[\![\be]\!]}$ and $(\zeta_{\AD}^\be)_\ast[\bC_{[\![\AD]\!]}] = [\bC_{[\![\be]\!]}]$,
it suffices to show that we have the canonical isomorphism and identity
\beq\label{need}
\bF_1\cong \phi\sta\bF_2\and \ti\phi_\ast[\bC_1]=[\bC_2],
\eeq
where $\ti\phi: \bF_1\to \bF_2$ is the induced projection. 
But this follows from the Cartesian square
$$\begin{CD}
\barM(\hat V_1\times_T\hat V_2) @>\phi>> \barM(\hat V_1)\times_T\barM(\hat V_2)\\
@VVV@VVV\\
T\times\frak D(d_1,d_2) @>>> T\times \cM_{0,3}(d_1)\times\cM_{0,3}(d_2),
\end{CD}
$$
similar to the one stated in Lemma \ref{0.4} (originally constructed in Proposition~5 of \cite{Beh2}).
Since the lower horizontal line is birational, and $T$ is smooth and projective,
by Theorem~5.0.1 in \cite{Cos}, we have the isomorphism and identities in \eqref{need}. 
This proves the lemma.
\end{proof}

\subsection{Multi-sections and pseudo-cycle representatives}
\def\lalp{_\alpha}
\def\lbe{_\beta}

In this subsection, we use multi-sections to intersect the cycles $C_{[\![\alpha,\delta]\!]}$ to obtain 
pseudo-cycle representatives of $[Y_T^{[\![\ad]\!]}]\virt$.

In the remainder of this section, we will work with analytic topology and smooth ($C^\infty$) sections.
Let $V$ be a vector bundle over a DM stack $W$.
In case $W$ is singular, we stratify $W$ into 
a union of smooth locally closed DM stacks $W=\coprod W_\alpha$,
and use continuous sections that are smooth when restricted to each stratum $W\lalp$.
Without further commenting, all sections used in this section are
stratified sections; we denote the space of such sections by
$\cC(W,V)$. Also, we will use $|W|$ and $|V|$ to denote the coarse moduli of $W$ and $V$. 

We recall the notion of multi-sections, 
% of a vector bundle $V$ over a DM-stack $W$,
following \cite{FO, LT2}. 
We first consider the case where $W=U/G$ is a quotient stack and 
%$V$ a vector bundle on $W$, which is equivalent to
$V$ is a $G$-vector bundle on $U$. 
Let $S^n(V)\to U$ be the $n$-th symmetric product bundle of $V$.
A {\it liftable multi-section} $s$ of $V$ (of multiplicity $n$) is 
a $G$-equivariant section $s\in \cC(U, S^n(V))^G$  such that there are 
$n$ sections $s_1,\cdots, s_n\in \cC(U,V)$ so that $s$ is the image of 
$(s_1,\cdots, s_n)$. For a multi-section $s\in \cC(U, S^n(V))^G$ that is 
the image of $(s_1,\cdots, s_n)$, we define its integer multiple 
$ms\in \cC(U, S^{mn}(V))^G$ be the image of $(s_1,\cdots, s_1,\cdots, 
s_n,\cdots, s_n)$, where each $s_i$ is repeated $m$ times. 
Given two multi-sections $s$ and $s'$ of multiplicities $n$ and $n'$, 
we say that $s$ and $s'$ are equivalent, denoted by $s\approx s'$,
if $n's=ns'$ as multi-sections. 

In general, since $W$ is a DM-stack,
it can be covered by (analytic) open quotient stacks $U\lalp/G\lalp\sub W$, 
and the restriction $V|_{U\lalp/G\lalp}=V\lalp/G\lalp$
for $G\lalp$-vector bundles $V\lalp$ on $U\lalp$. 
%We comment that since in the end we will work
%with homology of $W$ with $\QQ$-coefficient, which is equivalent to $H_*(|W|,\QQ)$, we will not
%be strict in distinguishing the quotient $U\lalp/G\lalp$ as a quotient stack or as a quotient coarse moduli space.
A {\it multi-section} $s$ of $V$ consists of an analytic open 
covering $U\lalp/G\lalp$ of $W$ and a collection of liftable multi-sections $s\lalp$ of $V|_{U\lalp/G\lalp}$ so
that for any pair $(\alpha,\beta)$, the pullbacks of $s\lalp$ and $s\lbe$ to $U\lalp\times_W U\lbe$
are equivalent. We denote the space of multi-sections of $V$ by $\cC_{\mathrm{mu}}(W,V)$.
(Thus multi-sections in this paper are always locally liftable.)

The space of multi-sections of $V$ has the same extension property as the space of
sections of a vector bundle on a manifold. The usual extension property of vector
bundles on manifolds is proved by using the partition of unity and the addition structure of the vector
bundles. For multi-sections, over a chart $U\lalp/G\lalp$, we define the sum of two (liftable) multi-sections 
$s$ and $s'$ (with lifting $(s_i)_{i=1}^n$ and $(s'_j)_{j=1}^m$, respectively) be the multiplicity $nm$ multi-section that is the image of
$s+s'=(s_i+s_j')$. This local sum extends to sum of two 
multi-sections on $W$. Thus combined with the partition of unity of $|W|$, we conclude that
the mentioned extension property holds for $\cC_{\mathrm{mu}}(W, V)$.

We also have the following transversality property.  Given a closed integral substack $C\sub V$ and
a multi-section $s\in \cC_{\mathrm{mu}}(W, V)$, we say that $s$ intersects $C$ transversally if
there is a stratification of $C$ so that each strata $C_{\alpha}$ of $C$ lies over a strata of $W$, say $W_{\alpha'}$,
and the section $s|_{W_{\alpha'}}$ intersects $C_{\alpha}$ transversally,
meaning that the local liftings of $s|_{W_{\alpha'}}$ intersect $C\lalp$ transversally.
Given a cycle $[C]=\sum n_i[C_i]$ with $C_i$ closed integral algebraic substacks, 
we say $s$ intersects $[C]$ transversally if it intersects each $C_i$ transversally.

\begin{lemma}\label{rat-gen}
Let $p: W'\to W$ be a proper morphism of DM-stacks; let $V$ be a vector bundle on $W$ and
$\tilde p: p\sta V\to V$ be the induced projection.
Suppose $[C']\in Z\lsta( p\sta V)$ is an algebraic cycle and $[C]=\ti p\lsta [C']$. 
If $s\in \cC_{\mathrm{mu}}(W,V)$ intersects $[C]$ transversally, then
$p\sta s\in \cC_{\mathrm{mu}}(W',p\sta V)$ intersects $[C']$ transversally.
\end{lemma}
\begin{proof}
We pick stratifications $W=\coprod W\lalp$ and $W'=\coprod W\lalp'$ so that $p(W'\lalp)=W\lalp$
and $p\lalp=p|_{W'\lalp}: W\lalp'\to W\lalp$ are smooth. We then pick a stratification 
$C'=\coprod C\lbe'$ so that each $C\lbe'$ lies over a stratum of $W'$, 
and that $\ti p|_{C\lbe'}: C\lbe'\to \ti p(C\lbe')$ is smooth.
Therefore, by the definition of transversal to $C$, we are reduced to check when $p: W'\to W$
and $C' \to \ti p(C')$ are smooth. In this case, the statement of the lemma holds by direct local
coordinate checking. This proves the lemma.
\end{proof}

We now construct pseudo-cycle representatives of the topological Gysin map
\beq\label{Gysin}
0^!_V: Z\lsta V\lra H_\ast(|W|,\QQ),
\eeq
via intersecting with multi-sections \cite{FO, LT2, LT3, McD, Zin}.

We assume $W$ is proper. Let $\pi :V\to W$ and $\bar\pi: |V|\to |W|$ be the projections.
Given a closed integral algebraic substack $C\sub V$, 
%that is a union of closed integral substacks $C_i$ with multiplicities $n_i$,
we find a multi-section $s$ of $V$ so that it intersects
$C$ transversally. Let $k=2(\text{rank } V-\dim C)$.
%Let $R\sub |V|$ be the singular locus of $|C|$. 
%Let $R\sub V$ be the singular locus of $|V|$ and $U$ be a small neighborhood
%of $\bar\pi(R)$ that deformation retracts to $R$. 
By slightly perturbing $s$ if necessary, we can assume that there is 
a closed (stratifiable) subset $R\sub |V|$ of $\dim_\RR R\le k-2$ 
% and a closed neighborhood $\ti R$ of $R\sub |V|$
and an (analytic) open covering of $W$ by
quotient stacks $U\lalp/G\lalp$ so that,  letting $q\lalp: V\lalp\to |V|$ be the projections, 
\begin{enumerate}
\item
$s|_{U\lalp/G\lalp}$ are images of $s_{\alpha,1},\cdots,s_{\alpha,m\lalp}$ in $\cC(U\lalp,V\lalp)$;

\item
there are topological spaces $S_{\alpha, i}$ and proper embeddings 
%smooth embeddings of smooth manifolds with boundaries 
$f_{\alpha,i}: S_{\alpha,i}\to V\lalp$ such that
\begin{enumerate}
\item 
there are dense open subsets $S_{\alpha, i}^\circ\sub S_{\alpha, i}$ so that $S_{\alpha, i}^\circ$ are 
smooth manifolds and $f_{\alpha,i}|_{S_{\alpha,i}^\circ}: S_{\alpha, i}^\circ\to V\lalp$ 
are smooth embeddings;
\item
$s_{\alpha,i}\cap(C\times_V V\lalp-q\lalp\upmo(R))= f_{\alpha, i}(S_{\alpha, i}^\circ)$;
\item 
$f_{\alpha,i}(S_{\alpha,i}-S_{\alpha,i}^\circ)\sub q\lalp\upmo(R)$.
\end{enumerate}
\end{enumerate}

Since $s\in \cC_{\mathrm{mu}}(W,V)$, by definition, 
$\sum_{i=1}^{m\lalp}f_{\alpha,i}(S_{\alpha,i}^\circ)$ is $G\lalp$-equivariant. Define
\beq\label{currents}
(C\cap s)|_{|V\lalp|}=\frac{1}{m\lalp}\bigl(\sum_{i=1}^{m\lalp}f_{\alpha,i}(S_{\alpha,i})\bigr)/ G\lalp, 
\eeq
viewed as a sum of piecewise smooth $k$-dimensional $\QQ$-currents away from a $(k-2)$-dimensional subset.
Since $(s_{\alpha,i})$ are local lifts of a global multi-section $s$, the $\QQ$-currents \eqref{currents}
patch to form a piecewise smooth $\QQ$-currents with vanishing boundary in $|V|-R$.
We denote this current by $C\cap s$. Since $|W|$ is compact, the current $C\cap s$ defines a homology class in
$H_k(|V|,R;\QQ)=H_k(|V|;\QQ)$. Applying the projection $\bar\pi: |V|\to |W|$, we obtain 
the image $\QQ$-current $\bar\pi(C\cap s)$ and its associated homology class
$[\bar \pi(C\cap s)]\in H_k(|W|;\QQ).$
Following the topological construction of Gysin map of intersecting with the zero-section of $V$,
$$0^!_V[C]=[\bar\pi(C\cap s)]\in H_*(|W|;\QQ)
$$ 
is the image of $[C]$ under the topological Gysin map $0_V^!$.
By the linearity of Gysin map, this defines the topological $0_V^!$ in \eqref{Gysin}.
The current $\bar\pi(C\cap s)$ is called a {\it pseudo-cycle representative} of the Gysin map.

%Let $\pi: V\to W$ be the projection. For the integral $C\sub V$ as stated, and for sufficiently
%divisible $m$, we can find an $m$-section $f$ of $V$ so that it intersects $C$ transversally. Then the
%intersection $f\cap C$ forms a close (w/o boundary) pseudo-cycle; we denote the image
%pseodu-cycle by $D(f)=\pi(f\cap C)$, (with $C$ implicitly understood).
%It is a closed pseudo-cycle, thus representing a homology class in
%(the Borel-Moore) $H_*(W,\QQ)$. By its construction and the 
%topological construction of Gysin maps,
%$$\frac{1}{m} [D(f)]=0_V^![C]\in H_*(W,\QQ).
%$$

We now assume in addition that $\sF$ is a quotient sheaf 
$\phi: \sO_W(V)\to\sF$, and the cycle $[C]=\sum n_i[C_i]\in Z\lsta W$ has the property
\begin{enumerate}
\item[{\rm (P)}] 
{\sl for each $C_i$, and any closed $z\in W$ and $a\in \sF|_z$, letting 
$\phi_z: V_z\to \sF|_z$ be $\phi$ restricting to $z$, we have either $\phi_z\upmo(a)\cap C_i=\emptyset$ or
$\phi_z\upmo(a)\cap C_i=\phi_z\upmo(a)$. \footnote{As argued in \cite{CL}, 
this means that $C$ is a pull back of a ``substack" of $\sF$.}
 }
\end{enumerate}

%Note that $C_{[\![\alpha,\delta]\!]}\sub E_{[\![\alpha,\delta]\!]}$ satisfies the property (P).
%We fix a stratification $W=\coprod W_\alpha$ so that in addition it 
%satisfies that for all $\alpha$, $\sF|_{W_\alpha}$ are locally free sheaves of $\sO_{W_\alpha}$-modules. 

\begin{definition}
Two multi-sections $s$ and $s'$ of $V$ are $\sF$-equivalent, denoted by $s \sim_\sF s'$, 
if for any $x\in W$, as $\QQ$-zero-cycles, we have
$(\phi_x)\lsta(s(x))=(\phi_x)\lsta(s'(x))$. A multi-section of $\sF$ is 
an $\sim_\sF$ equivalence class of multi-sections of $V$. We say a multi-section $\bold s$ 
of $\sF$ intersects $C\sub V$ transversally if 
a representative $s$ of $\bold s$ intersects $C$ transversally.
\end{definition}

We comment that when $C$ satisfies property (P), the notion that a multi-section 
of $\sF$ intersects $C$ transversally is well-defined, 
after we pick the stratification of $W$ so that $\sF$ restricts to each stratum is
locally free, which we always assume in the remaining discussion.

%By the discussion in \cite{CL}, we can make sense of the image of $C\sub V$
%under $V\to\sF$. We denote the image\footnote{The groupoid of sections of $\sF$
%forms a stack (not necessarily algebraic stack), called the sheaf stack of $\sF$. 
%The image of $C$ in $\sF$ is the image substack of this sheaf-stack.}
%by $\sC\sub \sF$
%According to the prior definition, we can speak of $m$-section $f$ of $\sF$ that intersects
%$\sC\sub \sF$ transversally; if $f$ does intersects $\sC$ transversally, then the
%image of the intersection
%$\frac{1}{m}\cdot \pi(\sC\cap f)\sub W$ forms a closed pseudo-cycle in $W$, thus defines a homology class in
%(the Borel-Moore) homology group $H_*(W,\QQ)$.

We apply this discussion to $C_{[\![\alpha,\delta]\!]}\sub E_{[\![\alpha,\delta]\!]}$.
%We assume $T$ is smooth and compact.
Let $\sF_{[\ad]}=H^1(\EE_{[\ad]})$, a coherent sheaf on $Y^{[\ad]}_T$,
and let $\sF_{[\![\ad]\!]}=\rho_\ad\sta \sF_{[\ad]}$, the pullback sheaf on $Y^{[\![\ad]\!]}_T$.
(Note that $\sF_{[\ad]}$ is the obstruction sheaf of the relative obstruction theory of $Y^{[\ad]}_T$.)
Then $\sF_{[\![\ad]\!]}$ is the quotient sheaf of $E_{[\![\ad]\!]}$ via
$$\phi_{[\![\ad]\!]}: E_{[\![\ad]\!]}\lra \bF_{[\![\ad]\!]}
=h^1/h^0(\rho_\ad\sta \EE_{[\ad]})\lra H^1(\rho_\ad\sta \EE_{[\![\ad]\!]})
=\sF_{[\![\ad]\!]}.
$$
%By the discussion in \cite{CL}, we can make sense of the image of $C_{[\![\ad]\!]}\sub E_{[\![\ad]\!]}$
%under $\phi_{[\![\ad]\!]}$. We denote the image\footnote{Given the coherent sheaf $\sF_{[\![\ad]\!]}$, 
%the groupoid of its sections
%form a stack (not necessarily algebraic stack), called the sheaf stack of $\sF_{[\![\ad]\!]}$. 
%The image of $C_{[\![\ad]\!]}$ in $\sF_{[\![\ad]\!]}$ is the image substack of this sheaf-stack.}
%by $\sC_{[\![\ad]\!]}\sub \sF_{[\![\ad]\!]}$. 
Since $C_{[\![\ad]\!]}$ is the pullback of the cycle $\bC_{[\![\ad]\!]}$
in $\bF_{[\![\ad]\!]}$, the cycle $C_{[\![\ad]\!]}$ satisfies property (P) for the pair
$E_{[\![\ad]\!]}\to\sF_{[\![\ad]\!]}$. Thus we can speak of multi-sections $\bold s$ of 
$\sF_{[\![\ad]\!]}$ intersecting $C_{[\![\ad]\!]}\sub E_{[\![\ad]\!]}$ transversally. 
%, and the intersection defines a pseudo-cycle

In the future, we will call a multi-section of $\sF_{[\![\ad]\!]}$ intersecting 
$C_{[\![\ad]\!]}$ transversally a {\it good multi-section}. 
Let $k_{[\ad]}$ be the virtual dimension of $Y^{[\ad]}_T$.
For a good multi-section $\bs_{[\![\ad]\!]}$ of $\sF_{[\![\ad]\!]}$, we denote 
$$ D(\bs_{[\![\ad]\!]}) =\bar\pi(C_{[\![\ad]\!]}\cap s_{[\![\ad]\!]}),
$$
where $s_{[\![\ad]\!]}$ is a representative of $\bs_{[\![\ad]\!]}$, 
and $D(\bs_{[\![\ad]\!]})$ is a piecewise smooth $k_{[\ad]}$-dimensional $\QQ$-current 
away from a subset of dimension at most $k_{[\ad]}-2$.
(Note that $D(\bs_{[\![\ad]\!]})$ is independent of the choice of $s_{[\![\ad]\!]}$.)
We denote
$$[D(\bs_{[\![\ad]\!]})] \in H_{k_{[\ad]}}(|Y^{[\![\ad]\!]}_T|;\QQ)
$$
the homology class it represents. 

Applying the pseudo-cycle representative of Gysin maps, we obtain

\begin{proposition}     \label{MulSecVir}
Given a good multi-section $\bs_{[\![\ad]\!]}$ of $\sF_{[\![\ad]\!]}$,
we have 
$$[D(\bs_\adbracket)]=[Y_T^\adbracket]\virt\in H_*(|Y_T^\adbracket|;\mathbb Q).
$$
\end{proposition}

\subsection{Comparison of virtual cycles}
\label{subsect_CompVirCyc}

Our goal in this subsection is to compare the virtual cycles in terms of pseudo-cycle representatives. 
We will prove the analogue of Lemma~5.6 in \cite{LiJ}. 

To begin with, we recall $\alpha$-diagonals, their tubular neighborhoods, and the associated 
partitions from \cite{LiJ}. For $\alpha \in \mathcal P_\Lambda$, we form the strict $\alpha$-diagonal:
\beq\label{a-diag}
\Delta_\alpha = \Delta_\alpha^Y 
= \{x \in Y_T^\Lambda|\, a \sim_\alpha b \Rightarrow x_a = x_b \};
\eeq
it is closed in $Y_T^\Lambda$ and isomorphic to $Y_T^l$ when 
$\alpha = (\alpha_1, \ldots, \alpha_l)$. Fix a sufficiently small number $c > 0$ and
a large real $N$, and pick a function $\epsilon: \mathcal P_\Lambda \to (0, c)$ whose values on 
any ordered pair $\alpha > \beta$ satisfy $\epsilon(\alpha) > N \cdot \epsilon(\beta)$.
After fixing a Riemannian metric on $Y$, 
we define the $\epsilon$-neighborhood of $\Delta_\alpha \subset Y_T^\Lambda$ to be
\beq\label{a-neigh}
\Delta_{\alpha, \epsilon} = \Delta_{\alpha, \epsilon}^Y =
\{x \in Y_T^\Lambda \mid  {\rm dist}(x, \Delta_\alpha) < \epsilon(\alpha) \}.
\eeq

For a pair $\alpha \ge \beta$, we define $\Delta^{\alpha}_{\beta, \epsilon} 
= \cup_{\alpha \ge \gamma \ge \beta} \Delta_{\gamma, \epsilon}$ and 
$Q^{\alpha}_{\beta, \epsilon} 
= \Delta_{\beta, \epsilon} - \cup_{\alpha \ge \gamma > \beta} \Delta^{\alpha}_{\gamma, \epsilon}
= \Delta_{\beta, \epsilon} - \cup_{\alpha \ge \gamma > \beta} \Delta_{\gamma, \epsilon}.$
Then, $Q^{\alpha}_{\beta, \epsilon}$ is a closed subset of $\Delta_{\beta, \epsilon}$.
%In case $\alpha=\Lambda$, we abbreviate $Q^\Lambda_{\beta,\eps}=Q_{\beta,\eps}$. 
%
By Lemma~5.5 of \cite{LiJ}, if $\Delta_{\beta_1, \epsilon} \cap Q^{\alpha}_{\beta_2, \epsilon}
\ne \emptyset$ for some $\beta_1, \beta_2 \le \alpha$, then 
\beq   \label{beta21}
\beta_1 \le \beta_2.
\eeq
It follows that $\Delta^{\alpha}_{\beta, \epsilon}
= \coprod _{\alpha \ge \gamma \ge \beta} Q^{\alpha}_{\gamma, \epsilon}.$
In particular, for any $\alpha$, by taking $\beta = 1^\Lambda$, we get 
$Y_T^\Lambda = \coprod _{\gamma \le \alpha} Q^{\alpha}_{\gamma, \epsilon}$.
Further, letting $\mathcal Q^{[\![\alpha, \delta]\!]}_{\beta, \epsilon}
= Y_T^\adbracket\times_{Y^\Lambda_T} Q^{\alpha}_{\beta, \epsilon}$, we obtain
$Y_T^{[\![\alpha, \delta]\!]}
= \coprod_{\beta \le \alpha} \mathcal Q^{[\![\alpha, \delta]\!]}_{\beta, \epsilon}.$
Note that for fixed $\beta$ with $\beta \le \alpha$, we have
$\mathcal Q^{[\![\alpha, \delta]\!]}_{\beta, \epsilon} \subset \coprod_{(\beta, \eta) \le 
(\alpha, \delta)} Y^{[\![\alpha, \delta]\!]}_{(\beta, \eta)}$. Define
$
\mathcal Q^{[\![\alpha, \delta]\!]}_{(\beta, \eta), \epsilon} 
= \mathcal Q^{[\![\alpha, \delta]\!]}_{\beta, \epsilon} \cap
Y^{[\![\alpha, \delta]\!]}_{(\beta, \eta)}
$
for $(\beta, \eta) \le (\alpha, \delta)$. Then, we obtain a partition:
\begin{eqnarray}  \label{PtnMYAlpha}
Y_T^{[\![\alpha, \delta]\!]} = \coprod_{(\beta, \eta) \le (\alpha, \delta)} 
\mathcal Q^{[\![\alpha, \delta]\!]}_{(\beta, \eta), \epsilon}.
\end{eqnarray}

\begin{lemma}   \label{lma5.6-LiJ}
For sufficiently small $\eps$, we can find a collection of good multi-sections $\bs_{[\![\ad]\!]}$
of $\sF_{[\![\ad]\!]}$ that satisfy the properties
\begin{enumerate}
\item[{\rm (i)}] each $\bs_{[\![\ad]\!]}$ intersects transversally with the cycle $C_{[\![\ad]\!]}\sub E_{[\![\ad]\!]}$;

\item[{\rm (ii)}] for $(\beta, \eta) < (\alpha, \delta)$, the pseudo-cycles (as $\QQ$-currents)
$$
(\zeta_\ad^\be)_*\bigl(D(\bs_{[\![\ad]\!]})\cap \mathcal Q^{[\![\ad]\!]}_{(\be), \epsilon}\bigr) =
D(\bs_{[\![\be]\!]}) \cap \zeta_\ad^\be \bigl(\mathcal Q^{[\![\ad]\!]}_{(\be), \epsilon}\bigr).
$$
\end{enumerate}
\end{lemma}

\begin{proof}
We follow the proof of \cite[Lemma 5.6]{LiJ} line by line, with $\mathcal Q^\alpha_{\beta,\epsilon}$
(respectively, $s_\alpha$) in \cite[p. 2156]{LiJ} replaced by $\mathcal Q^{[\![\ad]\!]}_{(\be),\epsilon}$ 
(respectively, $\bs_{[\![\ad]\!]}$).
 
To carry the argument in \cite[p. 2156]{LiJ} through in the current situation, two modifications are
necessary. The first is using multi-sections of $\sF_{[\![\ad]\!]}$, etc. The two properties of
sections we used in the proof of \cite[Lemma 5.6]{LiJ} are the existence of extensions and
general position results. For multiple-sections, similar results hold as we have mentioned before.

The other is to choose multi-section $\bs_{[\![\ad]\!]}|_{[\![\be]\!]}$ of 
$\sF_{[\![\ad]\!]}|_{\mathcal Q^{[\![\ad]\!]}_{(\be),\epsilon}}$ to be the pullback
$$\bs_{[\![\ad]\!]}|_{[\![\be]\!]} = (\zeta_\ad^\be)\sta
\bigl(\bs_{[\![\be]\!]}|_{\zeta_\ad^\be(\mathcal Q^{[\![\ad]\!]}_{(\be), \epsilon})}\bigr).
$$
(Compare the construction of $s_\alpha|_\beta=s_\beta|_{\mathcal Q^\alpha_{(\beta,\alpha)}}$ 
in \cite[p. 2156]{LiJ}.) Since $\sF_{[\![\ad]\!]}|_{Y^{[\![\ad]\!]}_{(\be)}}
=(\zeta_\ad^\be)\sta\sF_{[\![\be]\!]}$, such pullback is well-defined.

What we need to make sure is that the section $\bs_{[\![\ad]\!]}|_{[\![\be]\!]}$ 
(of $\sF_{[\![\ad]\!]}|_{\mathcal Q^{[\![\ad]\!]}_{(\be),\epsilon}}$)
intersects transversally with the cycle $C_{[\![\ad]\!]}$; this is true, following Lemma \ref{lem0.5} 
and Lemma \ref{rat-gen}. This completes the proof  of the lemma.
\end{proof}

\subsection{Approximating virtual cycles}
\label{subsect_Approx(continued)}
 {In this subsection, we define the pseudo-cycle $\Theta^{[\![\alpha, \delta]\!]} $ and study its properties. The formula  (\ref{Identity-VirTheta}) below can be roughly thought of as a 
decomposition of the virtual cycle $\ev_*\big [Y_T^{[\![\alpha, \delta]\!]} \big ]^{\rm vir}$ as a sum of cycles $\Theta^{[\![\beta, \eta]\!]} $ supported near $\alpha$-diagonals. The ideal situation is that we have a similar 
decomposition for $\ev_*\big [Y_T^{[n,d ]} \big ]^{\rm vir}$ in $(Y_T^{[n]})^3$. Unfortunately, such a decomposition doesn't exist. However, the decomposition (\ref{Identity-VirTheta}) works equally well as if we had a decomposition for $\ev_*\big [Y_T^{[n,d ]} \big ]^{\rm vir}$. This is carried out in subsections \ref{subsect_IntWith} and \ref{subsect_Structure}. In fact, the main reason for introducing  Hilbert schemes of $\alpha$-points $Y_T^{[\![\alpha]\!}$,  non-separated spaces $Y_T^{[\![\le\alpha]\!}$, and moduli spaces $Y_T^{[\![\alpha, \delta]\!}$ of $\alpha$-stable maps to Hilbert schemes  is to provide appropriate spaces where we can define $\Theta^{[\![\alpha, \delta]\!]} $.}

Let $(\beta, \eta) \le (\alpha, \delta) \in \mathcal P_{\Lambda, d}$. Define
$$
\phi_{\beta, \alpha}: \, Y_T^{[\![\beta]\!]} \to Y_T^{[\![\le \alpha]\!]}, \quad
\w \phi_{\beta, \alpha}: \, Y_T^{[\![\le \beta]\!]} \to Y_T^{[\![\le \alpha]\!]}
$$
to be the open immersions induced from the construction (\ref{YLeAlpha}).
The evaluation map ${\rm ev}_i: Y_T^{[\Lambda, d]} \to Y_T^{[\Lambda]}$ induces 
an evaluation map $Y_T^{[\![\alpha, \delta]\!]} \to Y_T^{[\![\alpha]\!]}$ 
which will be denoted again by ${\rm ev}_i$. Let 
$\ev = {\rm ev}_1 \times {\rm ev}_2 \times {\rm ev}_3: 
Y_T^{[\![\alpha, \delta]\!]} \to (Y_T^{[\![\alpha]\!]})^3$.
Since $\ev_i: Y_T^{[\![\alpha, \delta]\!]} = Y_T^{[\alpha, \delta]} \times_{Y_T^{(\alpha)}} 
Y_T^\Lambda \to Y_T^{[\![\alpha]\!]} = Y_T^{[\alpha]} \times_{Y_T^{(\alpha)}} Y_T^\Lambda$ 
does not affect the factor $Y_T^\Lambda$, we have $\ev(Y_T^{[\![\alpha, \delta]\!]}) 
\subset (Y_T^{[\![\alpha]\!]})^3 \times_{(Y_T^\Lambda)^3} \iota_\Lambda(Y_T^\Lambda)$ where 
$$
\iota_\Lambda: Y_T^\Lambda \to (Y_T^\Lambda)^3
$$ 
is the diagonal embedding. Define the indexing morphism to be
\begin{eqnarray}  \label{IndexMap}
\ind: \quad \bigcup_{(\beta, \eta) \le (\alpha, \delta)} (\phi^3_{\beta, \alpha}) \ev 
(Y_T^{[\![\beta, \eta]\!]}) \lra \iota_\Lambda(Y_T^\Lambda) \cong Y_T^\Lambda.
\end{eqnarray}

\begin{definition}   \label{ThetaAlpha}
Define the pseudo-cycle $\Theta^{[\![\alpha, \delta]\!]} \sub (Y_T^{[\![\le \alpha]\!]})^3$ inductively by
\begin{eqnarray}  \label{ThetaAlpha.1}
\Theta^{[\![\alpha, \delta]\!]} 
= (\phi^3_{\alpha, \alpha})_*\ev_*D(\bs_{[\![\alpha, \delta]\!]}) - \sum_{(\beta, \eta) 
  < (\alpha, \delta)} (\W \phi^3_{\beta, \alpha})_*\Theta^{[\![\beta, \eta]\!]}.
\end{eqnarray}
\end{definition}

By Proposition~\ref{MulSecVir}, we obtain
\begin{eqnarray}  \label{Identity-VirTheta}
   (\phi^3_{\alpha, \alpha})_*\ev_*\big [Y_T^{[\![\alpha, \delta]\!]} \big ]^{\rm vir}
= \sum_{(\beta, \eta) \le (\alpha, \delta)} (\W \phi^3_{\beta, \alpha})_*[\Theta^{[\![\beta, \eta]\!]}].
\end{eqnarray}
Further properties of the pseudo-cycles $\Theta^{[\![\alpha, \delta]\!]}$ are contained 
in the next two lemmas which are the analogues of Lemmas 5.7 and 5.9 in \cite{LiJ}.

\begin{lemma}   \label{lma5.7-LiJ(PropA)}
$\Theta^{[\![\alpha, \delta]\!]} \subset (Y_T^{[\![\le \alpha]\!]})^3 
\times_{(Y_T^\Lambda)^3} \iota_\Lambda(\Delta_{\alpha, \epsilon})$ for sufficiently small $c > 0$.
\end{lemma}
\begin{proof}
We use induction on the order of $(\alpha, \delta) \in \mathcal P_{\Lambda, d}$.
Assume that $(\alpha, \delta)$ is minimal.
Then for each $i$, we have either $(\alpha_i, \delta_i) = (1, 0)$, 
or $\alpha_i = 2$ and $\delta_i > 0$. So 
$
Y_T^{[\![\alpha, \delta]\!]} 
= Y_T^{[\![\alpha, \delta]\!]} \times_{Y^\Lambda_T} \Delta_{\alpha}
= Y_T^{[\![\alpha, \delta]\!]} \times_{Y^\Lambda_T} \Delta_{\alpha, \epsilon}.
$
Thus, $\Theta^{[\![\alpha, \delta]\!]} 
= (\phi^3_{\alpha, \alpha})_*\ev_*D(\bs_{[\![\alpha, \delta]\!]}) \subset 
(Y_T^{[\![\le \alpha]\!]})^3 \times_{(Y_T^\Lambda)^3} \iota_\Lambda(\Delta_{\alpha, \epsilon})$.

Next, we assume that our lemma is true for every $(\gamma, \rho)$ with 
$(\gamma, \rho) < (\alpha, \delta)$. Recall that 
${Y^\Lambda_T} = \coprod_{\beta \le \alpha} Q^{\alpha}_{\beta, \epsilon}$ and $Q^{\alpha}_{\alpha, \epsilon} = \Delta_{\alpha, \epsilon}$.
So to prove the lemma, it suffices to verify $\Theta^{[\![\alpha, \delta]\!]} \cap 
\big ((Y_T^{[\![\le \alpha]\!]})^3 \times_{(Y_T^\Lambda)^3} 
\iota_\Lambda(Q^{\alpha}_{\beta, \epsilon}) \big ) = \emptyset$ for every $\beta < \alpha$.
By (\ref{ThetaAlpha.1}), this is equivalent to proving that the intersection
\begin{eqnarray}  \label{lma5.7-LiJ(PropA).1}
& &(\phi^3_{\alpha, \alpha})_*\ev_*D(\bs_{[\![\alpha, \delta]\!]}) 
   \cap \big ((Y_T^{[\![\le \alpha]\!]})^3 \times_{(Y_T^\Lambda)^3} 
   \iota_\Lambda(Q^{\alpha}_{\beta, \epsilon}) \big )  \nonumber    \\
&=&\sum_{(\gamma, \rho) < (\alpha, \delta)} (\W \phi^3_{\gamma, \alpha})_*
   \Theta^{[\![\gamma, \rho]\!]} \cap \big ((Y_T^{[\![\le \alpha]\!]})^3 
   \times_{(Y_T^\Lambda)^3} \iota_\Lambda(Q^{\alpha}_{\beta, \epsilon}) \big ).
\end{eqnarray}
On one hand, if $(\gamma, \rho) < (\alpha, \delta)$, then $\Theta^{[\![\gamma, \rho]\!]} 
\subset (Y_T^{[\![\le \gamma]\!]})^3 \times_{(Y_T^\Lambda)^3} 
\iota_\Lambda(\Delta_{\gamma, \epsilon})$ by induction. 
Thus, a nonempty $(\W \phi^3_{\gamma, \alpha})_*\Theta^{[\![\gamma, \rho]\!]} \cap 
\big ((Y_T^{[\![\le \alpha]\!]})^3 \times_{(Y_T^\Lambda)^3} 
\iota_\Lambda(Q^{\alpha}_{\beta, \epsilon}) \big )$ forces $\Delta_{\gamma, \epsilon} \cap 
Q^{\alpha}_{\beta, \epsilon} \ne \emptyset$ which in turn implies $\gamma \le \beta$ by
\eqref{beta21}. Therefore, the right-hand-side of (\ref{lma5.7-LiJ(PropA).1}) equals
\begin{eqnarray}     \label{lma5.7-LiJ(PropA).100}
& &\sum_{(\gamma, \rho) < (\alpha, \delta), \gamma \le \beta} (\W \phi^3_{\gamma, \alpha})_*
   \Theta^{[\![\gamma, \rho]\!]} \cap \big ((Y_T^{[\![\le \alpha]\!]})^3 
   \times_{(Y_T^\Lambda)^3} \iota_\Lambda(Q^{\alpha}_{\beta, \epsilon}) \big )   \nonumber \\
&=&\sum_{(\beta, \eta) \le (\alpha, \delta)} \sum_{(\gamma, \rho) \le (\beta, \eta)} 
   (\W \phi^3_{\beta, \alpha})_*(\W \phi^3_{\gamma, \beta})_*
   \Theta^{[\![\gamma, \rho]\!]} \cap \big ((Y_T^{[\![\le \alpha]\!]})^3 
   \times_{(Y_T^\Lambda)^3} \iota_\Lambda(Q^{\alpha}_{\beta, \epsilon}) \big ).
\end{eqnarray}
Since $(\phi^3_{\beta, \beta})_*\ev_*D(\bs_{[\![\beta, \eta]\!]}) = \sum_{(\gamma, \rho) 
\le (\beta, \eta)} (\W \phi^3_{\gamma, \beta})_*\Theta^{[\![\gamma, \rho]\!]}$, 
\eqref{lma5.7-LiJ(PropA).100} is equal to
\begin{eqnarray}   \label{lma5.7-LiJ(PropA).101}
%& &\sum_{(\gamma, \rho) < (\alpha, \delta)} (\W \phi^3_{\gamma, \alpha})_*
%   \Theta^{[\![\gamma, \rho]\!]} \cap \big ((Y_T^{[\![\le \alpha]\!]})^3 
%   \times_{(Y_T^\Lambda)^3} \iota_\Lambda(Q^{\alpha}_{\beta, \epsilon}) \big )       \\
& &\sum_{(\beta, \eta) \le (\alpha, \delta)} (\W \phi^3_{\beta, \alpha})_*
   (\phi^3_{\beta, \beta})_*\ev_*D(\bs_{[\![\beta, \eta]\!]}) 
   \cap \big ((Y_T^{[\![\le \alpha]\!]})^3 \times_{(Y_T^\Lambda)^3} 
   \iota_\Lambda(Q^{\alpha}_{\beta, \epsilon}) \big )    \nonumber     \\
&=&\sum_{(\beta, \eta) \le (\alpha, \delta)} (\phi^3_{\beta, \alpha})_*
   \ev_*D(\bs_{[\![\beta, \eta]\!]}) \cap \big ((Y_T^{[\![\le \alpha]\!]})^3 
   \times_{(Y_T^\Lambda)^3} \iota_\Lambda(Q^{\alpha}_{\beta, \epsilon}) \big ).
\end{eqnarray}
Since $Q^{\alpha}_{\beta, \epsilon} = \Delta_{\beta, \epsilon} - 
\cup_{\alpha \ge \gamma > \beta} \Delta^{\alpha}_{\gamma, \epsilon}$,
we see that $(\phi^3_{\beta, \alpha})_*\ev_*D(\bs_{[\![\beta, \eta]\!]}) \cap 
\big ((Y_T^{[\![\le \alpha]\!]})^3 \times_{(Y_T^\Lambda)^3} 
\iota_\Lambda(Q^{\alpha}_{\beta, \epsilon}) \big )$ is contained in 
$\phi^3_{\beta, \alpha} \ev \big ( \zeta_\ad^\be  
(\mathcal Q^{[\![\alpha, \delta]\!]}_{(\beta, \eta), \epsilon})\big )$.  
So \eqref{lma5.7-LiJ(PropA).101} (hence the right-hand-side of (\ref{lma5.7-LiJ(PropA).1})) equals
\begin{eqnarray}  \label{lma5.7-LiJ(PropA).2}
%& &\sum_{(\gamma, \rho) < (\alpha, \delta)} (\W \phi^3_{\gamma, \alpha})_*
%   \Theta^{[\![\gamma, \rho]\!]} \cap \big ((Y_T^{[\![\le \alpha]\!]})^3 
%   \times_{(Y_T^\Lambda)^3} \iota_\Lambda(Q^{\alpha}_{\beta, \epsilon}) \big )  \nonumber   \\
& &\sum_{(\beta, \eta) \le (\alpha, \delta)} (\phi^3_{\beta, \alpha})_*
   \ev_*D(\bs_{[\![\beta, \eta]\!]}) \cap \phi^3_{\beta, \alpha} \ev 
   \big ( \zeta_\ad^\be (\mathcal Q^{[\![\alpha, \delta]\!]}_{(\beta, \eta), \epsilon})
   \big )    \nonumber   \\
&=&\sum_{(\beta, \eta) \le (\alpha, \delta)} (\phi^3_{\beta, \alpha})_*\ev_* 
   \big (D(\bs_{[\![\beta, \eta]\!]}) \cap \zeta_\ad^\be  
   (\mathcal Q^{[\![\alpha, \delta]\!]}_{(\beta, \eta), \epsilon}) \big ) .
\end{eqnarray}
On the other hand, $(\phi^3_{\alpha, \alpha})_*\ev_*D(\bs_{[\![\alpha, \delta]\!]})$ is 
supported on $\cup_{(\gamma, \rho) \le (\alpha, \delta)} \phi^3_{\alpha, \alpha} \ev
(\mathcal Q^{[\![\alpha, \delta]\!]}_{(\gamma, \rho), \epsilon})$ by (\ref{PtnMYAlpha}). Moreover, 
$\phi^3_{\alpha, \alpha} \ev(\mathcal Q^{[\![\alpha, \delta]\!]}_{(\gamma, \rho), \epsilon})$
is contained in $(Y_T^{[\![\le \alpha]\!]})^3 \times_{(Y_T^\Lambda)^3} 
\iota_\Lambda(Q^{\alpha}_{\gamma, \epsilon})$, and the subsets 
$\iota_\Lambda(Q^{\alpha}_{\gamma, \epsilon}), \gamma \le \alpha$ are disjoint. 
So the left-hand-side of (\ref{lma5.7-LiJ(PropA).1}) is equal to
\begin{eqnarray*} 
%& &(\phi^3_{\alpha, \alpha})_*\ev_*D(\bs_{[\![\alpha, \delta]\!]}) 
%   \cap \big ((Y_T^{[\![\le \alpha]\!]})^3 \times_{(Y_T^\Lambda)^3} 
%   \iota_\Lambda(Q^{\alpha}_{\beta, \epsilon}) \big )   \\
& &\sum_{(\beta, \eta) \le (\alpha, \delta)} (\phi^3_{\alpha, \alpha})_*\ev_*
   D(\bs_{[\![\alpha, \delta]\!]}) \cap \phi^3_{\alpha, \alpha} \ev
   (\mathcal Q^{[\![\alpha, \delta]\!]}_{(\beta, \eta), \epsilon})    \\
&=&\sum_{(\beta, \eta) \le (\alpha, \delta)} (\phi^3_{\alpha, \alpha})_*\ev_*
   (D(\bs_{[\![\alpha, \delta]\!]}) \cap   
   \mathcal Q^{[\![\alpha, \delta]\!]}_{(\beta, \eta), \epsilon})  \\
&=&\sum_{(\beta, \eta) \le (\alpha, \delta)} (\phi^3_{\beta, \alpha})_*\ev_*
   \big ( (\zeta_\ad^\be)_*\bigl(D(\bs_{[\![\ad]\!]}) \cap 
   \mathcal Q^{[\![\ad]\!]}_{(\be), \epsilon}\bigr) \big )   \\
&=&\sum_{(\beta, \eta) \le (\alpha, \delta)} (\phi^3_{\beta, \alpha})_*\ev_*
   \big (D(\bs_{[\![\beta, \eta]\!]}) \cap \zeta_\ad^\be  
   (\mathcal Q^{[\![\alpha, \delta]\!]}_{(\beta, \eta), \epsilon}) \big )
\end{eqnarray*}
where we have used Lemma~\ref{lma5.6-LiJ}~(ii) in the last step. 
Combining with (\ref{lma5.7-LiJ(PropA).2}), we get (\ref{lma5.7-LiJ(PropA).1}).
\end{proof}

\begin{lemma}   \label{ProdTheta}
Let $(\alpha, \delta) \in \mathcal P_{\Lambda, d}$ with 
$\alpha = (\alpha_1, \ldots, \alpha_l)$. Then,
$\Theta^{[\![\alpha, \delta]\!]} = \prod_{i=1}^l \Theta^{[\![\alpha_i, \delta_i]\!]}$
via the natural identification $(Y_T^{[\![\le \alpha]\!]})^3 = \prod_{i=1}^l (Y_T^{[\![\le \alpha_i]\!]})^3$.
\end{lemma}
\begin{proof}
First of all, since $Y_T^{[\![\alpha, \delta]\!]} = \prod_{i=1}^l Y_T^{[\![\alpha_i, \delta_i]\!]}$, we have
\begin{eqnarray}      \label{ProdTheta.1}
D(\bs_{[\![\alpha, \delta]\!]}) = \prod_{i=1}^l D(\bs_{[\![\alpha_i, \delta_i]\!]}).
\end{eqnarray}
Next, to prove the lemma, we use induction on the size $|\Lambda|$ and 
on the order of 
$(\alpha, \delta) \in \mathcal P_{\Lambda, d}$. Assume that $(\alpha, \delta)$ is minimal in 
$\mathcal P_{\Lambda, d}$. Then $(\alpha_i, \delta_i)$ is minimal in 
$\mathcal P_{\alpha_i, \delta_i}$. By (\ref{ThetaAlpha.1}) and \eqref{ProdTheta.1}, 
\begin{eqnarray*}  
\Theta^{[\![\alpha, \delta]\!]} 
= (\phi^3_{\alpha, \alpha})_*\ev_*D(\bs_{[\![\alpha, \delta]\!]})
= \prod_{i=1}^l (\phi^3_{\alpha_i, \alpha_i})_*\ev_*D(\bs_{[\![\alpha_i, \delta_i]\!]})
= \prod_{i=1}^l  \Theta^{[\![\alpha_i, \delta_i]\!]}.
\end{eqnarray*}
In particular, the lemma holds for $|\Lambda| = 1$ (necessarily, $(\alpha, \delta) = (1, 0)$). 
Next, assume that $\Theta^{[\![\beta, \eta]\!]} = \prod_{i} \Theta^{[\![\beta_i, \eta_i]\!]}$ for 
every $(\beta, \eta) < (\alpha, \delta)$. By \eqref{ProdTheta.1} and (\ref{ThetaAlpha.1}),
\begin{eqnarray*}  
   (\phi^3_{\alpha, \alpha})_*\ev_*D(\bs_{[\![\alpha, \delta]\!]})
&=&\prod_{i=1}^l (\phi^3_{\alpha_i, \alpha_i})_*\ev_*D(\bs_{[\![\alpha_i, \delta_i]\!]})  \\
&=&\prod_{i=1}^l \sum_{(\beta^{(i)}, \eta^{(i)}) \le (\alpha_i, \delta_i)} 
   (\W \phi^3_{\beta^{(i)}, \alpha_i})_*\Theta^{[\![\beta^{(i)}, \eta^{(i)}]\!]}  \\
&=&\prod_{i=1}^l \Theta^{[\![\alpha_i, \delta_i]\!]} 
   + \sum_{(\beta, \eta) < (\alpha, \delta)}
   (\W \phi^3_{\beta, \alpha})_*\Theta^{[\![\beta, \eta]\!]}
\end{eqnarray*}
noting that induction has been used in the last step to handle those $\beta^{(i)}$ 
which have length greater than $1$. Applying (\ref{ThetaAlpha.1}) again, we obtain the lemma.
\end{proof}
\subsection{Co-section localizations}
\label{subsect_CosectLoc}

We now apply the co-section localization techniques from \cite{KL1, KL2, LL} to 
the constructions in the previous subsections.
Let $\theta$ be a meromorphic section of $\mathcal O_X(K_X)$, and let $D_0$ and $D_\infty$ be the 
vanishing and pole divisors of $\theta$ respectively. For simplicity, we assume that
$D_0$ and $D_\infty$ are smooth irreducible curves intersecting transversally. 
Let $X^{[n, d]}_{\theta}$ be the subset of 
$X^{[n, d]}$ consisting of those $\varphi$ whose standard decomposition
$(\varphi_1, \ldots, \varphi_l)$ have the property that for each $i$, 
either $\varphi_i$ is constant or the support of $\varphi_i$ lies in $D_0 \cup D_\infty$.
The meromorphic section $\theta$ induces a meromorphic section $\theta^{[n]}$ 
of $\Omega^2_{\Xn}$. By \cite{KL2, LL}, we obtain the localized virtual 
fundamental cycle $\big [X^{[n, d]} \big ]^{\rm vir}_{\rm loc}
\in A_*(X^{[n, d]}_{\theta})$ of $X^{[n, d]}$ 
such that $\iota_*\big [X^{[n, d]} \big ]^{\rm vir}_{\rm loc}
= \big [X^{[n, d]} \big ]^{\rm vir}$
where $\iota_*$ is the map induced by the inclusion map $\iota: X^{[n, d]}_{\theta} 
\hookrightarrow X^{[n, d]}$. For simplicity of notations, we write
$\big [X^{[n, d]} \big ]^{\rm vir}_{\rm loc} = \big [X^{[n, d]} \big ]^{\rm vir}$.

The constructions in \cite{KL2, LL} and Subsections~5.1-\ref{subsect_Approx(continued)} 
are canonical. Applying the constructions in \cite{KL2, LL} 
to Subsections~5.1-\ref{subsect_Approx(continued)}, we obtain localized cycles 
$\big [X^{[\![\alpha,\delta]\!]} \big ]^{\rm vir}_{\rm loc}
\in H_*(X^{[\![\alpha,\delta]\!]}_{\theta}; \mathbb Q)$, { $D(\bs_{[\![\alpha, \delta]\!]})_{{\rm loc}}$}, and 
$\Theta^{[\![\alpha, \delta]\!]}_{{\rm loc}} \sub \cup_{(\beta, \eta) \le 
(\alpha,\delta)} \phi^3_{\beta, \alpha}\ev(X^{[\![\beta, \eta]\!]}_{\theta})$ with
$\big [X^{[\![\alpha,\delta]\!]} \big ]^{\rm vir}_{\rm loc}
= \big [X^{[\![\alpha,\delta]\!]} \big ]^{\rm vir}$ and 
$[\Theta^{[\![\alpha,\delta]\!]}_{{\rm loc}}] = [\Theta^{[\![\alpha, \delta]\!]}]$
in $H_*(X^{[\![\alpha,\delta]\!]}; \mathbb Q)$ and 
$H_*\big ( (X^{[\![\le \alpha]\!]})^3; \mathbb Q \big )$ respectively.
Here the subset $X^{[\![\alpha,\delta]\!]}_\theta \subset X^{[\![\alpha,\delta]\!]}$ is
defined similarly as $X^{[n, d]}_{\theta} \subset X^{[n, d]}$.
\subsection{Extensions of Heisenberg monomial classes}
\label{subsect_Extension}

Let $(\beta, \eta) \in \mathcal P_{[n], d}$. To study the pairings with
$(\W \phi^3_{\beta, [n]})_*[\Theta^{[\![\beta, \eta]\!]}]$, we need to extend 
the classes $(f^{[\![n]\!]})^*w \in H^*(X^{[\![n]\!]})$ from 
$X^{[\![n]\!]}$ to $X^{[\![\le n]\!]}$, where $f^{[\![n]\!]}: X^{[\![n]\!]} \to X^{[n]}$ 
is the tautological map. Let $f^{[\![\beta]\!]} = \prod_i f^{[\![|\beta_i|]\!]}$.

\begin{lemma}   \label{ExtHeis}
Let $\alpha_i \in H^*(X)$ be homogeneous with $|\alpha_i| > 0$, and $\alpha_{i,j} = 1_X$. Let
\begin{eqnarray}      \label{ExtHeis.01}
w = \displaystyle{\left (
\prod_{i=1}^t \prod_{j=1}^{s_i} \frak a_{-i}(\alpha_{i,j}) \right )
\left ( \prod_{i=1}^k \frak a_{-n_i}(\alpha_i) \right ) \vac} \quad \in H^*(\Xn).
\end{eqnarray}
Then there exists a class $w^{[\![\le n]\!]} \in H^*(X^{[\![\le n]\!]})$ such that 
$(\phi_{[n], [n]})^*w^{[\![\le n]\!]} = (f^{[\![n]\!]})^*w$,
and that if $\beta = (\beta_1, \ldots, \beta_l) \le [n]$, then via the identification 
$X^{[\![\le \beta]\!]} = \prod_{i=1}^l X^{[\![\le \beta_i]\!]}$,
\begin{eqnarray}      \label{ExtHeis.03}
(\W \phi_{\beta, [n]})^*w^{[\![\le n]\!]} = \sum_{w_1 \circ \cdots \circ w_l = w} 
\otimes_{i=1}^l w_i^{[\![\le |\beta_i|]\!]}
\end{eqnarray}
where each $w_i \in H^*(X^{[|\beta_k|]})$ is a Heisenberg monomial class.
\end{lemma}
\begin{proof}
We use induction on $n$. The lemma is trivially true when $n=1$. In the following,
assume that the lemma holds for all $X^{[m]}$ with $m < n$.

Let $S$ be the set consisting of all $\beta < [n]$ such that there does not exist
$\gamma < [n]$ with $\beta < \gamma$. Then $X^{[\![\le n]\!]}$ is covered by the open subsets
$\phi_{[n], [n]}(X^{[\![n]\!]})$ and $\W \phi_{\beta, [n]}(X^{[\![\le \beta]\!]}), 
\beta \in S$. Define
$$
w^{[\![\le \beta ]\!]} = \sum_{w_1 \circ \cdots \circ w_l = w} 
\otimes_{i=1}^l w_i^{[\![\le |\beta_i|]\!]} \in H^*(X^{[\![\le \beta]\!]})
= H^*(\W \phi_{\beta, [n]}(X^{[\![\le \beta]\!]}))
$$
for each $\beta \in S$.
Applying the Mayer-Vietoris sequence successively, we see that to prove the existence
of $w^{[\![\le n]\!]} \in H^*(X^{[\![\le n]\!]})$, it suffices to show that $(f^{[\![n]\!]})^*w$ 
and $w^{[\![\le \beta ]\!]}, \beta \in S$ are equal on the overlaps of the open subsets 
$\phi_{[n], [n]}(X^{[\![n]\!]})$ and $\W \phi_{\beta, [n]}(X^{[\![\le \beta]\!]}), 
\beta \in S$.

First of all, let $\beta, \gamma = (\gamma_1, \ldots, \gamma_r) \in S$ and 
$\beta \ne \gamma$. Then, 
$$
\W \phi_{\beta, [n]}(X^{[\![\le \beta]\!]}) \cap \W \phi_{\gamma, [n]}(X^{[\![\le \gamma]\!]}) 
\sub \W \phi_{\beta \wedge \gamma, [n]}(X^{[\![\le \beta \wedge \gamma]\!]}).
$$
Let $\beta_i \wedge \gamma = (\beta_i \cap \gamma_1, \ldots, \beta_i \cap 
\gamma_r) \in \mathcal P_{\beta_i}$. 
Then $(\W \phi_{\beta \wedge \gamma, \beta})^*w^{[\![\le \beta]\!]}$ is equal to
$$
(\W \phi_{\beta \wedge \gamma, \beta})^*\sum_{w_1 \circ \cdots \circ w_l = w} 
      \otimes_{i=1}^l w_i^{[\![\le |\beta_i|]\!]}   
\,\, = \,\, \sum_{w_1 \circ \cdots \circ w_l = w} \otimes_{i=1}^l 
      (\W \phi_{\beta_i \wedge \gamma, \beta_i})^*w_i^{[\![\le |\beta_i|]\!]}.
$$
Applying induction to the classes $w_i^{[\![\le |\beta_i|]\!]}$, we see that
\begin{eqnarray}  \label{ExtHeis.1}
   (\W \phi_{\beta \wedge \gamma, \beta})^*w^{[\![\le \beta]\!]} 
&=&\sum_{w_1 \circ \cdots \circ w_l = w} \otimes_{i=1}^l \left (
   \sum_{w_{i,1} \circ \cdots \circ w_{i, r} = w_i} \otimes_{j=1}^r
   w_{i, j}^{[\![\le |\beta_i \cap \gamma_j|]\!]} \right )   \nonumber   \\
&=&\sum_{w_{1,1} \circ \cdots \circ w_{l, r} = w} \otimes_{i=1}^l \otimes_{j=1}^r
   w_{i, j}^{[\![\le |\beta_i \cap \gamma_j|]\!]}.
\end{eqnarray}
It follows immediately that $(\W \phi_{\beta \wedge \gamma, \beta})^*w^{[\![\le \beta]\!]}
= (\W \phi_{\beta \wedge \gamma, \gamma})^*w^{[\![\le \gamma]\!]}$.

Next, we claim that the restrictions of $(f^{[\![n]\!]})^*w$ and
$w^{[\![\le \gamma]\!]}, \gamma \in S$ to 
$\phi_{[n], [n]}(X^{[\![n]\!]}) \cap \W \phi_{\gamma, [n]}(X^{[\![\le \gamma]\!]})$ are equal. 
Note that $X^{[\![\le \gamma]\!]}$ is covered by the open subsets 
$\phi_{\beta, \gamma}(X^{[\![\beta]\!]}), \beta \le \gamma$, and  
$\phi_{[n], [n]}(X^{[\![n]\!]}) \cap \phi_{\beta, [n]}(X^{[\![\beta]\!]})$
is identified with the images of $X^{[\![n]\!]}_{\beta} \cong 
X^{[\![\beta]\!]}_{[n]}$. So it suffices to prove that
\begin{eqnarray}  \label{ExtHeis.2}
(f^{[\![n]\!]})^*w|_{X^{[\![n]\!]}_{\beta}} = (\zeta_{[n]}^\beta)^*
\big ((\phi_{\beta, \gamma})^*w^{[\![\le \gamma]\!]}|_{X^{[\![\beta]\!]}_{[n]}} \big ).
\end{eqnarray}
To see this, represent each $\alpha_i \in H^*(X)$ by a cycle $X_i$ such that 
$X_1, \ldots, X_k$ are in general position. By Proposition~\ref{prop_geom}, 
the class $w_{[n]} := w/\prod_{i=1}^t s_i!$ is represented by the closure $W$ of 
the subset consisting of elements of the form \eqref{prop_geom.1}.
Then, $(f^{[\![n]\!]})^*w_{[n]}$
is represented by $(f^{[\![n]\!]})^{-1}(W)$. By Proposition~\ref{prop_geom} again, 
the closure of $f^{[\![\beta]\!]} \Big (\zeta_{[n]}^\beta \big ((f^{[\![n]\!]})^{-1}(W) 
\cap X^{[\![n]\!]}_{\beta} \big ) \Big )$ in $X^{[\beta]}$ represents the class
$$
w_\beta := \sum_{w_1 \circ \cdots \circ w_l = w} \left (\prod_{i=1}^t 
\prod_{j=1}^l {1 \over s_{i, j}!} \right ) \cdot w_1 \otimes \cdots \otimes w_l
\in H^*(X^{[\beta]}) \cong \bigotimes_{i=1}^l H^*(X^{[\beta_i]})
$$
where each $w_j \in H^*(X^{[\beta_j]})$ contains exactly $s_{i, j}$ copies of 
$\frak a_{-i}(1_X)$. Note that $\sum_{j=1}^l s_{i, j} = s_i$.
Also, the class $(f^{[\![\beta]\!]})^*w_\beta \in H^*(X^{[\![\beta]\!]})$ is represented by 
the closure of $\zeta_{[n]}^\beta \big ((f^{[\![n]\!]})^{-1}(W) \cap 
X^{[\![n]\!]}_{\beta} \big )$ in $X^{[\![\beta]\!]}$. 
So $(f^{[\![n]\!]})^*w_{[n]}|_{X^{[\![n]\!]}_{\beta}} = (\zeta_{[n]}^\beta)^*
\big ((f^{[\![\beta]\!]})^*w_\beta|_{X^{[\![\beta]\!]}_{[n]}} \big ).$
Note that for fixed integers $s_{i, j}$, the number of choices of $w_1, \ldots, w_l$
satisfying $w_1 \circ \cdots \circ w_l = w$ is precisely equal to
$\prod_{i=1}^t s_i!/\prod_{i=1}^t \prod_{k=1}^l s_{i, j}!$. Therefore,
$(f^{[\![n]\!]})^*w|_{X^{[\![n]\!]}_{\beta}}$ is equal to
\begin{eqnarray}  \label{ExtHeis.3}
(\zeta_{[n]}^\beta)^*\big ((f^{[\![\beta]\!]})^*(\prod_{i=1}^t s_i! \cdot 
   w_\beta)|_{X^{[\![\beta]\!]}_{[n]}} \big )  
= (\zeta_{[n]}^\beta)^*\big ((f^{[\![\beta]\!]})^*\sum_{w_1 \circ \cdots \circ w_l = w} 
   w_1 \otimes \cdots \otimes w_l|_{X^{[\![\beta]\!]}_{[n]}} \big ).
\end{eqnarray}
On the other hand, since $\phi_{\beta, \gamma} = \W \phi_{\beta, \gamma} \circ 
\phi_{\beta, \beta}$, we obtain from \eqref{ExtHeis.1} that
\begin{eqnarray*}
   (\phi_{\beta, \gamma})^*w^{[\![\le \gamma]\!]}   
&=&(\phi_{\beta, \beta})^*\sum_{w_1 \circ \cdots \circ w_l = w} 
   \otimes_{i=1}^l w_i^{[\![\le |\beta_i|]\!]}    
   = \sum_{w_1 \circ \cdots \circ w_l = w} 
   \otimes_{i=1}^l (\phi_{\beta_i, \beta_i})^*w_i^{[\![\le |\beta_i|]\!]}   \\
&=&\sum_{w_1 \circ \cdots \circ w_l = w} 
   \otimes_{i=1}^l (f^{[\![|\beta_i|]\!]})^*w_i    
 = (f^{[\![\beta]\!]})^*\sum_{w_1 \circ \cdots \circ w_l = w} \otimes_{i=1}^l w_i
\end{eqnarray*}
where we have used induction in the third equality. Combining with \eqref{ExtHeis.3} 
verifies \eqref{ExtHeis.2}. 
\end{proof}

Our next lemma says that even though the extension $w^{[\![\le n]\!]} \in 
H^*(X^{[\![\le n]\!]})$ may not be unique, it does not affect the pairings with
$[\Theta^{[\![n, d]\!]}]$. Recall that the tautological map $\rho_{\alpha, \delta}: 
X^{[\![\alpha, \delta]\!]} \to X^{[\alpha, \delta]}$ is a finite map of degree $n!$.

\begin{lemma}   \label{IndepExt}
Let $A_1, A_2, A_3 \in H^*(\Xn)$ be Heisenberg monomial classes. Then, the pairing
$$
\left \langle [\Theta^{[\![n, d]\!]}], \,\, A_1^{[\![\le n]\!]} \otimes A_2^{[\![\le n]\!]} 
\otimes A_3^{[\![\le n]\!]} \right \rangle
$$
is independent of the choices of $A_1^{[\![\le n]\!]}, A_2^{[\![\le n]\!]}, A_3^{[\![\le n]\!]}$.
\end{lemma}
\begin{proof}
Since $[X^{[\![n, d]\!]}]^{\rm vir} = \rho_{[n], d}^*[X^{[n, d]}]^{\rm vir}$, $\left \langle 
[X^{[n,d]}]^{\rm vir}, \,\,\ev^*(A_1 \otimes A_2 \otimes A_3) \right \rangle$ is equal to
$$
{1 \over n!} \left \langle [X^{[\![n, d]\!]}]^{\rm vir}, \,\,
   \rho_{[n], d}^*\ev^*(A_1 \otimes A_2 \otimes A_3) \right \rangle    
= {1 \over n!} \left \langle [X^{[\![n, d]\!]}]^{\rm vir},  \,\,
   \ev^*\bigotimes_{i=1}^3 (f^{[\![n]\!]})^*A_i \right \rangle.
$$
By Lemma~\ref{ExtHeis} and (\ref{Identity-VirTheta}), $\left \langle [X^{[n,d]}]^{\rm vir},  \,\,
\ev^*(A_1 \otimes A_2 \otimes A_3) \right \rangle$ is equal to
\begin{eqnarray}    \label{IndepExt.1}
& &{1 \over n!} \left \langle [X^{[\![n, d]\!]}]^{\rm vir}, \,\,\ev^*\bigotimes_{i=1}^3
   (\phi_{[n], [n]})^*A_i^{[\![\le n]\!]} \right \rangle    \nonumber   \\     
%&=&{1 \over n!}[X^{[\![n, d]\!]}]^{\rm vir} \cdot \ev^*(\phi^3_{[n], [n]})^*
%   \bigotimes_{i=1}^3 A_i^{[\![\le n]\!]} \\  
&=&{1 \over n!} \left \langle (\phi^3_{[n], [n]})_*\ev_*[X^{[\![n, d]\!]}]^{\rm vir}, \,\,
   \bigotimes_{i=1}^3 A_i^{[\![\le n]\!]} \right \rangle       \nonumber   \\
&=&{1 \over n!} \sum_{(\alpha, \delta) \le ([n], d)} \left \langle
   (\W \phi^3_{\alpha, [n]})_*[\Theta^{[\![\alpha, \delta]\!]}], \,\, 
   \bigotimes_{i=1}^3 A_i^{[\![\le n]\!]} \right \rangle.    \qquad
\end{eqnarray}

Next, to prove the lemma, we use induction on $n$. When $n = 1$, 
the lemma is trivially true since $A_i^{[\![\le n]\!]} = A_i$. 
Assume that the lemma holds for all $X^{[m]}$ with $m < n$.
Let $(\alpha, \delta) < ([n], d)$. By Lemma~\ref{ProdTheta} and \eqref{ExtHeis.03}, 
$\left \langle (\W \phi^3_{\alpha, [n]})_*[\Theta^{[\![\alpha, \delta]\!]}], \,\,
A_1^{[\![\le n]\!]} \otimes A_2^{[\![\le n]\!]} \otimes A_3^{[\![\le n]\!]} \right \rangle$ 
is equal to
\begin{eqnarray}      \label{IndepExt.2}
%& &[\Theta^{[\![\alpha, \delta]\!]}] \cdot (\W \phi^3_{\alpha, [n]})^*
%   \big (A_1^{[\![\le n]\!]} \otimes A_2^{[\![\le n]\!]} \otimes A_3^{[\![\le n]\!]} \big ) \\
& &\left \langle [\Theta^{[\![\alpha, \delta]\!]}], \,\, 
   (\W \phi_{\alpha, [n]})^*A_1^{[\![\le n]\!]}
   \otimes (\W \phi_{\alpha, [n]})^*A_2^{[\![\le n]\!]} \otimes 
   (\W \phi_{\alpha, [n]})^*A_3^{[\![\le n]\!]} \right \rangle   \nonumber   \\
&=&\sum_{A_{1,1} \circ \cdots \circ A_{1,l}=A_1 
   \atop{A_{2,1} \circ \cdots \circ A_{2,l}=A_2 
   \atop A_{3,1} \circ \cdots \circ A_{3,l}=A_3}} \prod_{i=1}^l \left \langle
   [\Theta^{[\![\alpha_i, \delta_i]\!]}], \,\, 
   A_{1,i}^{[\![\le |\alpha_i|]\!]} \otimes A_{2,i}^{[\![\le |\alpha_i|]\!]} 
   \otimes A_{3,i}^{[\![\le |\alpha_i|]\!]} \right \rangle.
\end{eqnarray}
Now our lemma follows from \eqref{IndepExt.1} and induction.
\end{proof}

\begin{remark}   \label{IndepExt-Rmk}
Note that for $A\in H_k(W)$ and $B\in H^k(W)$ on a topological space $W$, the pairing $\langle A, B\rangle$ is the degree of the $0$-cycle $A\cap B\in H_0(W)$. 
%We can also regard all the pairings in the proof of Lemma~\ref{IndepExt} as cap products.
As $0$-cycles, $(\phi^3_{[n], [n]})_*\Big ( \ev_*[D(\bs_{[\![n,d]\!]})] \cap
((f^{[\![n]\!]})^*A_1 \otimes (f^{[\![n]\!]})^*A_2 \otimes 
(f^{[\![n]\!]})^*A_3) \Big )$ is equal to
$$
\sum_{(\alpha, \delta) \le ([n], d)} (\W \phi^3_{\alpha, [n]})_*
   \sum_{A_{1,1} \circ \cdots \circ A_{1,l}=A_1 
   \atop{A_{2,1} \circ \cdots \circ A_{2,l}=A_2 
   \atop A_{3,1} \circ \cdots \circ A_{3,l}=A_3}} \prod_{i=1}^l \left (
   [\Theta^{[\![\alpha_i, \delta_i]\!]}] \cap
   \big (A_{1,i}^{[\![\le |\alpha_i|]\!]} \otimes A_{2,i}^{[\![\le |\alpha_i|]\!]} 
   \otimes A_{3,i}^{[\![\le |\alpha_i|]\!]} \big ) \right ). 
$$
\end{remark}

Next, we extend the notation of Heisenberg monomial classes to a smooth family 
$Y \to T$ of quasi-projective surfaces. 

\begin{definition}   \label{HeisForY}
Fix positive integers $s_1, \ldots, s_t$ with $\sum_i i s_i = n$. Define 
\begin{eqnarray}      \label{HeisForY.0}
w^Y = \prod_{i=1}^t \frak a^Y_{-i}(1_X)^{s_i} \vac \quad \in H^*(Y_T^{[n]})
\end{eqnarray}
to be the cohomology class represented by the cycle $\prod_{i=1}^t s_i! \cdot [W] \in 
A_*(Y_T^{[n]})$ where $W \sub Y_T^{[n]}$ is the closure of the subset consisting of 
elements of the form
\begin{eqnarray*}   
\sum_{i=1}^t (\xi_{i, 1} + \ldots + \xi_{i, s_i}) \, \in (Y_u)^{[n]},
\quad u \in T
\end{eqnarray*}
where $\xi_{i, m} \in M_i(x_{i, m})$ for some $x_{i, m} \in Y_u$, 
and all the points $x_{i, m}$ are distinct. 
\end{definition} 

The following lemma is similar to Lemma~\ref{ExtHeis}, and its proof is omitted.

\begin{lemma}   \label{ExtHeisForY}
Let $w = w^Y$ be as in \eqref{HeisForY.0}.
Then there exists $w^{[\![\le n]\!]} \in H^*(Y_T^{[\![\le n]\!]})$ such that 
$(\phi_{[n], [n]})^*w^{[\![\le n]\!]} = (f^{[\![n]\!]})^*w$,
and that if $\beta = (\beta_1, \ldots, \beta_l) \le [n]$, then 
$(\W \phi_{\beta, [n]})^*w^{[\![\le n]\!]} = \sum_{w_1 \circ \cdots \circ w_l = w} 
\otimes_{i=1}^l w_i^{[\![\le |\beta_i|]\!]}$ via the identification 
$Y_T^{[\![\le \beta]\!]} = \prod_{i=1}^l Y_T^{[\![\le \beta_i]\!]}$.
\end{lemma}

\subsection{Normal slices and universal families}
\label{subsect_NormalSlices}
{This subsection mainly provides a necessary set-up for the proof of the universality result Lemma \ref{LiJ-PropB} in the next subsection.}

By Lemma~\ref{lma5.7-LiJ(PropA)}, we have
$
\Theta^{[\![\alpha, \delta]\!]} \subset (Y_T^{[\![\le \alpha]\!]})^3 
\times_{(Y_T^\Lambda)^3} \iota_\Lambda(\Delta_{\alpha, \epsilon}).
$
In this subsection, with $Y = X$ and $\alpha = [n]$, we will describe an analytic space,
independent of $\epsilon$, which contains $(X^{[\![\le n]\!]})^3 \times_{(X^n)^3} 
\iota_n(\Delta_{[n], \epsilon})$ whenever $\epsilon$ is sufficiently small.

To begin with, let $Y \to T$ be the total space of a rank-$2$ vector bundle, 
viewed as a smooth family of affine schemes. Define the fiber-wise averaging morphism 
$$
\mathfrak{av}: Y^{(n)}_T \to Y;\quad \sum m_i[x_i] \in Y_t^{(n)} \mapsto 
{1 \over n} \sum m_ix_i \in Y_t,\ \ t\in T.
$$
Here $\sum m_i x_i$ is the sum using the fiber-wise linear structure of $Y/T$.
Using $Y^n_T\to Y^{(n)}_T$ and $Y^{[n]}_T\to Y^{(n)}_T$, we obtain the averaging
maps $\mathfrak{av}: Y^n_T$ and $Y^{[n]}_T\to Y$. We define the relative Hilbert scheme of 
{\it centered $\alpha$-points} to be 
\beq\label{ave}
Y_{T, 0}^{[\![\alpha]\!]} = Y^{[\![\alpha]\!]}_T \times_{\mathfrak{av}, Y} 0_Y,
\eeq
where $0_Y\sub Y$ is the zero-section of $Y\to T$.

Next, like in \cite{LiJ}, we need to express an open neighborhood of the diagonal
$\Delta_{[2]} = \Delta_{[2]}^X \subset X \times X$ a vector bundle structure, 
using the first projection. As this is impossible in general, we will content to 
have a $C^\infty$-vector bundle structure. For this reason, we will again work 
with the analytic category. We will use differentiable map to mean a $C^\infty$-map; 
and an open subset will be open in analytic topology; we will use regular function 
and Zariski open subset to stand for their original meanings in algebraic geometry.

Consider the total space of the tangent bundle $T_X$, and its
zero-section $0_X \subset T_X$. For an open $\mathcal U\sub X\times X$, we view it as a 
space over $X$ via (that induced by the first projection)
$\pr_1|_{\mathcal U}: \mathcal U\to X$.
%where $\text{pr}_1: X\times X\to X$ is the first projection.
By Lemma~2.4 in \cite{LiJ}, there exists a diffeomorphism
\begin{eqnarray}  \label{Lma2.4-LiJ}
\varphi: \mathcal U \lra \mathcal V
\end{eqnarray}
of a tubular neighborhood $\mathcal U$ of $X_{[2]} \subset X \times X$ and 
a tubular neighborhood $\mathcal V$ of $0_X \subset T_X$, both considered as 
fiber bundles over $X$, such that 
\begin{enumerate}
\item[{\rm (A-i)}] restricting to each fiber $\mathcal U_x= \big ( \pr_1|_{\mathcal U} 
\big )^{-1}(x)$, the map $\varphi_x = \varphi|_{\mathcal U_x}: \mathcal U_x \to \mathcal V_x$ 
is a biholomorphism,

\item[{\rm (A-ii)}] $\varphi_x(x) = 0\in T_{X,x}$, and $d\varphi_x: T_{\mathcal U_x, x} 
\to T_{\mathcal V_x, 0}$ is the identity map.
\end{enumerate}

Since $\cV\sub T_X$ (over $X$), we define 
%$\cV^{[\![\alpha]\!]}_X$ be the open subset 
$$\cV^{[\![\alpha]\!]}_X=\{(\xi_1,\cdots,\xi_l)\in (T_X)^{[\![\alpha]\!]}_X\mid \Supp(\xi_i)\in \cV\}.
$$
For $\cU$ over $X$, we define $\cU^{[\![\alpha]\!]}_X=\coprod_{x\in X} (\cU_x)^{[\![\alpha]\!]}
$
endowed with the obvious smooth structure. By Lemma~2.5 in \cite{LiJ}, 
$\varphi$ induces a differentiable isomorphism
\begin{eqnarray}  \label{Lma2.5-LiJ}
\varphi^{[\![\alpha]\!]}: \mathcal U^{[\![\alpha]\!]}_X \lra \mathcal V^{[\![\alpha]\!]}_X
\end{eqnarray}
as stratified spaces. 
%(See \cite[Sect. 2]{LiJ} for 
%stratifications of singular spaces and smooth functions on stratified spaces).
Both $\cV^{[\![\alpha]\!]}_X$ and $\cU^{[\![\alpha]\!]}_X$ are bundles over $X$:
\beq\label{bundle}
\cV^{[\![\alpha]\!]}_X \lra X \and \cU^{[\![\alpha]\!]}_X \lra X.
\eeq
The first is induced by the bundle $\cV\sub T_X\to X$, and the second is via 
$(\cU_x)^{[\![\alpha]\!]} \mapsto \{x\}$. As $T_X\to X$ is a vector bundle, 
we obtain $(T_X)_{X, 0}^{[\![\alpha]\!]} \sub (T_X)_{X}^{[\![\alpha]\!]}$ as in \eqref{ave}. 
Let $\mathcal V^{[\![\alpha]\!]}_{X, 0} 
= \mathcal V^{[\![\alpha]\!]}_X \cap (T_X)_{X, 0}^{[\![\alpha]\!]},$
and let ${U}^{[\![\alpha]\!]} \subset X^{[\![\alpha]\!]}$ be the image of 
$\mathcal V^{[\![\alpha]\!]}_{X,0}$ under the composition
$$
\varrho_{\alpha}: \quad \mathcal V^{[\![\alpha]\!]}_{X,0} \mapright{\sub} \mathcal V^{[\![\alpha]\!]}_X 
\cong {\cU}^{[\![\alpha]\!]}_X \lra X^{[\![\alpha]\!]} \times X \mapright{\pr_1} X^{[\![\alpha]\!]},
$$
where the first factor of ${\cU}^{[\![\alpha]\!]}_X \to X^{[\![\alpha]\!]} \times X$ is
induced by the inclusion $(\cU_x)^{[\![\alpha]\!]}\sub X^{[\![\alpha]\!]}$, and the second 
is \eqref{bundle}. By the Lemma~2.6 and Lemma~2.7 in \cite{LiJ}, 
after shrinking $\mathcal V$ if necessary, ${U}^{[\![\alpha]\!]}$ is an open neighborhood 
of $X^{[\![\alpha]\!]} \times_{X^n} \Delta_{[n]} \sub X^{[\![\alpha]\!]}$, and 
\begin{eqnarray}  \label{VarrhoIso}
\varrho_{\alpha}: \mathcal V^{[\![\alpha]\!]}_{X, 0} \lra U^{[\![\alpha]\!]}
\end{eqnarray}
is a smooth isomorphism of stratified spaces fibered over $\Delta_{[n]}$, via the map
$$
U^{[\![\alpha]\!]} \sub {\cU}^{[\![\alpha]\!]}_X \lra X^{[\![\alpha]\!]} \times X \mapright{\pr_2} X,
$$
and preserves the partial equivalences of 
$\mathcal V^{[\![\alpha]\!]}_{X, 0}$ and ${U}^{[\![\alpha]\!]}$. Note that 
\begin{eqnarray}  \label{TopDiagDisJoin}
X^{[\![\le n]\!]} \times_{X^n} \Delta_{[n]} = 
\coprod_{\alpha \le [n]} \phi_{[\alpha], [n]}(X^{[\![\alpha]\!]} \times_{X^n} \Delta_{[n]}).
\end{eqnarray}
So $U^{[\![\le n]\!]} := \cup_{\alpha \le [n]} \phi_{[\alpha], [n]}(U^{[\![\alpha]\!]})$ 
is an open neighborhood of $X^{[\![\le n]\!]} \times_{X^n} \Delta_{[n]}$ in 
$X^{[\![\le n]\!]}$. Since $\epsilon$ is sufficiently small, 
$X^{[\![\le n]\!]} \times_{X^n} \Delta_{[n], \epsilon} \sub U^{[\![\le n]\!]}$. Thus,
\begin{eqnarray}  \label{EpsilonCont1}
(X^{[\![\le n]\!]})^3 \times_{(X^n)^3} \iota_n(\Delta_{[n], \epsilon}) 
\sub (U^{[\![\le n]\!]})^3
\end{eqnarray}
noting that by our convention, $(U^{[\![\le n]\!]})^3$ is a fibered product over $\Delta_{[n]}$.
Since $\mathcal V^{[\![\alpha]\!]}_{X, 0} \sub (T_X)_{X, 0}^{[\![\alpha]\!]}$,
we put $\mathcal V^{[\![\le n]\!]}_{X, 0} = \cup_{\alpha \le [n]} 
\phi_{[\alpha], [n]}(\mathcal V_{X, 0} ^{[\![\alpha]\!]}) \sub (T_X)_{X, 0}^{[\![\le n]\!]}$.
Then, the smooth isomorphisms $\varrho_{\alpha}$ from \eqref{VarrhoIso} induces
a smooth isomorphism 
\begin{eqnarray}  \label{EpsilonCont2}
\varrho_{[\![\le n]\!]}: \mathcal V^{[\![\le n]\!]}_{X, 0} \to U^{[\![\le n]\!]}
\end{eqnarray}
of stratified spaces fibered over $X \cong \Delta_{[n]}$. Combining with \eqref{EpsilonCont1}, 
we have 
\begin{eqnarray}  \label{EpsilonCont3}
((T_X)_{X, 0}^{[\![\le n]\!]})^3 \,\, \supset \,\, (\mathcal V^{[\![\le n]\!]}_{X, 0})^3 \,\, 
\overset{\varrho_{[\![\le n]\!]}^3} \lra \,\, (U^{[\![\le n]\!]})^3 \,\, \supset \,\, 
(X^{[\![\le n]\!]})^3 \times_{(X^n)^3} \iota_n(\Delta_{[n], \epsilon}).
\end{eqnarray}

To prove universality results later on, we pick a differentiable map
\begin{eqnarray}  \label{SmoothG}
g: X \lra Gr = Gr(2, \C^N)
\end{eqnarray}
with $N \gg 0$ so that $T_X \cong g^*F$ as smooth vector bundles, where $F \to Gr$ is 
the total space of the universal quotient rank-$2$ bundle over $Gr$. 
Let $F^{[\![\alpha_i]\!]}_{Gr, 0} \to Gr$ be the associated relative Hilbert scheme 
of centered $\alpha_i$-points. By Lemma~2.8 in \cite{LiJ}, $g$ induces isomorphisms 
(as stratified spaces) of fiber bundles over $X$:
$$
g^{\alpha_i}: (T_X)^{\alpha_i}_{X,0} \to  g^*F^{\alpha_i}_{Gr, 0} \and
g^{[\![\le n]\!]}: (T_X)^{[\![\le n]\!]}_{X,0} \to g^*F^{[\![\le n]\!]}_{Gr, 0}.
$$
\subsection{Pairings with $[\Theta^{[\![n, d]\!]}]$ when $d > 0$}
\label{subsect_IntWith}
 {Here we study more properties of  $ [\Theta^{[\![n, d]\!]}]$ which are the
crucial ingredients for the universality results of extremal Gromov-Witten invariants of $X^{n]}$.}
%By \eqref{IndepExt.2}, pairings with the classes $(\W \phi^3_{\alpha, [n]})_*
%[\Theta^{[\![\alpha, \delta]\!]}]$ can be reduced to pairings with the classes 
%$[\Theta^{[\![n, d]\!]}]$. In this subsection, we study pairings with  
%$[\Theta^{[\![n, d]\!]}]$ when $d > 0$. 

\begin{convention}    \label{Conven-Int}
Fix $d>0$ and Heisenberg monomial classes 
\begin{eqnarray}  \label{AssumpOnAi}
A_i = \mathfrak a_{-\la^{(i)}}(1_X) \mathfrak a_{-n_{i,1}}(\alpha_{i,1}) \cdots 
\mathfrak a_{-n_{i,u_i}}(\alpha_{i,u_i}) \vac \quad \in H^*(\Xn)
\end{eqnarray}
where $1 \le i \le 3$, $u_i \ge 0$, and $|\alpha_{i,j}| > 0$. When $|\alpha_{i,j}| = 4$, we let 
$\alpha_{i,j} = x$ (the cohomology class of a point). Moreover, if $|\alpha_{i,j}| = 2$, 
then $\alpha_{i,j}$ can be represented by a Riemann surface intersecting transversally 
with $D_0 \cup D_\infty$. For simplicity, put $A^{[\![\le n]\!]} 
= A_1^{[\![\le n]\!]} \otimes A_2^{[\![\le n]\!]} \otimes A_3^{[\![\le n]\!]}$.
\end{convention}

Our goal is to understand the pairing $\left \langle [\Theta^{[\![n, d]\!]}], \,\,
A^{[\![\le n]\!]} \right \rangle$ when $d > 0$. 

\begin{lemma}   \label{ThetaVanish}
Fix $d > 0$. Then, $\left \langle [\Theta^{[\![n, d]\!]}], \,\,
A^{[\![\le n]\!]} \right \rangle = 0$ if one of the following holds:
\begin{enumerate}
\item[{\rm (i)}] $|\alpha_{i,j}| = 4$ for some $(i, j)$;

\item[{\rm (ii)}] $|\alpha_{i,j}| = 2$ for two different pairs $(i, j)$. 
\end{enumerate}
\end{lemma}
\begin{proof}
(i) We begin with $d \ge 0$. Consider the $0$-cycle  $
 [\Theta^{[\![n, d]\!]}_{\rm loc}] \cap A^{[\![\le n]\!]}$  in 
$(X^{[\![\le n]\!]})^3$. Choose the point representation $x \in X$ of $\alpha_{i,j}$ 
such that $x \not \in D_0 \cup D_\infty$. By Proposition~\ref{prop_geom},
$A_i$ can be represented by a cycle $W_i \subset \Xn$ such that $x \in \Supp(\xi_1)$
for every $\xi_1 \in W_i$. Thus for every $\xi_2$ contained in the 
$0$-cycle $(\phi^3_{[n], [n]})_*\Big ( \ev_*[D(\bs_{[\![n,d]\!]})_{{\rm loc}}] \cap 
((f^{[\![n]\!]})^*A_1 \otimes (f^{[\![n]\!]})^*A_2 \otimes (f^{[\![n]\!]})^*A_3) \Big )$,
the point $x$ is a component of $\ind(\xi_2)$ where $\ind$ is from \eqref{IndexMap}. 
By the localized version of Remark~\ref{IndepExt-Rmk} and induction, 
we conclude that $x$ is a component of $\ind(\xi)$ if $\xi$ is contained in
$[\Theta^{[\![n, d]\!]}_{\rm loc}] \cap A^{[\![\le n]\!]}$.

Now let $d > 0$. By the localized version of Lemma~\ref{lma5.7-LiJ(PropA)}, we have
$$
\Theta^{[\![n, d]\!]}_{\rm loc} \subset \Big ( (X^{[\![\le n]\!]})^3 
\times_{(X^\Lambda)^3} \iota_\Lambda(\Delta_{[n], \epsilon}) \Big ) \cap 
\bigcup_{(\beta, \eta) \le ([n], d)} \phi^3_{\beta, [n]}\ev(X^{[\![\beta, \eta]\!]}_{\theta}).
$$
Thus, since $d > 0$, if $\xi \in \Theta^{[\![n, d]\!]}_{\rm loc}$, 
then $\ind(\xi) \in \Delta_{[n], \epsilon}$ and $y \in D_0 \cup D_\infty$ for
some component $y$ of $\ind(\xi)$. Since $\epsilon$ is sufficiently small,
we see from the previous paragraph that $[\Theta^{[\![n, d]\!]}_{\rm loc}] \cap
A^{[\![\le n]\!]}$ is empty. Hence as pairings, $\langle [\Theta^{[\![n, d]\!]}] , \,\,
A^{[\![\le n]\!]} \rangle= \langle [\Theta^{[\![n, d]\!]}_{\rm loc}], \,\,A^{[\![\le n]\!]}\rangle  = 0$.

(ii) Let $|\alpha_{i_1,j_1}| = |\alpha_{i_2,j_2}| = 2$ where $(i_1,j_1) \ne (i_2,j_2)$.
Represent $\alpha_{i_1,j_1}$ and $\alpha_{i_2,j_2}$ by Riemann surfaces
$C_{i_1,j_1}$ and $C_{i_2,j_2}$ respectively such that 
$C_{i_1,j_1}$, $C_{i_2,j_2}$ and $D_0 \cup D_\infty$ are in general position. 
As in the proof of (i), we see that if 
$\xi \in [\Theta^{[\![n, d]\!]}_{\rm loc}] \cap A^{[\![\le n]\!]}$, 
then $\ind(\xi) \in \Delta_{[n], \epsilon}$ and the components of $\ind(\xi)$ contain 
three points $x_1 \in C_{i_1,j_1}$, $x_2 \in C_{i_2,j_2}$ and 
$x_3 \in D_0 \cup D_\infty$. This is impossible since $\epsilon$ is sufficiently small
and $C_{i_1,j_1}$, $C_{i_2,j_2}$, $D_0 \cup D_\infty$ are in general position.
So the $0$-cycle $[\Theta^{[\![n, d]\!]}_{\rm loc}] \cap A^{[\![\le n]\!]}$ is empty.
\end{proof}

\begin{lemma}   \label{1CurveClass}
Let $u_1 = 1$, $u_2 = u_3 = 0$, and $|\alpha_{1,1}|=2$ in \eqref{AssumpOnAi}. Then,
$\langle [\Theta^{[\![n, d]\!]}] , \,\,A^{[\![\le n]\!]}\rangle = p \cdot \langle K_X, \alpha_{1, 1} \rangle$
where $p$ is a constant depending only on $n_{1,1}$ and the partitions $\la^{(i)}$.
\end{lemma}
\begin{proof}
Represent $\alpha_{1,1}$ by a Riemann surface $C_{1, 1}$ intersecting transversally with 
$D_0 \cup D_\infty$. Let
\begin{eqnarray*}
C_{1,1} \cap D_0 
&=& \{x_1, \ldots, x_{s_+}, x_{s_++1}, \ldots, x_{s_+ + s_-} \},     \\
C_{1,1} \cap D_\infty 
&=& \{x_{s_+ + s_-+1}, \ldots, x_{s_+ + s_- + t_+}, 
    x_{s_+ + s_- + t_+ +1}, \ldots, x_{s_+ + s_- + t_+ + t_-}\}
\end{eqnarray*}
so that the points $x_1, \ldots, x_{s_+ + s_- + t_+ + t_-}$ are distinct, 
the intersection of $C_{1,1}$ and $D_0$ at $x_i$ for $1 \le i \le s_+$ 
(respectively, for $s_+ + 1 \le i \le s_+ + s_-$) is equal to $1$ (respectively, $-1$),
and the intersection of $C_{1,1}$ and $D_\infty$ at $x_i$ for $s_+ + s_- + 1 \le i \le 
s_+ + s_-+ t_+$ (respectively, for $s_+ + s_- + t_+ + 1 \le i \le s_+ + s_- + t_+ + t_-$) 
is equal to $1$ (respectively, $-1$). So $s_+ - s_- = \langle D_0, \alpha_{1, 1} \rangle$ 
and $t_+ - t_- = \langle D_\infty, \alpha_{1, 1} \rangle$.
Let $x_i \in X_i$ be a small analytic open neighborhood of $x_i$ such that 
$X_1, \ldots, X_{s_+ + s_- + t_+ + t_-}$ are mutually disjoint. 
As in the proof of Lemma~\ref{ThetaVanish}~(i), 
we see that the $0$-cycle $[\Theta^{[\![n, d]\!]}_{\rm loc}] \cap A^{[\![\le n]\!]}$
is a disjoint union of $W_1, \ldots, W_{s_+ + s_- + t_+ + t_-}$ such that $\ind(W_i) 
\sub (X_i)^n$ for every $i$. Let $e_i$ be the contribution of each $W_i$ to the pairing 
$\langle [\Theta^{[\![n, d]\!]}_{\rm loc}], \,\, A^{[\![\le n]\!]}\rangle$. Then, 
$$\langle [\Theta^{[\![n, d]\!]}] , \,\,A^{[\![\le n]\!]}\rangle = \langle [\Theta^{[\![n, d]\!]}_{\rm loc}],  \,\,A^{[\![\le n]\!]}\rangle
 = e_1 + \ldots + e_{s_+ + s_- + t_+ + t_-}.$$
As in the proof of Lemma~4.3 in \cite{LL}, we conclude that each $e_i$ can be computed
from $X_i$ so that $e_1 = \ldots = e_{s_+} = -e_{s_+ + 1} = \ldots = -e_{s_+ + s_-}$ 
and $e_{s_+ + s_- +1} = \ldots = e_{s_+ + s_- + t_+}
= -e_{s_+ + s_- + t_+ + 1} = \ldots = -e_{s_+ + s_- + t_+ + t_-}$ depend only on 
$n_{1,1}$ and the partitions $\la^{(i)}$. Since $D_0 = K_X + D_\infty$,
\begin{eqnarray}     \label{1CurveClass.1}
& &\langle [\Theta^{[\![n, d]\!]}],  \,\,A^{[\![\le n]\!]} \rangle
  = (s_+ - s_-) e_1 + (t_+ - t_-) e_{s_+ + s_- +1}   \nonumber   \\
&=&e_1  \cdot \langle D_0, \alpha_{1, 1} \rangle 
    + e_{s_+ + s_- +1} \cdot \langle D_\infty, \alpha_{1, 1} \rangle    
= p \cdot \langle K_X, \alpha_{1, 1} \rangle 
    + p' \cdot \langle D_\infty, \alpha_{1, 1} \rangle
\end{eqnarray}
where $p = e_1$ and $p' = e_1 + e_{s_+ + s_- +1}$. Note that for $m \gg 0$, there exists 
a meromorphic section $\theta_m$ of $\mathcal O_X(K_X)$ such that $mD_\infty$ 
is the pole divisor of $\theta_m$. By (\ref{1CurveClass.1}), 
\begin{eqnarray*}    
  \langle [\Theta^{[\![n, d]\!]}] ,\,\,A^{[\![\le n]\!]} \rangle
= p \cdot \langle K_X, \alpha_{1, 1} \rangle 
    + p' \cdot \langle m D_\infty, \alpha_{1, 1} \rangle
\end{eqnarray*}
for all $m \gg 0$. It follows that $p' = 0$ and $\langle [\Theta^{[\![n, d]\!]}],\,\,
A^{[\![\le n]\!]} \rangle= p \cdot \langle K_X, \alpha_{1, 1} \rangle$.
\end{proof}

\begin{lemma}   \label{LiJ-PropB}
Let $d > 0$ and $A_i = \mathfrak a_{-\la^{(i)}}(1_X) \vac$ for $i \in \{1,2,3\}$. Then,
$
\langle [\Theta^{[\![n, d]\!]}] , \,\,A^{[\![\le n]\!]}\rangle = p \cdot \langle K_X, K_X \rangle
$
where the coefficient $p$ is a constant depending only on the partitions $\la^{(i)}$.
\end{lemma}
\begin{proof}
By Lemma~\ref{lma5.7-LiJ(PropA)}, $\Theta^{[\![n, d]\!]} \subset (X^{[\![\le n]\!]})^3 
\times_{(X^n)^3} \iota_n(\Delta_{[n], \epsilon})$.
Using \eqref{EpsilonCont3} and the smooth isomorphism \eqref{EpsilonCont2},
we transport the $0$-cycle $[\Theta^{[\![n, d]\!]}] \cap A^{[\![\le n]\!]}$
in $(X^{[\![\le n]\!]})^3 \times_{(X^n)^3} \iota_n(\Delta_{[n], \epsilon})$
to the following $0$-cycle in $(\mathcal V^{[\![\le n]\!]}_{X, 0})^3 \sub
((T_X)_{X, 0}^{[\![\le n]\!]})^3$:
$$
(\varrho_{[\![\le n]\!]}^3)^*[\Theta^{[\![n, d]\!]}] \cap
(\varrho_{[\![\le n]\!]}^3)^*\Big (A^{[\![\le n]\!]}|_{(X^{[\![\le n]\!]})^3 
\times_{(X^n)^3} \, \iota_n(\Delta_{[n], \epsilon})} \Big ).
$$
Note that these two $0$-cycles have the same degree. So as pairings,
\begin{eqnarray}  \label{LiJ-PropB.1}
\langle [\Theta^{[\![n, d]\!]}], \,\,A^{[\![\le n]\!]} \rangle
= \left\langle (\varrho_{[\![\le n]\!]}^3)^*[\Theta^{[\![n, d]\!]}], \,\,
(\varrho_{[\![\le n]\!]}^3)^*\Big (A^{[\![\le n]\!]}|_{(X^{[\![\le n]\!]})^3 
\times_{(X^n)^3} \, \iota_n(\Delta_{[n], \epsilon})} \Big )\right\rangle.
\end{eqnarray}

Let $g$ from \eqref{SmoothG} be generic, and let $F \to Gr$ be the total space of 
the universal quotient rank-$2$ bundle over $Gr = Gr(2, \C^N)$.
Let ${\overline T}_X \to X$ and ${\overline F} \to Gr$ be the projectifications of
$T_X \to X$ and $F \to Gr$ respectively. Then the differentiable isomorphism $T_X \cong g^*F$ 
induces a differentiable isomorphism ${\overline T}_X \cong g^*{\overline F}$.
Note that the top diagonal $\Delta_{[n]}^{F, 0} := \Delta_{[n]}^{F} \cap F^n_{Gr, 0}$
in $F^n_{Gr, 0}$ is the $0$-section of $F^n_{Gr, 0} \to Gr$. 
Put $\Delta_{[n], \epsilon}^{F, 0} = \Delta_{[n], \epsilon}^{F} \cap F^n_{Gr, 0}$.
Applying the previous constructions to the families ${\overline F} \to Gr$ and
${\overline T}_X \to X$ and adopting the proof of Lemma~6.1 in \cite{LiJ},
we conclude that there exists a cycle $\Theta^{[\![n, d]\!]}_F \sub 
(F_{Gr, 0}^{[\![\le n]\!]})^3 \times_{(F_{Gr, 0}^n)^3} \iota_n(\Delta_{[n], \epsilon}^{F, 0})$
such that
\begin{eqnarray}  \label{LiJ-PropB.100}
[\Theta^{[\![n, d]\!]}_F] \in H_*\Big ((F_{Gr, 0}^{[\![\le n]\!]})^3 
\times_{(F_{Gr, 0}^n)^3} \iota_n(\Delta_{[n]}^{F, 0}) \Big ),
\end{eqnarray}
the intersection $\Theta^{[\![n, d]\!]}_F \cap \Big (\big ( (F_{Gr, 0}^{[\![\le n]\!]})^3 
\times_{(F_{Gr, 0}^n)^3} \iota_n(\Delta_{[n], \epsilon}^{F, 0}) \big ) \times_{Gr} X \Big )$
is  transversal, and 
\begin{eqnarray}  \label{LiJ-PropB.2}
(\varrho_{[\![\le n]\!]}^3)^{-1}(\Theta^{[\![n, d]\!]})
= \Theta^{[\![n, d]\!]}_F \cap \Big (\big ( (F_{Gr, 0}^{[\![\le n]\!]})^3 
\times_{(F_{Gr, 0}^n)^3} \iota_n(\Delta_{[n], \epsilon}^{F, 0}) \big ) \times_{Gr} X \Big )
\end{eqnarray}
via $((T_X)_{X, 0}^{[\![\le n]\!]})^3 \times_{((T_X)_{X, 0}^n)^3} 
\iota_n(\Delta_{[n], \epsilon}^{T_X, 0}) \cong \big ( (F_{Gr, 0}^{[\![\le n]\!]})^3 
\times_{(F_{Gr, 0}^n)^3} \iota_n(\Delta_{[n], \epsilon}^{F, 0}) \big ) \times_{Gr} X$. Thus,
\begin{eqnarray}  \label{LiJ-PropB.200}
   (\varrho_{[\![\le n]\!]}^3)^*[\Theta^{[\![n, d]\!]}]
= [(\varrho_{[\![\le n]\!]}^3)^{-1}(\Theta^{[\![n, d]\!]})] 
\end{eqnarray}
is a homology class supported on $((T_X)_{X, 0}^{[\![\le n]\!]})^3 
\times_{((T_X)_{X, 0}^n)^3} \iota_n(\Delta_{[n]}^{T_X, 0})$.

Let $A_i^{T_X} = \mathfrak a_{-\la^{(i)}}^{T_X}(1_X) \vac \in H^*((T_X)_X^{[n]})$
be defined in Definition~\ref{HeisForY}, and put 
$$
(A^{T_X, 0})^{[\![\le n]\!]} = (A_1^{T_X, 0})^{[\![\le n]\!]} 
\otimes (A_2^{T_X, 0})^{[\![\le n]\!]} \otimes (A_3^{T_X, 0})^{[\![\le n]\!]}
$$
where $(A_i^{T_X, 0})^{[\![\le n]\!]} \in H^*((T_X)_{X, 0}^{[\![\le n]\!]})$
is the pull-back of $(A_i^{T_X})^{[\![\le n]\!]} \in H^*((T_X)_{X}^{[\![\le n]\!]})$
via the inclusion $(T_X)_{X, 0}^{[\![\le n]\!]} \sub (T_X)_{X}^{[\![\le n]\!]}$.
Let $S$ denote the intersection
$$
\Big (((T_X)_{X, 0}^{[\![\le n]\!]})^3 \times_{((T_X)_{X, 0}^n)^3} 
\iota_n(\Delta_{[n]}^{T_X, 0}) \Big ) \cap (\varrho_{[\![\le n]\!]}^3)^{-1}
((X^{[\![\le n]\!]})^3 \times_{(X^n)^3} \, \iota_n(\Delta_{[n], \epsilon})).
$$
Then, $S = (\varrho_{[\![\le n]\!]}^3)^{-1}
((X^{[\![\le n]\!]})^3 \times_{(X^n)^3} \, \iota_n(\Delta_{[n]}))$. We claim that 
\begin{eqnarray}  \label{LiJ-PropB.3}
(\varrho_{[\![\le n]\!]}^3)^*\Big (A^{[\![\le n]\!]}|_{(X^{[\![\le n]\!]})^3 
\times_{(X^n)^3} \, \iota_n(\Delta_{[n], \epsilon})} \Big )|_S 
\, = \, (A^{T_X, 0})^{[\![\le n]\!]}|_S,
\end{eqnarray}
i.e., $(\varrho_{[\![\le n]\!]}^3)^*\Big (A^{[\![\le n]\!]}|_{(X^{[\![\le n]\!]})^3 
\times_{(X^n)^3} \, \iota_n(\Delta_{[n]})} \Big ) \, = \, (A^{T_X, 0})^{[\![\le n]\!]}|_S$.
It suffices to prove that 
\begin{eqnarray}  \label{LiJ-PropB.4}
\varrho_{[\![\le n]\!]}^*\Big (A_i^{[\![\le n]\!]}|_{X^{[\![\le n]\!]} 
\times_{X^n} \, \Delta_{[n]}} \Big ) 
\, = \, (A_i^{T_X, 0})^{[\![\le n]\!]}|_{\varrho_{[\![\le n]\!]}^{-1} 
(X^{[\![\le n]\!]} \times_{X^n} \, \Delta_{[n]})}.
\end{eqnarray}
Indeed, for every $\alpha \le [n]$, we conclude from Lemma~\ref{ExtHeis} and 
Lemma~\ref{ExtHeisForY} that the same subvariety in 
$\varrho_{[\![\le n]\!]}^{-1}\big ({\phi_{[\alpha], [n]}
(X^{[\![\alpha]\!]} \times_{X^n} \Delta_{[n]})} \big )$ represents the cohomology classes
$$
\varrho_{[\![\le n]\!]}^*\Big (A_i^{[\![\le n]\!]}|_{\phi_{[\alpha], [n]}
(X^{[\![\alpha]\!]} \times_{X^n} \Delta_{[n]})} \Big ) \and
(A_i^{T_X, 0})^{[\![\le n]\!]}|_{\varrho_{[\![\le n]\!]}^{-1}\big ({\phi_{[\alpha], [n]}
(X^{[\![\alpha]\!]} \times_{X^n} \Delta_{[n]})} \big )}.
$$
Since $X^{[\![\le n]\!]} \times_{X^n} \, \Delta_{[n]} = \coprod_{\alpha \le [n]} 
\phi_{[\alpha], [n]}(X^{[\![\alpha]\!]} \times_{X^n} \Delta_{[n]})$,
we obtain \eqref{LiJ-PropB.4}.

By \eqref{LiJ-PropB.1}, \eqref{LiJ-PropB.3} and \eqref{LiJ-PropB.200}, as pairings, we have
\begin{eqnarray*} 
\langle [\Theta^{[\![n, d]\!]}] , \,\,A^{[\![\le n]\!]}\rangle
= \langle (\varrho_{[\![\le n]\!]}^3)^*[\Theta^{[\![n, d]\!]}], \,\,
   (A^{T_X, 0})^{[\![\le n]\!]}|_S  \rangle            
= \langle [(\varrho_{[\![\le n]\!]}^3)^{-1}(\Theta^{[\![n, d]\!]})], \,\,
   (A^{T_X, 0})^{[\![\le n]\!]}\rangle.
\end{eqnarray*} 
Let $(g^{[\![\le n]\!]})^3: ((T_X)_{X, 0}^{[\![\le n]\!]})^3 \to 
(F_{Gr, 0}^{[\![\le n]\!]})^3 \times_{Gr} X$ be the isomorphism induced by $g$.
By Lemma~\ref{ExtHeisForY}, $(A^{T_X, 0})^{[\![\le n]\!]}$ can be taken to be
$(g^{[\![\le n]\!]})^{3*}((A^{F, 0})^{[\![\le n]\!]}|_{(F_{Gr, 0}^{[\![\le n]\!]})^3 
\times_{Gr} X})$. So
$$
\langle [\Theta^{[\![n, d]\!]}], \,\,A^{[\![\le n]\!]} \rangle 
= \left\langle [(\varrho_{[\![\le n]\!]}^3)^{-1}(\Theta^{[\![n, d]\!]})],\,\,
(g^{[\![\le n]\!]})^{3*}\Big ((A^{F, 0})^{[\![\le n]\!]}|_{
  (F_{Gr, 0}^{[\![\le n]\!]})^3 \times_{Gr} X} \Big )\right\rangle.
$$
Combining with \eqref{LiJ-PropB.2} and putting $W_\epsilon = (F_{Gr, 0}^{[\![\le n]\!]})^3 
\times_{(F_{Gr, 0}^n)^3} \iota_n(\Delta_{[n], \epsilon}^{F, 0})$, we get
{\begin{eqnarray*} 
  \langle [\Theta^{[\![n, d]\!]}] ,  \,\,A^{[\![\le n]\!]} \rangle
&=&\left \langle [(\varrho_{[\![\le n]\!]}^3)^{-1}(\Theta^{[\![n, d]\!]})],\,\,
 (g_\epsilon^{[\![\le n]\!]})^{3*}
  \Big ((A^{F, 0})^{[\![\le n]\!]}|_{W_\epsilon} \Big ) \right\rangle  \\
&=&\left \langle((g_\epsilon^{[\![\le n]\!]})^{3})_*
  [(\varrho_{[\![\le n]\!]}^3)^{-1}(\Theta^{[\![n, d]\!]})],\,\,
  (A^{F, 0})^{[\![\le n]\!]}|_{W_\epsilon}   \right\rangle
\end{eqnarray*} 
where $(g_\epsilon^{[\![\le n]\!]})^{3}: ((T_X)_{X, 0}^{[\![\le n]\!]})^3 
\times_{((T_X)_{X, 0}^n)^3} \iota_n(\Delta_{[n], \epsilon}^{T_X, 0}) \to W_\epsilon$ 
is the morphism induced by $g$.
% and $\pi_\epsilon: W_\epsilon \to Gr$  is the tautological projection. 
By \eqref{LiJ-PropB.100}, $[\Theta_F^{[\![n, d]\!]}]$
is supported on $W := (F_{Gr, 0}^{[\![\le n]\!]})^3 \times_{(F_{Gr, 0}^n)^3} 
\iota_n(\Delta_{[n]}^{F, 0})$.  
Therefore, by (\ref{LiJ-PropB.2}), we obtain
\begin{eqnarray*} 
& &\langle[\Theta^{[\![n, d]\!]}],  \,\,A^{[\![\le n]\!]}  \rangle
     =\langle[\Theta_F^{[\![n, d]\!]}]\cap \pi^*( \text {PD}^{-1}[g(X)]  ) ,\,\,
     (A^{F, 0})^{[\![\le n]\!]}|_{W} \rangle  \\
&=&\langle[\Theta_F^{[\![n, d]\!]}] \cap (A^{F, 0})^{[\![\le n]\!]}|_{W} , \,\,
      \pi^*( \text {PD}^{-1}[g(X)]  ) \rangle= \left \langle\pi_*\Big ([\Theta_F^{[\![n, d]\!]}] \cap
     (A^{F, 0})^{[\![\le n]\!]}|_{W} \Big ), \,\,\text {PD}^{-1}[g(X)]  \right \rangle
\end{eqnarray*}
where $\pi: W \to Gr$ is the tautological projection.
Here a bit of topological argument is needed for the first equality. Observe  that $W$ is in fact a disjoint union 
of Hausdorff spaces. Then we use the Zariski local triviality of the bundle $F$ over $Gr$, the K\"unneth decomposition for Borel-Moore homology \cite{Iv}, and the properties of cap products \cite{GH, Iv, Sp}. }

 The Poincar\' e dual of 
$\pi_*\Big ([\Theta_F^{[\![n, d]\!]}] \cap (A^{F, 0})^{[\![\le n]\!]}|_{W} \Big )$
is a polynomial $P$ in the Chern classes $c_i(F)$. Hence, 
\begin{eqnarray}     \label{LiJ-PropB.5}
   \langle [\Theta^{[\![n, d]\!]}] , \,\,A^{[\![\le n]\!]} \rangle
= p \cdot \langle K_X, K_X \rangle + q \cdot \deg (e_X)
\end{eqnarray}
where $p$ and $q$ are constants depending only on the partitions $\la^{(i)}$.
%and $\chi(X) = \deg (e_X)$ is the Euler characteristic of $X$. 
%Moreover, as in the proof of Lemma~6.1 in \cite{LiJ}, $p$ and $q$ do not 
%depend on the integer $N$ coming from \eqref{SmoothG}.

Finally, it remains to prove that $q=0$ in \eqref{LiJ-PropB.5}.
To see this, choose the surface $X$ such that $|K_X|$ contains a smooth divisor $D$.
Let $\theta$ be a holomorphic section of $\mathcal O_X(K_X)$ such that 
the vanishing divisor of $\theta$ is $D=D_0$. By \eqref{LiJ-PropB.200},
$(\varrho_{[\![\le n]\!]}^3)^*[\Theta^{[\![n, d]\!]}_{\rm loc}]$
is a homology class supported on 
$\big ((T_X|_D)_{D, 0}^{[\![\le n]\!]} \big )^3 
\times_{((T_X|_D)_{D, 0}^n)^3} \iota_n(\Delta_{[n]}^{T_X|_D, 0})$.
Repeating the above argument and replacing $g: X \to Gr$ (respectively, $T_X \to X$)
by $g|_{D}: D \to Gr$ (respectively, $T_X|_D \to D$), we get
$$
\langle [\Theta^{[\![n, d]\!]}] , \,\,A^{[\![\le n]\!]} \rangle= p' \cdot \langle K_X, K_X \rangle 
= p \cdot \langle K_X, K_X \rangle + q \cdot \deg (e_X)
$$ 
where $p'$ depends only on the partitions $\la^{(i)}$. Since there exist two surfaces $X$ with smooth $D \in |K_X|$ such that the pairs 
$(\langle K_X, K_X \rangle, \deg (e_X))$ are linearly independent, $p = p'$ and $q=0$. 
\end{proof}

\subsection{Proofs of Theorem~\ref{A1ToAk} and Theorem~\ref{KX2-universal}}
\label{subsect_Structure}
 {In this subsection, we introduce a new class ${\mathfrak Z}_{n, d}\in H_*((\Xn)^3)$ in terms of 
cycles $\Theta^{[\![\alpha, \delta]\!]}$ studied intensively in previous subsections. Note that now ${\mathfrak Z}_{n, d}$  is on the Hilbert scheme $(X^{[n]})^3$, not on the non-separated spaces
$(X^{[\![\le n]\!]})^3$.}

Let $\mathcal B = \{ \beta_1, \ldots, \beta_b \}$ be a basis of $H^2(X)$. 
Then, $\{1_X, x, \beta_1, \ldots, \beta_b \}$ is a basis of $H^*(X)$, 
and $H^*(\Xn)$ has a basis $\mathcal B^{[n]}$ consisting of the elements
$
\mathfrak a_{-\la}(1_X) \mathfrak a_{-\mu}(x) \mathfrak a_{-\nu^{(1)}}(\beta_1) 
\cdots \mathfrak a_{-\nu^{(b)}}(\beta_b) \vac
$
where $|\la| + |\mu| + \sum_i |\nu^{(i)}| =n$. Via the K\" unneth decomposition, 
a basis of $H^*((\Xn)^3)$ consists of the elements 
$A_1 \otimes A_2 \otimes A_3 = \prod_{i=1}^3 \pi_{n, i}^*A_i$, 
where $A_1, A_2, A_3 \in \mathcal B^{[n]}$ and $\pi_{n, i}$ denotes 
the $i$-th projection $(X^{[n]})^3 \to \Xn$.

\begin{definition}   \label{Theta-CalZ}
\begin{enumerate}
\item[{\rm (i)}] Let $d \ge 1$, and let $\mathcal P_{[n], d}^+$ be the subset of $\mathcal P_{[n], d}$
consisting of all the weighted partitions $(\alpha, \delta)$ such that 
$\delta_i > 0$ for every $i$. 

\item[{\rm (ii)}] For $d \ge 1$, define the class ${\mathfrak Z}_{n, d} = 
{\mathfrak Z}_{n, d}^{\mathcal B} \in H_*((\Xn)^3)$ by putting
\begin{eqnarray}  \label{Theta-CalZ.0}
\left \langle{\mathfrak Z}_{n, d}, \,\,\prod_{i=1}^3 \pi_{n, i}^*A_i  \right \rangle
=  {1 \over n!} \cdot \sum_{(\alpha, \delta) \in \mathcal P_{[n], d}^+}
     \left\langle (\W \phi^3_{\alpha, [n]})_*[\Theta^{[\![\alpha, \delta]\!]}], \,\,
     A_1^{[\![\le n]\!]} \otimes A_2^{[\![\le n]\!]} \otimes A_3^{[\![\le n]\!]} \right\rangle
\end{eqnarray}
for the basis elements $A_1, A_2, A_3 \in \mathcal B^{[n]}$.
\end{enumerate}
\end{definition}

Next, we prove Theorem~\ref{A1ToAk} and Theorem~\ref{KX2-universal}
which determine the structure of the $3$-pointed genus-$0$ extremal Gromov-Witten 
invariants of $\Xn$. Note from Theorem~\ref{KX2-universal} that
the class ${\mathfrak Z}_{n, d} = {\mathfrak Z}_{n, d}^{\mathcal B}$ 
is independent of the choice of the basis $\mathcal B$ of $H^2(X)$. 
So from now on, the basis $\mathcal B$ of $H^2(X)$ will be implicit in our presentation.

\medskip\noindent 
{\it Proof of Theorem~\ref{A1ToAk}}.
By \eqref{IndepExt.1}, the invariant $\langle A_1, A_2, A_3 \rangle_{0, d \beta_n}$ is equal to
\begin{eqnarray}     \label{A1ToAk.1}
 {1 \over n!} \sum_{(\alpha, \delta) \le ([n], d)}\left \langle
   (\W \phi^3_{\alpha, [n]})_*[\Theta^{[\![\alpha, \delta]\!]}], \,\,
   A_1^{[\![\le n]\!]} \otimes A_2^{[\![\le n]\!]} \otimes A_3^{[\![\le n]\!]} \right\rangle.
\end{eqnarray}
Define $\alpha^0 = \{ (\alpha_i)_i|\, \delta_i = 0\}$, and let $(\alpha^0, 0)$ be 
the weighted partition such that all the weights are equal to $0$.
Let $(\alpha', \delta')$ be the weighted partition obtained from $(\alpha, \delta)$ by deleting 
all the $\alpha_i$ and $\delta_i$ with $\delta_i = 0$. Let $|\alpha'| = m$, $\Lambda_{\alpha^0} 
= \coprod_i (\alpha^0)_i$, and $\Lambda_{\alpha'} = \coprod_i (\alpha')_i$. Then, 
$\alpha = (\alpha^0, 0) \coprod (\alpha', \delta')$,  $|\alpha^0| = n-m$, and 
$[n] = \Lambda_{\alpha^0} \coprod \Lambda_{\alpha'}$. By \eqref{IndepExt.2}, 
$\langle(\W \phi^3_{\alpha, [n]})_*[\Theta^{[\![\alpha, \delta]\!]}] ,\,\, A_1^{[\![\le n]\!]} 
\otimes A_2^{[\![\le n]\!]} \otimes A_3^{[\![\le n]\!]} \rangle $ equals
\begin{eqnarray*}   
& &\sum_{A_{1,1} \circ \cdots \circ A_{1,l}=A_1 
   \atop{A_{2,1} \circ \cdots \circ A_{2,l}=A_2 
   \atop A_{3,1} \circ \cdots \circ A_{3,l}=A_3}} \prod_{i=1}^l \left \langle
   [\Theta^{[\![\alpha_i, \delta_i]\!]}],  \,\,A_{1,i}^{[\![\le |\alpha_i|]\!]} \otimes 
   A_{2,i}^{[\![\le |\alpha_i|]\!]} \otimes A_{3,i}^{[\![\le |\alpha_i|]\!]}  \right\rangle   \\
&=&\sum_{A_{1,1} \circ A_{1,2} = A_1 \atop {A_{2,1} \circ A_{2,2} = A_2 \atop 
   A_{3,1} \circ A_{3,2} = A_3}} 
\left  \langle (\W \phi^3_{\alpha^0, \Lambda_{\alpha^0}})_*[\Theta^{[\![\alpha^0, 0]\!]}],
   A_{1,1}^{[\![\le (n-m)]\!]} \otimes A_{2, 1}^{[\![\le (n-m)]\!]} \otimes 
   A_{3,1}^{[\![\le (n-m)]\!]} \right\rangle           \\
& &\qquad \qquad \cdot \left\langle(\W \phi^3_{\alpha', \Lambda_{\alpha'}})_*[\Theta^{[\![\alpha', \delta']\!]}] ,
   A_{1,2}^{[\![\le m]\!]} \otimes A_{2, 2}^{[\![\le m]\!]} \otimes 
   A_{3,2}^{[\![\le m]\!]} \right\rangle.
\end{eqnarray*}
Put $\Lambda = \Lambda_{\alpha'}$. By \eqref{A1ToAk.1}, 
$\langle A_1, A_2, A_3 \rangle_{0, d \beta_n}$ is equal to
\begin{eqnarray}     \label{A1ToAk.2}
& &{1 \over n!} \cdot \sum_{m \le n}
   \sum_{A_{1,1} \circ A_{1,2} = A_1 \atop {A_{2,1} \circ A_{2,2} = A_2 \atop 
   A_{3,1} \circ A_{3,2} = A_3}}
   \sum_{\Lambda \subset [n] \atop |\Lambda| = m}
   \sum_{(\alpha', \delta') \in \mathcal P_{\Lambda, d}^+}   \nonumber   \\
& &\sum_{\alpha^0 \in \mathcal P_{[n]- \Lambda}} 
   \left\langle (\W \phi^3_{\alpha^0, [n]- \Lambda})_*[\Theta^{[\![\alpha^0, 0]\!]}],
   A_{1,1}^{[\![\le (n-m)]\!]} \otimes A_{2, 1}^{[\![\le (n-m)]\!]} \otimes 
   A_{3,1}^{[\![\le (n-m)]\!]} \right\rangle  \nonumber   \\
& &\qquad \qquad \cdot \left\langle (\W \phi^3_{\alpha', \Lambda})_*[\Theta^{[\![\alpha', \delta']\!]}] ,
     A_{1,2}^{[\![\le m]\!]} \otimes A_{2, 2}^{[\![\le m]\!]} \otimes 
   A_{3,2}^{[\![\le m]\!]}\right\rangle .
\end{eqnarray}

In particular, setting $d = 0$ in (\ref{A1ToAk.2}), we see that 
$\langle  A_1, A_2, A_3 \rangle$ is equal to
\begin{eqnarray*}  
{1 \over n!} \cdot \sum_{\alpha \in \mathcal P_{[n]}} \left\langle (\W \phi^3_{\alpha, [n]})_*
[\Theta^{[\![\alpha, 0]\!]}] , A_1^{[\![\le n]\!]} \otimes A_2^{[\![\le n]\!]} 
\otimes A_3^{[\![\le n]\!]} \right\rangle.
\end{eqnarray*}
Therefore, by \eqref{A1ToAk.2}, 
$\langle A_1, A_2, A_3 \rangle_{0, d \beta_n}$ is equal to
\begin{eqnarray*}  
& &{1 \over n!} \cdot \sum_{m \le n}
   \sum_{A_{1,1} \circ A_{1,2} = A_1 \atop {A_{2,1} \circ A_{2,2} = A_2 \atop 
   A_{3,1} \circ A_{3,2} = A_3}}
  \sum_{\Lambda \subset [n] \atop |\Lambda| = m}
  \sum_{(\alpha', \delta') \in \mathcal P_{\Lambda, d}^+}  
  (n-m)! \cdot \langle A_{1,1}, A_{2,1}, A_{3,1} \rangle \\
& &\qquad \qquad \cdot \langle (\W \phi^3_{\alpha', \Lambda})_*[\Theta^{[\![\alpha', \delta']\!]}]
   , A_{1,2}^{[\![\le m]\!]} \otimes A_{2, 2}^{[\![\le m]\!]} \otimes 
   A_{3,2}^{[\![\le m]\!]} \rangle  \\
&=&{1 \over n!} \cdot \sum_{m \le n}
  \sum_{A_{1,1} \circ A_{1,2} = A_1 \atop {A_{2,1} \circ A_{2,2} = A_2 \atop 
   A_{3,1} \circ A_{3,2} = A_3}}
  \sum_{(\alpha, \delta) \in \mathcal P_{[m], d}^+} {n \choose m} 
 (n-m)! \cdot \langle A_{1,1}, A_{2,1}, A_{3,1} \rangle   \\
& &\qquad \qquad \cdot \langle (\W \phi^3_{\alpha, [m]})_*[\Theta^{[\![\alpha, \delta]\!]}] 
  , A_{1,2}^{[\![\le m]\!]} \otimes A_{2, 2}^{[\![\le m]\!]} \otimes 
   A_{3,2}^{[\![\le m]\!]} \rangle.
%&=&\sum_{m \le n} \sum_{A_{1,1} \circ A_{1,2} = A_1 \atop {A_{2,1} \circ A_{2,2} = A_2 \atop 
%   A_{3,1} \circ A_{3,2} = A_3}} \langle A_{1,1}, A_{2,1}, A_{3,1} \rangle  
%   \cdot \left ( {\mathfrak Z}_{m, d} \cdot \prod_{i=1}^3 \pi_{m, i}^*A_{i,2} \right )
\end{eqnarray*}
Using the definition of ${\mathfrak Z}_{m, d}$, we complete the proof of our theorem.
\qed

\medskip\noindent
{\it Proof of Theorem~\ref{KX2-universal}}. 
Let $A_i = \mathfrak a_{-\la^{(i)}}(1_X)\mathfrak a_{-\mu^{(i)}}(x)
\mathfrak a_{-n_{i,1}}(\alpha_{i,1}) \cdots \mathfrak a_{-n_{i,u_i}}(\alpha_{i,u_i}) \vac$
with $|\alpha_{i,j}|=2$. By linearity, we may assume  $\alpha_{i,j} \in \mathcal B$ for 
every $i$ and $j$. By \eqref{Theta-CalZ.0} and \eqref{IndepExt.2}, 
\begin{eqnarray}   \label{KX2-universal.1}
\left\langle {\mathfrak Z}_{n, d}, \,\,\prod_{i=1}^3 \pi_{n, i}^*A_i \right\rangle= {1 \over n!} \cdot \sum_{(\alpha, \delta) \in \mathcal P_{[n], d}^+}
   \sum_{A_{1,1} \circ \cdots \circ A_{1,l}=A_1 
   \atop{A_{2,1} \circ \cdots \circ A_{2,l}=A_2 
   \atop A_{3,1} \circ \cdots \circ A_{3,l}=A_3}} \prod_{i=1}^l 
 \left \langle  [\Theta^{[\![\alpha_i, \delta_i]\!]}] ,\,\,
    \otimes_{j=1}^3 A_{j,i}^{[\![\le |\alpha_i|]\!]}\right \rangle .
\end{eqnarray}
So our theorem, except the degree of $p$ in (ii), follows from Lemma~\ref{ThetaVanish}, 
Lemma~\ref{1CurveClass} and Lemma~\ref{LiJ-PropB}. To see the degree of $p$ in (ii),
consider a nonzero term in \eqref{KX2-universal.1}:
\begin{eqnarray}   \label{KX2-universal.2}
\prod_{i=1}^l \left \langle [\Theta^{[\![\alpha_i, \delta_i]\!]}], \,\,
   A_{1,i}^{[\![\le |\alpha_i|]\!]} \otimes 
   A_{2,i}^{[\![\le |\alpha_i|]\!]} \otimes A_{3,i}^{[\![\le |\alpha_i|]\!]} \right\rangle.
\end{eqnarray}
By Lemma~\ref{ThetaVanish}~(ii),  for each $i$ in \eqref{KX2-universal.2},
the classes $A_{1,i}, A_{2, i}, A_{3,i}$ together contains at most one Heisenberg factor
of the form $\mathfrak a_{-n_{j, k}}(\alpha_{j,k})$. By Lemma~\ref{1CurveClass} and 
Lemma~\ref{LiJ-PropB}, the degree of \eqref{KX2-universal.2} as a monomial of 
$\langle K_X, K_X \rangle$ is equal to $|I|$ where $I$ is the set consisting of the index 
$i \in \{1, \ldots, l\}$ such that the classes $A_{1,i}, A_{2, i}, A_{3,i}$ together do not 
contain any Heisenberg factor of the form $\mathfrak a_{-n_{j, k}}(\alpha_{j,k})$. 
Now for each $i \in I$, $|\alpha_i| \ge 2$ since $\delta_i \ge 1$. So we conclude that
$$
|I| \le {1 \over 2} \sum_{i \in I} |\alpha_i| = {1 \over 2} (n - \sum_{i \not \in I} |\alpha_i|)
\le {1 \over 2} (n - \sum_{j,k} n_{j,k}).
$$
Hence the degree of $p$ as a polynomial of $\langle K_X, K_X \rangle$ is at most
$(n - \sum_{i, j} n_{i, j})/2$.
\qed

\begin{corollary}  \label{PtClass}
Let $d \ge 1$, and let $A_1, A_2, A_3 \in H^*(\Xn)$ be Heisenberg monomial classes.
\begin{enumerate}
\item[{\rm (i)}] If $A_1 = \mathfrak a_{-1}(1_X)^{n-1}\mathfrak a_{-1}(\alpha) \vac$, 
then $\langle{\mathfrak Z}_{n, d}, \,\, \prod_{i=1}^3 \pi_{n, i}^*A_i \rangle= 0$.

\item[{\rm (ii)}] If $A_1 = \mathfrak a_{-1}(1_X)^{n-1-|\la|} \mathfrak a_{-1}(\alpha)
\mathfrak a_{-\la}(x)\vac$ for some $\la$, then 
$\langle A_1, A_2, A_3 \rangle_{0, d \beta_n} = 0.$
\end{enumerate}
\end{corollary}
\begin{proof}
(i) First of all, if $\alpha = x$, then $\langle {\mathfrak Z}_{n, d},  \,\,\prod_{i=1}^3 \pi_{n, i}^*A_i\rangle = 0$
by Theorem~\ref{KX2-universal}~(i).

Next, let $\alpha = 1_X$. Use induction on $n$. Since $d \ge 1$, the conclusion is trivially true when 
$n = 1$. Let $n > 1$. Recall that $1/n! \cdot A_1$ is the fundamental class $1_{\Xn}$ of $\Xn$.
By Theorem~\ref{A1ToAk} and the Fundamental Class Axiom of Gromov-Witten theory, 
\begin{eqnarray*}  
\sum_{m = 2}^{n} \sum_{A_{1,1} \circ A_{1,2} = A_1 \atop {A_{2,1} \circ A_{2,2} = A_2 \atop
A_{3,1} \circ A_{3,2} = A_3}}  \langle A_{1,1}, A_{2,1}, A_{3,1}  \rangle \cdot 
\left \langle {\mathfrak Z}_{m, d},  \,\,\prod_{i=1}^3 \pi_{m, i}^*A_{i,2} \right \rangle = 0.
\end{eqnarray*}
Since $A_1 = \mathfrak a_{-1}(1_X)^n \vac$, we have $A_{1,2} = \mathfrak a_{-1}(1_X)^m \vac$. 
By induction, $\langle{\mathfrak Z}_{m, d} , \,\,\prod_{i=1}^3 \pi_{m, i}^*A_{i,2}\rangle = 0$
if $2 \le m \le n-1$. It follows that $\langle {\mathfrak Z}_{n, d} , \,\,\prod_{i=1}^3 \pi_{n, i}^*A_i \rangle= 0$.

Now let $|\alpha| = 2$. By the Divisor Axiom of Gromov-Witten theory and 
$\langle A_1, \beta_n \rangle = 0$, we have $\langle A_1, A_2, A_3 \rangle_{0, d \beta_n} = 0$.
Using an argument similar to the one in the previous paragraph, we conclude that
$\langle{\mathfrak Z}_{n, d} , \,\,\prod_{i=1}^3 \pi_{n, i}^*A_i \rangle= 0$.

(ii) We compute $\langle A_1, A_2, A_3 \rangle_{0, d \beta_n}$ by using \eqref{A1ToAk.0}.
Note that the class $A_{1,2}$ in \eqref{A1ToAk.0} is equal to $\mathfrak a_{-1}(1_X)^m\vac$, 
or is equal to $\mathfrak a_{-1}(1_X)^{m-1} \mathfrak a_{-1}(\alpha)\vac$, 
or contains a factor $\mathfrak a_{-i}(x)$ for some $i > 0$. By (i) and 
Theorem~\ref{KX2-universal}~(i), we get $\langle A_1, A_2, A_3 \rangle_{0, d \beta_n} = 0$.
\end{proof}

\section{\bf Proofs of (\ref{Axiom22}) and Theorem~\ref{Intro-Thm1}}
\label{sect_ProofOfAxiom22}

Let $X$ be a simply connected smooth projective surface. Our goal in this section is 
to prove (\ref{Axiom22}) and Theorem~\ref{Intro-Thm1} for $A^{[n]} = 
H^*_{\rho_n}(\Xn)$. The proof of (\ref{Axiom22}) is divided into three cases 
depending on the cohomology degree of the class $\alpha$ in (\ref{Axiom22}) 
and leading to Proposition~\ref{Axiom22A=x}, Proposition~\ref{Axiom22A=2}
and Proposition~\ref{Axiom22A=1X}. Assuming these three propositions, 
we now prove Theorem~\ref{Intro-Thm1}.

%\begin{theorem}  \label{ProofOfRuanConj}
%Let $X$ be a simply connected smooth projective surface. Then, Ruan's Cohomological 
%Crepant Resolution Conjecture~\ref{conj_ruan} holds for the Hilbert-Chow morphism
%$\rho_n$, i.e., the rings $H^*_{\rho_n}(\Xn)$ and $H^*_{\rm CR}(X^{(n)})$ are isomorphic.
%\end{theorem}
\medskip\noindent
{\it Proof of Theorem~\ref{Intro-Thm1}.}
Note that the shift number (or the age) of the class $\mathfrak p_{-n_1}(\alpha_1)\cdots
\mathfrak p_{-n_s}(\alpha_s)\vac$ is equal to $n_1 + \ldots + n_s - s$.
Define a linear isomorphism
\begin{eqnarray}    \label{ProofOfRuanConj.2}
\Psi: \Fock \rightarrow \fock
\end{eqnarray}
by sending $\sqrt{-1}^{n_1 + \ldots + n_s -s} \mathfrak p_{-n_1}(\alpha_1) \cdots 
\mathfrak p_{-n_s}(\alpha_s)\vac$ to $\mathfrak a_{-n_1}(\alpha_1)\cdots 
\mathfrak a_{-n_s}(\alpha_s)\vac$. This induces a linear isomorphism $\Psi_n:
H^*_{\rm CR}(X^{(n)}) \rightarrow H^*(\Xn)$ for each $n$. 
Moreover, $\Psi_1$ is simply the identity map on the cohomology group of the surface $X$. 

By (\ref{OprtWG1X}), Proposition~\ref{Axiom22A=x}, Proposition~\ref{Axiom22A=2}
and Proposition~\ref{Axiom22A=1X}, the two formulas (\ref{Axiom21}) and (\ref{Axiom22}) 
hold for $A^{[n]} = H^*_{\rho_n}(\Xn)$. By the proof of 
Theorem~\ref{char_th} (i.e., Theorem~4.7 in \cite{LQW3}),
\begin{eqnarray}    \label{ProofOfRuanConj.1}
\W {\mathfrak G}_k(\alpha) 
= -\sum_{\ell(\lambda) = k+2, |\lambda|=0}
   {1 \over \lambda^!} \mathfrak a_{\lambda}(\tau_{*}\alpha)
   + \sum_{\ell(\lambda) = k, |\lambda|=0}
   {\lambsq - 2 \over 24\lambda^!}
   \mathfrak a_{\lambda}(\tau_{*}(e_X\alpha)).
\end{eqnarray}
Combining this with formula (\ref{Oprt-OkAlphaN}), we check directly that 
\begin{eqnarray*}    
& &\Psi_n \Big ( \sqrt{-1}^k O_k(\alpha, n) \bullet \sqrt{-1}^{n_1 + \ldots + n_s -s} 
   \mathfrak p_{-n_1}(\alpha_1) \cdots \mathfrak p_{-n_s}(\alpha_s)\vac \Big )  \\
&=&\Psi_n \Big ( \sqrt{-1}^{k+n_1 + \ldots + n_s -s} \mathfrak O_k(\alpha)
   \mathfrak p_{-n_1}(\alpha_1) \cdots \mathfrak p_{-n_s}(\alpha_s)\vac \Big )  \\
&=&\W {\mathfrak G}_k(\alpha) \mathfrak a_{-n_1}(\alpha_1) \cdots 
   \mathfrak a_{-n_s}(\alpha_s)\vac \\
&=&\W G_k(\alpha, n) \cdot \mathfrak a_{-n_1}(\alpha_1) \cdots \mathfrak a_{-n_s}(\alpha_s)\vac
\end{eqnarray*}
where $n_1 + \ldots + n_s = n$. In particular, letting $s=n$, $n_1 = \ldots = n_s = 1$
and $\alpha_1 = \ldots = \alpha_s = 1_X$, we obtain 
$\Psi_n \big ( \sqrt{-1}^k O_k(\alpha, n) \big ) = \W G_k(\alpha, n)$. Thus,
\begin{eqnarray*}    
& &\Psi_n \Big ( \sqrt{-1}^k O_k(\alpha, n) \bullet \sqrt{-1}^{n_1 + \ldots + n_s -s} 
   \mathfrak p_{-n_1}(\alpha_1) \cdots \mathfrak p_{-n_s}(\alpha_s)\vac \Big )  \\
&=&\Psi_n \big ( \sqrt{-1}^k O_k(\alpha, n) \big ) \cdot 
   \Psi_n \big ( \sqrt{-1}^{n_1 + \ldots + n_s -s} \mathfrak p_{-n_1}(\alpha_1) 
   \cdots \mathfrak p_{-n_s}(\alpha_s)\vac \big ).
\end{eqnarray*}
Since the classes $O_k(\alpha, n)$ with $k \ge 0, \alpha \in H^*(X)$ generate the ring
$H^*_{\rm CR}(X^{(n)})$, we conclude that $\Psi_n:
H^*_{\rm CR}(X^{(n)}) \rightarrow H^*(\Xn)$ is a ring isomorphism.
\qed

\begin{remark}  \label{RmkPairing}
Using Heisenberg monomial classes, one checks that the ring isomorphism $\Psi_n$ preserves 
the pairings on $H^*_{\rm CR}(X^{(n)})$ and $H^*(\Xn)$.
\end{remark}

In the next three subsections, we will verify (\ref{Axiom22}) by proving 
Propositions~\ref{Axiom22A=x}, \ref{Axiom22A=2} and \ref{Axiom22A=1X} 
used in the proof of Theorem~\ref{Intro-Thm1}. For simplicity, put
$
\langle w_1, w_2, w_3 \rangle_d = \langle w_1, w_2, w_3 \rangle_{0, d \beta_n}.
$
In addition, $w_1, w_2$ and $w_3$ will stand for Heisenberg monomial classes.
\subsection{The case $\alpha = x$}
\label{subsect_Alpha=x}

We begin with a setup for the proof of (\ref{Axiom22}) for arbitrary $\alpha, \beta \in H^*(X)$.
To prove (\ref{Axiom22}), it is equivalent to verify that
\begin{eqnarray}  \label{Axiom22Alt}
\big \langle [ \W {\mathfrak G}_k(\alpha), \mathfrak a_{-1}(\beta)]w_1, w_2 \big \rangle
= {1 \over k!} \, \big \langle \mathfrak a^{\{ k \}}_{-1}(\alpha \beta) w_1, w_2 \big \rangle.
\end{eqnarray}
for $w_1 \in H^*_{\rho_n}(X^{[n-1]}) = H^*(X^{[n-1]})$ and 
$w_2 \in H^*_{\rho_n}(\Xn) = H^*(X^{[n]})$. Put
\begin{eqnarray}  \label{DABW1W2}
D^\alpha_\beta(w_1, w_2; q)
:= \langle [ \W {\mathfrak G}_k(\alpha; q), \mathfrak a_{-1}(\beta)]w_1, w_2 \rangle
- {1 \over k!} \, \big \langle \mathfrak a^{\{ k \}}_{-1}(\alpha \beta) w_1, w_2 \big \rangle
\end{eqnarray}
where we have omitted $k$ in $D^\alpha_\beta(w_1, w_2; q)$ since it will be clear from the context.

\begin{lemma} \label{difference}
The difference $D^\alpha_\beta(w_1, w_2; q)$ is equal to 
\begin{eqnarray}  \label{Axiom22AltLeft}
& &\sum_{0 \le j \le k} \sum_{\lambda \vdash (j+1) \atop \ell(\lambda)=k-j+1}
   {(-1)^{|\lambda|-1} \over \lambda^! \cdot |\lambda|!}
   \sum_{d \ge 1} \Big ( \big \langle {\bf 1}_{-(n-j-1)} \mathfrak a_{-\lambda}(\tau_*\alpha)\vac, 
   \mathfrak a_{-1}(\beta)w_1, w_2 \big \rangle_d       \nonumber \\
& &\qquad - \, \big \langle {\bf 1}_{-(n-j-2)} \mathfrak a_{-\lambda}(\tau_*\alpha)\vac, w_1, 
   \mathfrak a_{-1}(\beta)^\dagger w_2 \big \rangle_d \Big ) q^d  \nonumber \\
&+&\sum_{{\epsilon} \in \{K_X, K_X^2\}}  \,\,\,
   \sum_{\ell(\lambda) = k+1-|{\epsilon}|/2 \atop |\lambda|=-1}
   {\w f_{|{\epsilon}|}(\lambda)} \cdot \big \langle 
   \mathfrak a_{\lambda}(\tau_{*}(\epsilon \alpha \beta)) w_1, w_2 \big \rangle  \nonumber \\
&-&\sum_{\epsilon \in \{K_X, K_X^2\} \atop 0 \le j \le k}
   \sum_{\lambda \vdash (j+1) \atop \ell(\lambda)=k-j+1-|{\epsilon}|/2}
   \w g_{|{\epsilon}|}(\lambda) \cdot \Big ( \big \langle {\bf 1}_{-(n-j-1)}
   \mathfrak a_{-\lambda}(\tau_*(\epsilon\alpha))\vac, \mathfrak a_{-1}(\beta)w_1, 
   w_2 \big \rangle  \nonumber \\
& &\qquad - \, \big \langle {\bf 1}_{-(n-j-2)} \mathfrak a_{-\lambda}(\tau_*(\epsilon\alpha)) \vac, 
   w_1, \mathfrak a_{-1}(\beta)^\dagger w_2 \big \rangle \Big )
\end{eqnarray}
where $\mathfrak a_{-1}(\beta)^\dagger = -\mathfrak a_1(\beta)$ is the adjoint operator of 
$\mathfrak a_{-1}(\beta)$, and the functions $\w f_{|{\epsilon}|}(\lambda)$ and 
$\w g_{|{\epsilon}|}(\lambda)$ depend only on $k, |{\epsilon}|$ and $\lambda$.
\end{lemma}
\begin{proof}
By (\ref{OprtWGq}), $\big \langle [ \W {\mathfrak G}_k(\alpha; q), \mathfrak a_{-1}(\beta)]w_1, 
w_2 \big \rangle$ is equal to
\begin{eqnarray}   \label{UsedForn=1}
& &\big \langle \W {\mathfrak G}_k(\alpha; q) \big (\mathfrak a_{-1}(\beta)w_1 \big ), 
         w_2 \big \rangle
   - \big \langle \mathfrak a_{-1}(\beta) \W {\mathfrak G}_k(\alpha; q)(w_1), w_2 
         \big \rangle \nonumber \\
&=&\big \langle \W {\mathfrak G}_k(\alpha; q) \big (\mathfrak a_{-1}(\beta)w_1 \big ), 
         w_2 \big \rangle
   - \big \langle \W {\mathfrak G}_k(\alpha; q)(w_1), \mathfrak a_{-1}(\beta)^\dagger w_2 
         \big \rangle \nonumber \\
&=&\sum_{d \ge 0} \,\,\Big ( \big \langle \W G_k(\alpha,n), \mathfrak a_{-1}(\beta)w_1, 
         w_2 \big \rangle_d      
      -  \big \langle \W G_k(\alpha,n-1), w_1, 
         \mathfrak a_{-1}(\beta)^\dagger w_2 \big \rangle_d \Big ) q^d.  \qquad
\end{eqnarray}
If $d \ge 1$, then we see from (\ref{WGkAlphaN}) and Corollary~\ref{PtClass}~(ii) that 
\begin{eqnarray}   \label{difference.1}
& &\big \langle \W G_k(\alpha,n), \mathfrak a_{-1}(\beta)w_1, w_2 \big \rangle_d   \nonumber  \\
&=&\sum_{0 \le j \le k} \sum_{\lambda \vdash (j+1) \atop \ell(\lambda)=k-j+1}
   {(-1)^{|\lambda|-1} \over \lambda^! \cdot |\lambda|!}
   \big \langle {\bf 1}_{-(n-j-1)} \mathfrak a_{-\lambda}(\tau_*\alpha)\vac, 
   \mathfrak a_{-1}(\beta)w_1, w_2 \big \rangle_d.
\end{eqnarray}
Similarly, if $d \ge 1$, then $\big \langle \W G_k(\alpha,n-1), w_1, 
\mathfrak a_{-1}(\beta)^\dagger w_2 \big \rangle_d$ is equal to
\begin{eqnarray}   \label{difference.2}
\sum_{0 \le j \le k} \sum_{\lambda \vdash (j+1) \atop \ell(\lambda)=k-j+1}
   {(-1)^{|\lambda|-1} \over \lambda^! \cdot |\lambda|!}
   \big \langle {\bf 1}_{-(n-j-2)} \mathfrak a_{-\lambda}(\tau_*\alpha)\vac, w_1, 
   \mathfrak a_{-1}(\beta)^\dagger w_2 \big \rangle_d.
\end{eqnarray}

Next, we study the two terms with $d= 0$ in (\ref{UsedForn=1}). 
By (\ref{WGkAlphaN}) and Theorem~\ref{g_{1_X, e}}, 
$$
\W G_k(\alpha, n) = G_k(\alpha, n) - \sum_{\epsilon \in \{K_X, K_X^2\} \atop 0 \le j \le k}
   \sum_{\lambda \vdash (j+1) \atop \ell(\lambda)=k-j+1-|{\epsilon}|/2}
   \w g_{|{\epsilon}|}(\lambda) \cdot {\bf 1}_{-(n-j-1)}
   \mathfrak a_{-\lambda}(\tau_*(\epsilon\alpha))\vac
$$
where $\w g_{|{\epsilon}|}(\lambda)$ depends only on $k, |{\epsilon}|$ and $\lambda$.
By Theorem~\ref{commutator}~(iii), Theorem~\ref{deriv_th} and Lemma~\ref{derivLLlma}, 
$\langle G_k(\alpha,n), \mathfrak a_{-1}(\beta)w_1, w_2 \rangle 
- \langle G_k(\alpha,n-1), w_1, \mathfrak a_{-1}(\beta)^\dagger w_2 \rangle$ equals
\begin{eqnarray*}    
& &\langle G_k(\alpha,n) \cdot \mathfrak a_{-1}(\beta)w_1, w_2 \rangle 
    - \langle G_k(\alpha,n-1) \cdot w_1, \mathfrak a_{-1}(\beta)^\dagger w_2 \rangle    \\
&=&\langle \mathfrak G_k(\alpha) \mathfrak a_{-1}(\beta)w_1, w_2 \rangle 
    - \langle \mathfrak a_{-1}(\beta)\mathfrak  G_k(\alpha) w_1, w_2 \rangle    \\
&=&\big \langle [\mathfrak G_k(\alpha), \mathfrak a_{-1}(\beta)]w_1, w_2 \rangle   
     = {1 \over k!} \, \langle \mathfrak a^{(k)}_{-1}(\alpha \beta) w_1, w_2 \rangle   \\
&=&{1 \over k!} \, \langle \mathfrak a^{\{ k \}}_{-1}(\alpha \beta) w_1, w_2 \rangle 
   + \sum_{{\epsilon} \in \{K_X, K_X^2\}}  \,\,\,
   \sum_{\ell(\lambda) = k+1-|{\epsilon}|/2 \atop |\lambda|=-1}
   {\w f_{|{\epsilon}|}(\lambda)} \cdot \big \langle 
   \mathfrak a_{\lambda}(\tau_{*}(\epsilon \alpha \beta)) w_1, w_2 \big \rangle.
\end{eqnarray*}
Thus, $\big \langle \W G_k(\alpha,n), \mathfrak a_{-1}(\beta)w_1, w_2 \big \rangle 
- \big \langle \W G_k(\alpha,n-1), w_1, \mathfrak a_{-1}(\beta)^\dagger w_2 \big \rangle$ is equal to
\begin{eqnarray}    \label{UsedForA=2B=1X}
& &{1 \over k!} \, \langle \mathfrak a^{\{ k \}}_{-1}(\alpha \beta) w_1, w_2 \rangle 
   + \sum_{{\epsilon} \in \{K_X, K_X^2\}}  \,\,\,
   \sum_{\ell(\lambda) = k+1-|{\epsilon}|/2 \atop |\lambda|=-1}
   {\w f_{|{\epsilon}|}(\lambda)} \cdot \big \langle 
   \mathfrak a_{\lambda}(\tau_{*}(\epsilon \alpha \beta)) w_1, w_2 \big \rangle   \nonumber \\
&-&\sum_{\epsilon \in \{K_X, K_X^2\} \atop 0 \le j \le k}
   \sum_{\lambda \vdash (j+1) \atop \ell(\lambda)=k-j+1-|{\epsilon}|/2}
   \w g_{|{\epsilon}|}(\lambda) \cdot \Big ( \big \langle {\bf 1}_{-(n-j-1)}
   \mathfrak a_{-\lambda}(\tau_*(\epsilon\alpha))\vac, \mathfrak a_{-1}(\beta)w_1, 
   w_2 \big \rangle  \nonumber \\
& &- \big \langle {\bf 1}_{-(n-j-2)} \mathfrak a_{-\lambda}(\tau_*(\epsilon\alpha))\vac, 
   w_1, \mathfrak a_{-1}(\beta)^\dagger w_2 \big \rangle \Big ).
\end{eqnarray}

Finally, our lemma follows from (\ref{UsedForn=1}), (\ref{difference.1}), (\ref{difference.2}) 
and (\ref{UsedForA=2B=1X}).
\end{proof}

Now we deal with the simplest case when $\alpha = x$ and $\beta$ is arbitrary.

\begin{proposition}  \label{Axiom22A=x}
If $\alpha = x$ is the cohomology class of a point, then (\ref{Axiom22}) is true.
\end{proposition}
\begin{proof}
By Corollary~\ref{PtClass}~(ii), every term in (\ref{Axiom22AltLeft}) is equal to zero. 
So $D^x_\beta(w_1, w_2; q) = 0$.
Setting $q = -1$, we conclude immediately that (\ref{Axiom22}) is true.
\end{proof}

\subsection{The case $|\alpha| = 2$}
\label{subsect_Alpha=2}

We begin with two lemmas about the structures of the intersections in $H^*(\Xn)$.

\begin{lemma}   \label{IntLaMuNu}
Let $\la$ be a partition with $|\la| \le n$. For $i = 1$ and $2$, let
\begin{eqnarray}   \label{IntLaMuNu.00}
w_i = \mathfrak a_{-\lambda^{(i)}}(x)\mathfrak a_{-\mu^{(i)}}(1_X)
\mathfrak a_{-n_{i,1}}(\alpha_{i,1}) \cdots \mathfrak a_{-n_{i,u_i}}(\alpha_{i,u_i})\vac
%\in H^*(\Xn)
\end{eqnarray}
where $|\alpha_{i,j}| = 2$ for all $i$ and $j$. Then, 
$\big \langle \mathfrak a_{-1}(1_X)^{n - |\la|}\mathfrak a_{-\la}(x)\vac, w_1, w_2 \big \rangle$ 
is equal to
\begin{eqnarray}   \label{IntLaMuNu.0}
\delta_{u_1, u_2} \cdot \sum_{\sigma \in {\rm Perm}\{1, \ldots, u_1\}} \prod_{i=1}^{u_1}
   \langle \alpha_{1, i}, \alpha_{2, \sigma(i)} \rangle  \cdot p(\sigma)
\end{eqnarray}
where $p(\sigma)$ depends only on $\sigma, n, \la$ and all the $\lambda^{(i)}, \mu^{(i)}, n_{i, j}$.
\end{lemma}
\begin{proof}
By Lemma~\ref{PolyGkAlphaRmk}~(i), $\mathfrak a_{-1}(1_X)^{n - |\la|}\mathfrak a_{-\la}(x)\vac$ 
is a polynomial of the classes $G_k(x, n), k \ge 0$ whose coefficients are independent of $X$. 
In addition, the integers $k$ involved depend only on $\la$. Note that
\begin{eqnarray*}   
& &\big \langle G_{k_1}(x, n) \cdots G_{k_l}(x, n), w_1, w_2 \big \rangle  
    = \big \langle \mathfrak G_{k_1}(x) \cdots \mathfrak G_{k_l}(x)w_1, w_2 \big \rangle  \\
&=&\big \langle \mathfrak a_{-\lambda^{(1)}}(x)\mathfrak a_{-n_{1,1}}(\alpha_{1,1}) 
   \cdots \mathfrak a_{-n_{1,u_1}}(\alpha_{1,u_1})\mathfrak G_{k_1}(x) \cdots 
   \mathfrak G_{k_l}(x)\mathfrak a_{-\mu^{(1)}}(1_X)\vac, w_2 \big \rangle.
\end{eqnarray*}
So by Theorem~\ref{char_th} and Theorem~\ref{commutator}~(i),
$\big \langle G_{k_1}(x, n) \cdots G_{k_l}(x, n), w_1, w_2 \big \rangle$ equals
\begin{eqnarray}   \label{IntLaMuNu.1}
\delta_{u_1, u_2} \cdot \sum_{\sigma \in {\rm Perm}\{1, \ldots, u_1\}} \prod_{i=1}^{u_1}
   \langle \alpha_{1, i}, \alpha_{2, \sigma(i)} \rangle  \cdot \w p(\sigma)
\end{eqnarray}
where $\w p(\sigma)$ depends only on $\sigma, n, k_1, \ldots, k_l$ and 
all the $\lambda^{(i)}, \mu^{(i)}, n_{i, j}$.
\end{proof}

\begin{lemma}  \label{New1xAlBe}
Let $n_0 \ge 1$, $|\alpha| = 2$, and $\la$ be a partition. Let $w_1$ and $w_2$ be given by
(\ref{IntLaMuNu.00}). Then, $\big \langle {\bf 1}_{-(n-|\la|-n_0)} 
\mathfrak a_{-\la}(x) \mathfrak a_{-n_0}(\alpha)\vac, w_1, w_2 \big \rangle$ is equal to
\begin{eqnarray}   \label{New1xAlBe.0}
& &\langle K_X, \alpha \rangle \cdot \delta_{u_1, u_2} \cdot 
   \sum_{\sigma_1 \in {\rm Perm} \{1, \ldots, u_1\}} \prod_{i=1}^{u_1}
   \langle \alpha_{1, i}, \alpha_{2, \sigma_1(i)} \rangle  \cdot p_1(\sigma_1)   \nonumber   \\
&+&\sum_{j=1}^{u_1} \langle \alpha, \alpha_{1, j} \rangle \cdot \delta_{u_1-1, u_2} \cdot 
   \sum_{\sigma_2} \prod_{i \ne j}
   \langle \alpha_{1, i}, \alpha_{2, \sigma_2(i)} \rangle  \cdot p_2(\sigma_2)    \nonumber   \\
&+&\sum_{j=1}^{u_2} \langle \alpha, \alpha_{2, j} \rangle \cdot \delta_{u_1, u_2-1} \cdot 
   \sum_{\sigma_3} \prod_{i =1}^{u_1}
   \langle \alpha_{1, i}, \alpha_{2, \sigma_3(i)} \rangle  \cdot p_3(\sigma_3)
\end{eqnarray}
where $\sigma_2$ runs over all bijections $\{1, \ldots, u_1\}-\{j\} \to \{1, \ldots, u_2\}$,
$\sigma_3$ runs over all bijections $\{1, \ldots, u_1\} \to \{1, \ldots, u_2\} -\{j\}$,
and $p_1(\sigma_1)$ (respectively, $p_2(\sigma_2)$, $p_3(\sigma_3)$) depend only on 
$\sigma_1$ (respectively, $\sigma_2$, $\sigma_3$), $n, n_0, \la$ and 
all the $\lambda^{(i)}, \mu^{(i)}, n_{i, j}$.
\end{lemma}
\begin{proof}
By Lemma~\ref{PolyGkAlphaRmk}~(ii), ${\bf 1}_{-(n-|\la|-n_0)} \mathfrak a_{-\la}(x) 
\mathfrak a_{-n_0}(\alpha)\vac$ can be written as 
$$
\langle K_X, \alpha \rangle \cdot F_1(n) + \sum_i G_{k_i}(\alpha, n) \cdot F_{2, i}(n)
$$ 
where $F_1(n)$ and $F_{2, i}(n)$ are polynomials of $G_k(x, n), k \ge 0$
whose coefficients are independent of $n$ and $\alpha$. Moreover,
the integers $k$ and $k_i$ depend only on $\la$ and $n_0$. Thus,
\begin{eqnarray}   \label{New1xAlBe.1}
& &\big \langle {\bf 1}_{-(n-|\la|-n_0)}\mathfrak a_{-\la}(x) 
     \mathfrak a_{-n_0}(\alpha)\vac, w_1, w_2 \big \rangle  \nonumber   \\
&=&\langle K_X, \alpha \rangle \cdot \big \langle F_1(n), w_1, w_2 \big \rangle
+ \sum_i \big \langle G_{k_i}(\alpha, n) \cdot F_{2, i}(n), w_1, w_2 \big \rangle.
\end{eqnarray}
As in the proof of Lemma~\ref{IntLaMuNu}, $\big \langle F_1(n), w_1, w_2 \big \rangle$ is 
of the form
\begin{eqnarray}   \label{New1xAlBe.2}
\delta_{u_1, u_2} \cdot \sum_{\sigma_1 \in {\rm Perm}\{1, \ldots, u_1\}} \prod_{i=1}^{u_1}
   \langle \alpha_{1, i}, \alpha_{2, \sigma_1(i)} \rangle  \cdot \w p_{1,1}(\sigma_1)
\end{eqnarray}
where $\w p_{1,1}(\sigma_1)$ depends only on $\sigma_1, n, n_0, \la$ and 
all the $\lambda^{(i)}, \mu^{(i)}, n_{i, j}$. Also, 
\begin{eqnarray*}   
\big \langle G_{k_i}(\alpha, n)G_{s_1}(x, n) \cdots G_{s_l}(x, n), w_1, w_2 \big \rangle  
= \big \langle \mathfrak G_{s_1}(x) \cdots \mathfrak G_{s_l}(x)\mathfrak G_{k_i}(\alpha)w_1, 
      w_2 \big \rangle.
\end{eqnarray*}
By Theorem~\ref{char_th} and Lemma~\ref{tau_k_tau_{k-1}}, $\big \langle G_{k_i}(\alpha, n)
G_{s_1}(x, n) \cdots G_{s_l}(x, n), w_1, w_2 \big \rangle$ equals
\begin{eqnarray}      \label{New1xAlBe.3}
& &\langle K_X, \alpha \rangle \cdot \delta_{u_1, u_2} \cdot 
   \sum_{\sigma_1 \in {\rm Perm} \{1, \ldots, u_1\}} \prod_{i=1}^{u_1}
   \langle \alpha_{1, i}, \alpha_{2, \sigma_1(i)} \rangle  \cdot \w p_{1,2}(\sigma_1)   \nonumber \\
&+&\sum_{j=1}^{u_1} \langle \alpha, \alpha_{1, j} \rangle \cdot \delta_{u_1-1, u_2} \cdot 
   \sum_{\sigma_2} \prod_{i \ne j}
   \langle \alpha_{1, i}, \alpha_{2, \sigma_2(i)} \rangle  \cdot \w p_2(\sigma_2)  \nonumber   \\
&+&\sum_{j=1}^{u_2} \langle \alpha, \alpha_{2, j} \rangle \cdot \delta_{u_1, u_2-1} \cdot 
   \sum_{\sigma_3} \prod_{i =1}^{u_1}
   \langle \alpha_{1, i}, \alpha_{2, \sigma_3(i)} \rangle  \cdot \w p_3(\sigma_3)
\end{eqnarray}
where $\sigma_2$ runs over all the bijections $\{1, \ldots, u_1\}-\{j\} \to \{1, \ldots, u_2\}$,
and $\sigma_3$ runs over all the bijections $\{1, \ldots, u_1\} \to \{1, \ldots, u_2\} -\{j\}$.
Hence $\sum_i \big \langle G_{k_i}(\alpha, n) \cdot F_{2, i}(n), w_1, w_2 \big \rangle$ is of 
the form \eqref{New1xAlBe.3} as well.
Combining with (\ref{New1xAlBe.1}) and (\ref{New1xAlBe.2}), 
we obtain (\ref{New1xAlBe.0}).
\end{proof}

Next, we introduce the notion of universal polynomials 
$P(K_X, S_1, S_2)$ in $\langle K_X, K_X \rangle$ 
of degree at most $m$ and of type $(u_1, u_2)$, and prove a vanishing lemma.

\begin{definition}   \label{UniPoly}
Fix three integers $m, u_1, u_2 \ge 0$. Then {\it a universal polynomial 
$P(K_X, S_1, S_2)$ in $\langle K_X, K_X \rangle$ 
of degree at most $m$ and of type $(u_1, u_2)$} is of the form
\begin{eqnarray}    \label{UniPoly.1}
& &\sum_{1 \le j_1 < \ldots < j_s \le u_1 \atop 1 \le l_1 < \ldots < l_s \le u_2}
   \prod_{i \not \in \{j_1, \ldots, j_s\}} \langle K_X, \alpha_{1, i} \rangle \cdot
   \prod_{i \not \in \{l_1, \ldots, l_s\}} \langle K_X, \alpha_{2, i} \rangle  \nonumber     \\
& &\cdot \sum_{\sigma \in {\rm Perm}\{l_1, \ldots, l_s\}} \prod_{i=1}^s
   \langle \alpha_{1, j_i}, \alpha_{2, \sigma(l_i)} \rangle  \cdot 
   p(j_1,\ldots, j_s; l_1, \ldots, l_s; \sigma)
\end{eqnarray}
where $S_i = \{ \alpha_{i, 1}, \ldots, \alpha_{i, u_i}\} \subset H^2(X)$,
and $p(j_1,\ldots, j_s; l_1, \ldots, l_s; \sigma)$ is 
a polynomial in $\langle K_X, K_X \rangle$ whose degree is at most $m$ and
whose coefficients are independent of $X$ and the classes $\alpha_{i, j}$.
\end{definition}

\begin{lemma} \label{VanishingUniPoly}
Fix $m, u_1, u_2 \ge 0$. Let $P(K_X, S_1, S_2)$ be a universal polynomial 
in $\langle K_X, K_X \rangle$ of degree at most $m$ and of type $(u_1, u_2)$.
Assume that $P(K_X, S_1, S_2) = 0$ whenever $X$ is a smooth projective toric surface.
Then $P(K_X, S_1, S_2) = 0$ for every smooth projective surface $X$.
\end{lemma}  
\begin{proof}
Let $r \gg m+u_1+u_2$, and let $X_r$ be a smooth toric surface obtained from $\Pee^2$ as 
an $r$-fold blown-up. Let $L_0$ be a line in $\Pee^2$, and let $e_1, \ldots, e_r$ be 
the exceptional divisors. Then, $K_{X_r} = -3L_0 + e_1 + \ldots + e_r$. For fixed 
$j_1,\ldots, j_s, l_1, \ldots, l_s$ and $\sigma$, let
\begin{eqnarray*}
\big \{ \alpha_{1, i}|\, i \in \{1, \ldots, u_1\} - \{j_1, \ldots, j_s\} \}
   &=& \{-e_1, \ldots, -e_{u_1-s} \big \},  \\
\big \{ \alpha_{2, i}|\, i \in \{1, \ldots, u_2\} - \{l_1, \ldots, l_s\} \}
   &=& \{-e_{u_1-s+1}, \ldots, -e_{u_1-s+u_2-s} \big \},
\end{eqnarray*}
and $\alpha_{1, j_i} = \alpha_{2, \sigma(l_i)} = e_{u_1-s+u_2-s+2i} - e_{u_1-s+u_2-s+2i-1}$ for
$i = 1, \ldots, s$. Then,
\begin{eqnarray}    \label{VanishingUniPoly.1}
0 = P(K_{X_r}, S_1, S_2) = (-2)^s \cdot p(j_1,\ldots, j_s; l_1, \ldots, l_s; \sigma)
\end{eqnarray}
by (\ref{UniPoly.1}). It follows that $p(j_1,\ldots, j_s; l_1, \ldots, l_s; \sigma) = 0$
for all the surfaces $X_r$ with $r \gg m+u_1+u_2$. 
Since $p(j_1,\ldots, j_s; l_1, \ldots, l_s; \sigma)$ is 
a polynomial in $\langle K_{X_r}, K_{X_r} \rangle$ whose degree is at most $m$, 
we conclude that as polynomials, $p(j_1,\ldots, j_s; l_1, \ldots, l_s; \sigma) = 0$.
Therefore, $P(K_X, S_1, S_2) = 0$ for every smooth projective surface $X$.
\end{proof}

Our next lemma is about the structure of certain $3$-pointed extremal Gromov-Witten invariants,
and provides the motivation for Definition~\ref{UniPoly}.

\begin{lemma} \label{Alpha=2Lma1}
Let $d, n_0 \ge 1$ and $|\alpha| = 2$. Let $w_1$ and $w_2$ be given by (\ref{IntLaMuNu.00}). Then, 
\begin{eqnarray}    \label{Alpha=2Lma1.0}
\big \langle {\bf 1}_{-(n-|\la|-n_0)} 
\mathfrak a_{-\lambda}(x) \mathfrak a_{-n_0}(\alpha)\vac, w_1, w_2 \big \rangle_d
= \langle K_X, \alpha \rangle \cdot P(K_X, S_1, S_2)
\end{eqnarray}
where $S_1 = \{\alpha_{1,1}, \ldots, \alpha_{1,u_1} \}$, 
$S_2 = \{\alpha_{2,1}, \ldots, \alpha_{2,u_2} \}$, and $P(K_X, S_1, S_2)$ is
a universal polynomial in $\langle K_X, K_X \rangle$ of degree at most $(n-n_0)/2$ and 
of type $(u_1, u_2)$.
\end{lemma}  
\begin{proof}
For simplicity, let $w_0 = {\bf 1}_{-(n-|\la|-n_0)} \mathfrak a_{-\lambda}(x) 
\mathfrak a_{-n_0}(\alpha)\vac$. Also, for $i = 1$ and $2$, let
$\W w_i = \mathfrak a_{-\mu^{(i)}}(1_X)
\mathfrak a_{-n_{i,1}}(\alpha_{i,1}) \cdots \mathfrak a_{-n_{i,u_i}}(\alpha_{i,u_i})\vac$.
We compute $\big \langle w_0, w_1, w_2 \big \rangle_d$ by using (\ref{A1ToAk.0}). 
Consider the following term from (\ref{A1ToAk.0}):
\begin{eqnarray}  \label{Alpha=2Lma1.1}
\langle B_0, B_1, B_2 \rangle \cdot \left \langle{\mathfrak Z}_{m, d}  
   ,  \,\,\pi_{m, 1}^* \left ({w_0 \over B_0} \right ) \cdot 
\pi_{m, 2}^* \left ({w_1 \over B_1} \right ) \cdot 
\pi_{m, 3}^* \left ({w_2 \over B_2} \right )\right\rangle
\end{eqnarray}
where $m \le n$, $B_0, B_1, B_2 \in H^*(X^{[n-m]})$, $B_0 \subset w_0$, $B_1 \subset w_1$, and 
$B_2 \subset w_2$. By Theorem~\ref{KX2-universal}~(i) and Corollary~\ref{PtClass}~(i), 
such a term is nonzero only if 
$B_0 = \mathfrak a_{-1}(1_X)^j\mathfrak a_{-\lambda}(x)\vac$ with $j \le (n-|\la|-n_0)$, 
$B_1 = \mathfrak a_{-\lambda^{(1)}}(x) \W B_1$ with $\W B_1 \subset \W w_1$, and 
$B_2 = \mathfrak a_{-\lambda^{(2)}}(x) \W B_2$ with $\W B_2 \subset \W w_2$.
In this situation, (\ref{Alpha=2Lma1.1}) can be rewritten as
\begin{eqnarray}  
& &\big \langle \mathfrak a_{-1}(1_X)^j\mathfrak a_{-\lambda}(x)\vac, 
   \mathfrak a_{-\lambda^{(1)}}(x) \W B_1, 
   \mathfrak a_{-\lambda^{(2)}}(x) \W B_2 \big \rangle      \label{Alpha=2Lma1.2} \\
& &\cdot \left\langle {\mathfrak Z}_{m, d}  
   ,\,\pi_{m, 1}^* \left ({{\bf 1}_{-(n-|\la|-n_0)} \mathfrak a_{-n_0}(\alpha) \vac 
   \over \mathfrak a_{-1}(1_X)^j\vac} \right ) \cdot 
   \pi_{m, 2}^* \left ({\W w_1 \over \W B_1} \right ) \cdot 
   \pi_{m, 3}^* \left ({\W w_2 \over \W B_2} \right )\right\rangle.       \label{Alpha=2Lma1.3}
\end{eqnarray}
Note that $\W B_1 = \mathfrak a_{-\nu^{(1)}}(1_X) \mathfrak a_{-n_{1, j_1}}(\alpha_{1, j_1}) 
\cdots \mathfrak a_{-n_{1, j_s}}(\alpha_{1, j_s})\vac$ for some $1 \le j_1 < \ldots < j_s \le u_1$
and some sub-partition $\nu^{(1)}$ of $\mu^{(1)}$ (i.e., every part of $\nu^{(1)}$ is 
a part of $\mu^{(1)}$). Similarly, 
$\W B_2 = \mathfrak a_{-\nu^{(2)}}(1_X) \mathfrak a_{-n_{2, l_1}}(\alpha_{2, l_1}) 
\cdots \mathfrak a_{-n_{2, l_t}}(\alpha_{2, l_t})\vac$ 
for some $1 \le l_1 < \ldots < l_t \le u_2$ and some sub-partition $\nu^{(2)}$ 
of $\mu^{(2)}$. By Lemma~\ref{IntLaMuNu}, (\ref{Alpha=2Lma1.2}) equals
\begin{eqnarray}   \label{Alpha=2Lma1.4}
\delta_{s, t} \cdot \sum_{\sigma \in {\rm Perm}\{l_1, \ldots, l_s\}} \prod_{i=1}^s
   \langle \alpha_{1, j_i}, \alpha_{2, \sigma(l_i)} \rangle  \cdot 
   p_1(j_1,\ldots, j_s; l_1, \ldots, l_s; \sigma)
\end{eqnarray}
where $p_1$ is a number independent of the surface $X$ and the classes $\alpha_{i,j}$.
By Theorem~\ref{KX2-universal}, we see that the factor (\ref{Alpha=2Lma1.3}) is equal to
\begin{eqnarray*}    
\langle K_X, \alpha \rangle \cdot
   \prod_{i \not \in \{j_1, \ldots, j_s\}} \langle K_X, \alpha_{1, i} \rangle \cdot
   \prod_{i \not \in \{l_1, \ldots, l_s\}} \langle K_X, \alpha_{2, i} \rangle  
\cdot p_2(j_1,\ldots, j_s; l_1, \ldots, l_s; \sigma)
\end{eqnarray*}
where $p_2$ is a polynomial in $\langle K_X, K_X \rangle$ whose degree is at most 
$(m-n_0)/2 \le (n-n_0)/2$, and whose coefficients are independent of the surface $X$ and 
the classes $\alpha_{i,j}$. Combining this with (\ref{Alpha=2Lma1.1}), 
(\ref{Alpha=2Lma1.2}), (\ref{Alpha=2Lma1.3}) and (\ref{Alpha=2Lma1.4}),
we obtain (\ref{Alpha=2Lma1.0}).
\end{proof}

\begin{proposition}  \label{Axiom22A=2}
If $|\alpha| = 2$, then (\ref{Axiom22}) is true.
\end{proposition}
\noindent
{\it Proof.}
Recall that (\ref{Axiom22}) is equivalent to (\ref{Axiom22Alt}), and the difference 
$D^\alpha_\beta(w_1, w_2; q)$ from (\ref{DABW1W2}) is computed by Lemma~\ref{difference}.
Let $w_1$ and $w_2$ be given by (\ref{IntLaMuNu.00}).
Let $u_1' = \delta_{2, |\beta|} + u_1$ and $S_2 = \{\alpha_{2,1}, \ldots, \alpha_{2,u_2} \}$.
Let $S_1 = \{\alpha_{1,1}, \ldots, \alpha_{1,u_1} \}$ if $|\beta| \ne 2$,
and $S_1 = \{\beta, \alpha_{1,1}, \ldots, \alpha_{1,u_1} \}$ if $|\beta| = 2$.

By Lemma~\ref{Axiom22P2} and Lemma~\ref{VanishingUniPoly}, it suffices to prove that
\begin{eqnarray}   \label{Axiom22A=2.1}
D^\alpha_\beta(w_1, w_2; -1) = \langle K_X, \alpha \rangle \cdot P(K_X, S_1, S_2)
\end{eqnarray}
where $P(K_X, S_1, S_2)$ is a universal polynomial in $\langle K_X, K_X \rangle$ 
of degree at most $(n-1)/2$ and of type $(u_1', u_2)$. This follows if we can prove that 
\begin{eqnarray}   \label{Axiom22A=2.2}
D^\alpha_\beta(w_1, w_2; q) = \langle K_X, \alpha \rangle \cdot \sum_{d \ge 0} 
P(K_X, S_1, S_2; d) \, q^d
\end{eqnarray}
where every $P(K_X, S_1, S_2; d)$ is a universal polynomial in $\langle K_X, K_X \rangle$ 
of degree at most $(n-1)/2$ and of type $(u_1', u_2)$. We remark that $d$ has been inserted into 
the notation $P(K_X, S_1, S_2; d)$ to emphasis its dependence on $d$.

In the following, we will show that the contribution of every term in (\ref{Axiom22AltLeft})
is of the form $P(K_X, S_1, S_2; d)$ for a suitable $d \ge 0$. 
Note that in $H^*(X^i)$, 
$$
\tau_{i*}(\alpha) = \alpha \otimes x \otimes \cdots \otimes x + 
x \otimes \alpha \otimes x \otimes \cdots \otimes x + \ldots +
x \otimes \cdots \otimes x \otimes \alpha.
$$
Thus, by Lemma~\ref{Alpha=2Lma1}, $\big  \langle {\bf 1}_{-(n-j-1)} 
\mathfrak a_{-\lambda}(\tau_*\alpha)\vac, \mathfrak a_{-1}(\beta)w_1, w_2 \big \rangle_d$ 
is equal to 
$$
\langle K_X, \alpha \rangle \cdot P_1(K_X, S_1, S_2; d)
$$
where $P_1(K_X, S_1, S_2; d)$ is a universal polynomial in $\langle K_X, K_X \rangle$ 
of degree at most $(n-1)/2$ and of type $(u_1', u_2)$. Similarly, 
since $\mathfrak a_{-1}(\beta)^\dagger w_2 = -\mathfrak a_{1}(\beta) w_2$,
we see from Theorem~\ref{commutator}~(i) and Lemma~\ref{Alpha=2Lma1} that
$$
\big \langle {\bf 1}_{-(n-j-2)} \mathfrak a_{-\lambda}(\tau_*\alpha)\vac, w_1, 
   \mathfrak a_{-1}(\beta)^\dagger w_2 \big \rangle_d 
= \langle K_X, \alpha \rangle \cdot P_2(K_X, S_1, S_2; d).
$$

Next, we move to the term $\big \langle \mathfrak a_{\lambda}(\tau_{*}(\epsilon \alpha \beta)) w_1, 
w_2 \big \rangle$ in (\ref{Axiom22AltLeft}), where ${\epsilon} \in \{K_X, K_X^2\}$.
Such a term is zero unless ${\epsilon}= K_X$ and $|\beta| = 0$. In this case, we may assume that 
$\beta = 1_X$. So let ${\epsilon}= K_X$ and $\beta = 1_X$. Then, 
$$
\big \langle \mathfrak a_{\lambda}(\tau_{*}(\epsilon \alpha \beta)) w_1, w_2 \big \rangle 
= \langle K_X, \alpha \rangle \cdot \big \langle \mathfrak a_{\lambda}(x) w_1, w_2 \big \rangle
= \langle K_X, \alpha \rangle \cdot P_3(K_X, S_1, S_2; 0).
$$
by Theorem~\ref{commutator}~(i), where $P_3(K_X, S_1, S_2; 0)$ is a universal polynomial in 
$\langle K_X, K_X \rangle$ of degree $0$ (i.e., $\langle K_X, K_X \rangle$ does not appear) 
and of type $(u_1', u_2)$. 

Finally, $\tau_*(\epsilon\alpha)$ is zero unless ${\epsilon}= K_X$. Let ${\epsilon}= K_X$. 
By Lemma~\ref{IntLaMuNu},
\begin{eqnarray*}  
   \big \langle {\bf 1}_{-(n-j-1)} \mathfrak a_{-\lambda}(\tau_*(\epsilon\alpha))\vac, 
   \mathfrak a_{-1}(\beta)w_1, w_2 \big \rangle         
&=&\langle K_X, \alpha \rangle \cdot \big \langle {\bf 1}_{-(n-j-1)} \mathfrak a_{-\lambda}(x)\vac, 
   \mathfrak a_{-1}(\beta)w_1, w_2 \big \rangle    \\
&=&\langle K_X, \alpha \rangle \cdot P_4(K_X, S_1, S_2; 0)
\end{eqnarray*}
where $P_4(K_X, S_1, S_2; 0)$ is a universal polynomial in $\langle K_X, K_X \rangle$ of degree $0$ 
and of type $(u_1', u_2)$. Similarly, since $\mathfrak a_{-1}(\beta)^\dagger w_2 
= -\mathfrak a_{1}(\beta) w_2$, we obtain
\begin{equation}
\big \langle {\bf 1}_{-(n-j-2)} \mathfrak a_{-\lambda}(\tau_*(\epsilon\alpha)) \vac, 
w_1, \mathfrak a_{-1}(\beta)^\dagger w_2 \big \rangle 
= \langle K_X, \alpha \rangle \cdot P_5(K_X, S_1, S_2; 0).                   \tag*{$\qed$}
\end{equation}

\subsection{The case $\alpha = 1_X$}
\label{subsect_Alpha=1X}

\begin{lemma} \label{Alpha=1XLma1}
Let $d, n_0 \ge 1$. Let $w_1$ and $w_2$ be given by (\ref{IntLaMuNu.00}). Then, 
\begin{eqnarray}    \label{Alpha=1XLma1.0}
\big \langle {\bf 1}_{-(n-|\la|-n_0)} 
\mathfrak a_{-\lambda}(x) \mathfrak a_{-n_0}(1_X)\vac, w_1, w_2 \big \rangle_d
= P(K_X, S_1, S_2)
\end{eqnarray}
where $S_1 = \{\alpha_{1,1}, \ldots, \alpha_{1,u_1} \}$, 
$S_2 = \{\alpha_{2,1}, \ldots, \alpha_{2,u_2} \}$, and $P(K_X, S_1, S_2)$ is
a universal polynomial in $\langle K_X, K_X \rangle$ of degree at most $n/2$ and 
of type $(u_1, u_2)$.
\end{lemma}  
\begin{proof}
This follows from the proof of Lemma~\ref{Alpha=2Lma1} by replacing $\alpha$ by $1_X$
(and then by noticing that the factor $\langle K_X, \alpha \rangle$ there 
will not appear here).
\end{proof}

\begin{lemma} \label{Alpha=1XTau}
Let $d \ge 1$ and $|\la| \le n$. Let $w_1$ and $w_2$ be given by (\ref{IntLaMuNu.00}). Then, 
\begin{eqnarray}    \label{Alpha=1XTau.0}
\big \langle {\bf 1}_{-(n-|\la|)} \mathfrak a_{-\lambda}(\tau_*1_X)\vac, w_1, w_2 \big \rangle_d
= P(K_X, S_1, S_2)
\end{eqnarray}
where $S_1 = \{\alpha_{1,1}, \ldots, \alpha_{1,u_1} \}$, 
$S_2 = \{\alpha_{2,1}, \ldots, \alpha_{2,u_2} \}$, and $P(K_X, S_1, S_2)$ is
a universal polynomial in $\langle K_X, K_X \rangle$ of degree at most $n/2$ and 
of type $(u_1, u_2)$.
\end{lemma}  
\begin{proof}
 For $i = 1$ and $2$, let $\W w_i = \mathfrak a_{-\mu^{(i)}}(1_X)
\mathfrak a_{-n_{i,1}}(\alpha_{i,1}) \cdots \mathfrak a_{-n_{i,u_i}}(\alpha_{i,u_i})\vac$.
Note that if the K\" unneth decomposition of $\tau_{2*}1_X \in H^*(X^2)$ is given by
$$
\displaystyle{\tau_{2*}1_X 
= x \otimes 1_X + 1_X \otimes x + \sum_j \gamma_{j, 1} \otimes \gamma_{j, 2}}
$$ 
where $|\gamma_{j, 1}| = |\gamma_{j, 2}| = 2$, then up to permutations of factors, 
a typical term in the K\" unneth decomposition of $\tau_{i*}1_X \in H^*(X^i)$ with $i \ge 3$ 
is either $x \otimes \cdots \otimes x \otimes 1_X$ 
or $x \otimes \cdots \otimes x \otimes \gamma_{j, 1} \otimes \gamma_{j, 2}$. 
In view of Lemma~\ref{Alpha=1XLma1}, it suffices to verify that
\begin{eqnarray}    \label{Alpha=1XTau.1}
\sum_j \big \langle {\bf 1}_{-\w n} \mathfrak a_{-\w \lambda}(x)\mathfrak a_{-n_1}(\gamma_{j, 1})
\mathfrak a_{-n_2}(\gamma_{j, 2})\vac, w_1, w_2 \big \rangle_d
= P_1(K_X, S_1, S_2)
\end{eqnarray}
where $\w n = n - |\w \lambda| - n_1 - n_2$, and $P_1(K_X, S_1, S_2)$ is a universal polynomial 
in $\langle K_X, K_X \rangle$ of degree at most $n/2$ and of type $(u_1, u_2)$. For simplicity, let 
$$
w_0 = {\bf 1}_{-\w n} \mathfrak a_{-\w \lambda}(x)
\mathfrak a_{-n_1}(\gamma_{j, 1})\mathfrak a_{-n_2}(\gamma_{j, 2})\vac. 
$$
We see from (\ref{A1ToAk.0}) that to prove (\ref{Alpha=1XTau.1}), it suffices to show that
\begin{eqnarray}    \label{Alpha=1XTau.2}
\sum_j \langle B_0, B_1, B_2 \rangle \cdot \left\langle {\mathfrak Z}_{m, d} 
   , \,\,\pi_{m, 1}^* \left ({w_0 \over B_0} \right ) \cdot 
\pi_{m, 2}^* \left ({w_1 \over B_1} \right ) \cdot 
\pi_{m, 3}^* \left ({w_2 \over B_2} \right )\right\rangle
\end{eqnarray}
is equal to $P_2(K_X, S_1, S_2)$,
where $m \le n$, $B_0, B_1, B_2 \in H^*(X^{[n-m]})$, $B_0 \subset w_0$, $B_1 \subset w_1$, and 
$B_2 \subset w_2$. By Theorem~\ref{KX2-universal}~(i) and Corollary~\ref{PtClass}~(i), 
such a term is nonzero only if 
$B_1 = \mathfrak a_{-\lambda^{(1)}}(x) \W B_1$ with $\W B_1 \subset \W w_1$,  
$B_2 = \mathfrak a_{-\lambda^{(2)}}(x) \W B_2$ with $\W B_2 \subset \W w_2$, and
$B_0 = \mathfrak a_{-1}(1_X)^s\mathfrak a_{-\w \lambda}(x)\vac$ or
$\mathfrak a_{-1}(1_X)^s\mathfrak a_{-\w \lambda}(x)\mathfrak a_{-n_1}(\gamma_{j, 1})\vac$ or 
$\mathfrak a_{-1}(1_X)^s\mathfrak a_{-\w \lambda}(x)\mathfrak a_{-n_2}(\gamma_{j, 2})\vac$
where $s \le \w n$. In the following, we assume that (\ref{Alpha=1XTau.2}) is nonzero.
By symmetry, we need only to consider two cases for $B_0$: 
$$
B_0 = \mathfrak a_{-1}(1_X)^s\mathfrak a_{-\w \lambda}(x)\vac, \quad {\rm or } \,\, 
B_0 = \mathfrak a_{-1}(1_X)^s\mathfrak a_{-\w \lambda}(x)\mathfrak a_{-n_1}(\gamma_{j, 1})\vac.
$$

We begin with the case $B_0 = \mathfrak a_{-1}(1_X)^s\mathfrak a_{-\w \lambda}(x)\vac$. 
Then (\ref{Alpha=1XTau.2}) becomes
\begin{eqnarray*}  
& &\sum_j \big \langle \mathfrak a_{-1}(1_X)^s\mathfrak a_{-\w \lambda}(x)\vac, 
   \mathfrak a_{-\lambda^{(1)}}(x) \W B_1, 
   \mathfrak a_{-\lambda^{(2)}}(x) \W B_2 \big \rangle       \\
& &\cdot \left\langle {\mathfrak Z}_{m, d}
   ,\,\,\pi_{m, 1}^* \left ({ {\bf 1}_{-\w n} \mathfrak a_{-n_1}(\gamma_{j, 1})
   \mathfrak a_{-n_2}(\gamma_{j, 2})\vac 
   \over \mathfrak a_{-1}(1_X)^s \vac} \right ) \cdot 
   \pi_{m, 2}^* \left ({\W w_1 \over \W B_1} \right ) \cdot 
   \pi_{m, 3}^* \left ({\W w_2 \over \W B_2} \right )\right\rangle.       
\end{eqnarray*}
Applying the same arguments as in the computations of (\ref{Alpha=2Lma1.2}) and 
(\ref{Alpha=2Lma1.3}), we conclude that the term (\ref{Alpha=1XTau.2}) is equal to
\begin{eqnarray*}    
\sum_j \langle K_X, \gamma_{j, 1} \rangle \cdot \langle K_X, \gamma_{j, 2} \rangle 
\cdot P_3(K_X, S_1, S_2)
\end{eqnarray*}
where $P_3(K_X, S_1, S_2)$ is a universal polynomial in $\langle K_X, K_X \rangle$ of 
degree at most $(m-n_1-n_2)/2$ and of type $(u_1, u_2)$. Note that for $\beta_1, \beta_2 
\in H^2(X)$, we have
\begin{eqnarray}    \label{Alpha=1XTau.3}
\sum_j \langle \beta_1, \gamma_{j, 1} \rangle \cdot \langle \beta_2, \gamma_{j, 2} \rangle 
= \langle \beta_1, \beta_2 \rangle.
\end{eqnarray}
Therefore, (\ref{Alpha=1XTau.2}) is equal to $\langle K_X, K_X \rangle \cdot 
P_3(K_X, S_1, S_2)$ which is a universal polynomial in 
$\langle K_X, K_X \rangle$ of degree at most $m/2 \le n/2$ and of type $(u_1, u_2)$.

Next, let $B_0 = \mathfrak a_{-1}(1_X)^s\mathfrak a_{-\w \lambda}(x)
\mathfrak a_{-n_1}(\gamma_{j, 1})\vac$. This time, (\ref{Alpha=1XTau.2}) becomes
\begin{eqnarray*}  
& &\sum_j \big \langle \mathfrak a_{-1}(1_X)^s\mathfrak a_{-\w \lambda}(x)
   \mathfrak a_{-n_1}(\gamma_{j, 1})\vac, \mathfrak a_{-\lambda^{(1)}}(x) \W B_1, 
   \mathfrak a_{-\lambda^{(2)}}(x) \W B_2 \big \rangle       \\
& &\cdot \left\langle {\mathfrak Z}_{m, d}
   ,\,\, \pi_{m, 1}^* \left ({ {\bf 1}_{-\w n} \mathfrak a_{-n_2}(\gamma_{j, 2})\vac 
   \over \mathfrak a_{-1}(1_X)^s \vac} \right ) \cdot 
   \pi_{m, 2}^* \left ({\W w_1 \over \W B_1} \right ) \cdot 
   \pi_{m, 3}^* \left ({\W w_2 \over \W B_2} \right )\right\rangle.       
\end{eqnarray*}
Using Lemma~\ref{New1xAlBe}, Theorem~\ref{KX2-universal} and (\ref{Alpha=1XTau.3}), 
we conclude that (\ref{Alpha=1XTau.2}) is equal to $P_4(K_X, S_1, S_2)$ which is 
a universal polynomial in $\langle K_X, K_X \rangle$ of degree at most 
$$
(m - n_2)/2 + 1 \le ((n-n_1) - n_2)/2 + 1 \le n/2
$$ 
and of type $(u_1, u_2)$. This completes the proof of (\ref{Alpha=1XTau.0}).
\end{proof}

\begin{proposition}  \label{Axiom22A=1X}
If $\alpha = 1_X$, then (\ref{Axiom22}) is true.
\end{proposition}
\begin{proof}
We adopt the same notations and approaches as in the proof of Proposition~\ref{Axiom22A=2}.
By Lemma~\ref{Axiom22P2} and Lemma~\ref{VanishingUniPoly}, it suffices to prove that
\begin{eqnarray}   \label{Axiom22A=1X.1}
D^{1_X}_\beta(w_1, w_2; -1) = P(K_X, S_1, S_2)
\end{eqnarray}
where $P(K_X, S_1, S_2)$ is a universal polynomial in $\langle K_X, K_X \rangle$ 
of degree at most $(n+1)/2$ and of type $(u_1', u_2)$. This follows if we can prove that 
\begin{eqnarray}   \label{Axiom22A=1X.2}
D^{1_X}_\beta(w_1, w_2; q) = \sum_{d \ge 0} P(K_X, S_1, S_2; d) \, q^d
\end{eqnarray}
where $P(K_X, S_1, S_2; d)$ is a universal polynomial in $\langle K_X, K_X \rangle$ 
of degree at most $(n+1)/2$ and of type $(u_1', u_2)$. 
In the following, we will show that the contribution of every term in (\ref{Axiom22AltLeft})
is of the form $P(K_X, S_1, S_2; d)$ for a suitable $d \ge 0$. 

 First of all, when $d \ge 1$, we conclude from Lemma~\ref{Alpha=1XTau} that 
$$
\big \langle {\bf 1}_{-(n-j-1)} \mathfrak a_{-\lambda}(\tau_*1_X)\vac, 
   \mathfrak a_{-1}(\beta)w_1, w_2 \big \rangle_d        
- \big \langle {\bf 1}_{-(n-j-2)} \mathfrak a_{-\lambda}(\tau_*1_X)\vac, w_1, 
   \mathfrak a_{-1}(\beta)^\dagger w_2 \big \rangle_d
$$
is equal to $P_1(K_X, S_1, S_2; d)$ which is a universal polynomial in $\langle K_X, K_X \rangle$ 
of degree at most $n/2$ and of type $(u_1', u_2)$. 

Next, consider $\big \langle \mathfrak a_{\lambda}(\tau_{*}(\epsilon \alpha \beta)) w_1, 
w_2 \big \rangle = \big \langle \mathfrak a_{\lambda}(\tau_{*}(\epsilon \beta)) w_1, 
w_2 \big \rangle$ from (\ref{Axiom22AltLeft}), where ${\epsilon} \in \{K_X, K_X^2\}$. 
It is zero unless $\epsilon = K_X^2$ and $\beta = 1_X$ (when $|\beta| = 0$, we let $\beta = 1_X$), 
or $\epsilon = K_X$ and $|\beta| = 2$, or $\epsilon = K_X$ and $\beta = 1_X$. 
If $\epsilon = K_X^2$ and $\beta = 1_X$, then
$$
\big \langle \mathfrak a_{\lambda}(\tau_{*}(\epsilon \beta)) w_1, w_2 \big \rangle 
= \langle K_X, K_X \rangle \cdot \big \langle \mathfrak a_{\lambda}(x) w_1, w_2 \big \rangle  
= \langle K_X, K_X \rangle \cdot P_2(K_X, S_1, S_2; 0)
$$
by Theorem~\ref{commutator}~(i), where $P_2(K_X, S_1, S_2; 0)$ is a universal polynomial in 
$\langle K_X, K_X \rangle$ of degree $0$ and of type $(u_1', u_2)$. 
If $\epsilon = K_X$ and $|\beta| = 2$, then
$$
\big \langle \mathfrak a_{\lambda}(\tau_{*}(\epsilon \beta)) w_1, w_2 \big \rangle
= \langle K_X, \beta \rangle \cdot \big \langle \mathfrak a_{\lambda}(x) w_1, w_2 \big \rangle  
= \langle K_X, \beta \rangle \cdot P_3(K_X, S_1, S_2; 0)
$$
which is a universal polynomial in $\langle K_X, K_X \rangle$ of degree $0$ and of 
type $(u_1', u_2)$. If $\epsilon = K_X$ and $\beta = 1_X$, then we obtain
$\big \langle \mathfrak a_{\lambda}(\tau_{*}(\epsilon \beta)) w_1, w_2 \big \rangle 
= \big \langle \mathfrak a_{\lambda}(\tau_{*}K_X) w_1, w_2 \big \rangle$
which again is a universal polynomial in $\langle K_X, K_X \rangle$ of degree $0$ and of 
type $(u_1', u_2)$. 

Finally, let ${\epsilon} \in \{K_X, K_X^2\}$. We have $\tau_*(\epsilon\alpha) = \tau_*\epsilon$.
Let $I_\epsilon$ be the difference
$$
\big \langle {\bf 1}_{-(n-j-1)} \mathfrak a_{-\lambda}(\tau_*\epsilon)\vac, 
   \mathfrak a_{-1}(\beta)w_1, w_2 \big \rangle     
- \big \langle {\bf 1}_{-(n-j-2)} \mathfrak a_{-\lambda}(\tau_*\epsilon) \vac, 
   w_1, \mathfrak a_{-1}(\beta)^\dagger w_2 \big \rangle
$$
 from (\ref{Axiom22AltLeft}). When ${\epsilon} = K_X^2$, 
we see from Lemma~\ref{IntLaMuNu} that
\begin{eqnarray*}  
   I_\epsilon 
&=&\langle K_X, K_X \rangle \cdot \big \langle {\bf 1}_{-(n-j-1)} 
   \mathfrak a_{-\lambda}(x)\vac, \mathfrak a_{-1}(\beta)w_1, w_2 \big \rangle     \\
& &- \, \langle K_X, K_X \rangle \cdot \big \langle {\bf 1}_{-(n-j-2)} 
   \mathfrak a_{-\lambda}(x) \vac, w_1, \mathfrak a_{-1}(\beta)^\dagger w_2 \big \rangle \\
&=&\langle K_X, K_X \rangle \cdot P_4(K_X, S_1, S_2; 0)
\end{eqnarray*}
where $P_4(K_X, S_1, S_2; 0)$ is a universal polynomial in $\langle K_X, K_X \rangle$ of degree $0$ 
and of type $(u_1', u_2)$. When ${\epsilon} = K_X$, we see from Lemma~\ref{New1xAlBe} that
$$
I_\epsilon 
= \big \langle {\bf 1}_{-(n-j-1)} \mathfrak a_{-\lambda}(\tau_*K_X)\vac, 
   \mathfrak a_{-1}(\beta)w_1, w_2 \big \rangle      
-  \big \langle {\bf 1}_{-(n-j-2)} \mathfrak a_{-\lambda}(\tau_*K_X) \vac, 
   w_1, \mathfrak a_{-1}(\beta)^\dagger w_2 \big \rangle
$$
is a universal polynomial in $\langle K_X, K_X \rangle$ of degree at most $1$ 
and of type $(u_1', u_2)$. 
\end{proof}

\end{document}